\pdfoutput=1

\documentclass[11pt,a4paper]{article}

\usepackage{amscd,amsmath}
\usepackage{amssymb,amsthm}
\usepackage{longtable}
\usepackage{graphicx}
\usepackage[all]{xy}
\usepackage{ifpdf}
\usepackage{hyperref}

\theoremstyle{definition}
\newtheorem{Thm}{Theorem}[section]
\newtheorem{Def}{Definition}[section]
\newtheorem{Lem}[Thm]{Lemma}
\newtheorem{Cor}[Thm]{Corollary}
\newtheorem{prop}[Thm]{Proposition}

\newtheorem{Rem}{Remark}[section]

\newtheorem{Ex}{Example}[subsection]

\newtheorem*{ThmA}{Theorem A}
\newtheorem*{ThmB}{Theorem B}
\newtheorem*{ThmC}{Theorem C}
\newtheorem*{ThmD}{Theorem D}
\newtheorem*{ThmAS}{Theorem for Floer homology of a cotangent bundle}
\newtheorem*{ThmFC}{Theorem for the fiberwise convexities of the rotating Kepler problem and Hill's lunar problem}
\newtheorem*{prop 4.5.}{Proposition 4.5}
\newtheorem*{Lem 4.9.}{Lemma 4.9}

\begin{document}

%%%%%%%%%%%%%%%%%%%%%%%%%%%%%%%%%%%%%%%%%%%%%%%%%%%%%%%%%%%%%%%%%
%%%%%%%%%%%%%%%%%%%%%%% Title
%%%%%%%%%%%%%%%%%%%%%%%%%%%%%%%%%%%%%%%%%%%%%%%%%%%%%%%%%%%%%%%%%

\title{Spectrum estimates of Hill's lunar problem}
\author{\bf{Junyoung Lee} 
\\ Junyoung.Lee@math.uni-augsburg.de \\ Institut for Mathematics \\ Augsburg University}
\date{}
\maketitle

%%%%%%%%%%%%%%%%%%%%%%%%%%%%%%%%%%%%%%%%%%%%%%%%%%%%%%%%%%%%%%%%%
%%%%%%%%%%%%%%%%%%%%%%% Preliminaries
%%%%%%%%%%%%%%%%%%%%%%%%%%%%%%%%%%%%%%%%%%%%%%%%%%%%%%%%%%%%%%%%%

\begin{abstract}
We investigate the action spectrum of Hill's lunar problem by observing inclusions between the Liouville domains enclosed by the regularized energy hypersurfaces of the rotating Kepler problem and Hill's lunar problem. In this paper, we reinterpret the spectral invariant corresponding to every nonzero homology class $\alpha \in H_*(\Lambda N)$ in the loop homology as a symplectic capacity $c_N(M, \alpha)$ for a fiberwise star-shaped domain $M$ in a cotangent bundle with canonical symplectic structure $(T^*N, \omega_{can}=d \lambda_{can})$. Also, we determine the action spectrum of the regularized rotating Kepler problem. As a result, we obtain estimates of the action spectrum of Hill's lunar problem. This will show that there exists a periodic orbit of Hill's lunar problem whose action is less than $\pi$ for any energy $-c < -{3^{4 \over 3} \over 2}$.
\footnote{For a comparison, simple periodic orbits of Kepler problem have the action $2\pi\sqrt{{1 \over -2c}}$ at energy level $c<0$. The reader should be careful not to confuse the action with actual physical time.}
\end{abstract}

\section{Introduction}

Celestial mechanics has provided a huge playground for mathematicians and physicists for a long time. One of interesting object is the motion of the Moon. Before Hill, the accuracy of the lunar theory was not so good. Hill introduced a problem for the lunar theory which reflects successfully the perturbation effect of the Sun. This problem is called Hill's lunar problem. Hill's lunar problem can be derived from the (circular planar) restricted three body problem. The restricted three body problem is obtained from the three body problem by assuming one particle, say $M$(Moon), is massless and two primaries, say $S, E$(Sun, Earth), take the Keplerian circular motion on the plane. 
\footnote{In the restricted three body problem, many authors use the convention of letting the massless particle $S$(Satellite) and two primaries $E, M$(Earth, Moon). Here we use the Moon as a massless particle in the Sun-Earth system in order to emphasize the relation with Hill's lunar problem.}
With normalizations of physical constants, one can derive the Hamiltonian
\begin{gather*}
H_{R3BP}: T^*(\mathbb{R}^2-\{(-\mu, 0), (1-\mu, 0)\}) \rightarrow \mathbb{R},
\\ H_{R3BP}(q, p):={1 \over 2}|p|^2-{1-\mu \over |q-(\mu, 0)|}-{\mu \over |q-(1-\mu, 0)|}+p_1q_2-p_2q_1
\end{gather*}
for the motion of $M$ where $\mu={M_E \over M_S+M_E}$ is the mass ratio between the mass of the Earth and the total mass. In order to obtain a time-independent Hamiltonian we used the rotating reference frame. The term $p_1 q_2-p_2 q_1$ is due to this rotating coordinate. If one takes the limit $\mu \rightarrow 0$, then we get the Hamiltonian
\begin{gather*}
H_{R}: T^*(\mathbb{R}^2-(0,0)) \rightarrow \mathbb{R},
\\ H_R (q, p)={1 \over 2}|p|^2-{1 \over |q|}+p_1 q_2-p_2 q_1
\end{gather*}
of the rotating Kepler problem. As one can see, this is the Kepler problem on the rotating reference frame. The rotating Kepler problem is completely integrable like Kepler problem. In fact, $L=p_1 q_2-p_2 q_1$ is another integral. This problem provides a good starting point to approach the restrict three body problem as a limit. Because we can figure out all periodic orbit of the rotating Kepler problem from the Kepler problem, if one has a technique to obtain information by comparing or perturbing the Hamiltonians, then the rotating Kepler problem can be a good candidate as a reference problem to reach the restricted three body problem and Hill's lunar problem. At this point of view, In this paper, we will also use this problem to study Hill's lunar problem. Hill's lunar problem, in modern language, can be obtained by not only taking $\mu \rightarrow 0$ but also thinking of the blow-up coordinate of order $\mu^{1 \over 3}$ near the Earth, see \cite{MHO} for the derivation. We recall the Hamiltonian 
\begin{gather*}
H_H : T^*(\mathbb{R}^2-\{(0,0)\}) \rightarrow \mathbb{R}
\\ H_H (q,p)={1 \over 2}|p|^2-{1 \over |q|}+p_1q_2-p_2q_1-q_1^2+{1 \over 2}q_2^2.
\end{gather*}
of Hill's lunar problem. Hill's lunar problem was introduced by Hill in order to study the stability of the orbit of the Moon in \cite{Hill}. Hill assumed that the Sun is infinitely far away from the Earth and has infinite mass. This approach brought us a simple Hamiltonian with great improvement in accuracy. As one can see, the difference on Hamiltonians of the rotating Kepler problem and Hill's lunar problem is only the degree 2 term $-q_1^2+{1 \over 2}q_2^2$. However, in the dynamics, this difference gives a dramatic change. For example, Hill's lunar problem is not completely integrable while the rotating Kepler problem is completely integrable. Non-integrability of Hill's lunar problem has been proved by many authors with many versions. The analytic non-integrability of Hill's lunar problem was proved by Meletlidou, Ichtiaroglou and Winterberg in \cite{MIW}. Morales-Ruiz, Sim$\acute{o}$ and Simon gave an algebraic proof of meromorphic non-integrability in \cite{MSS}. Recently, Llibre and Roberto in \cite{LR} discussed the $C^1$ integrability based on the existence of two periodic orbits on every positive energy level. One can see the chaotic feature of Hill's lunar problem in the numerical research of  Sim$\acute{o}$ and Stuchi in \cite{SS}.

The fundamental motivation of this paper comes from Poincar$\acute{e}$. Poincar$\acute{e}$ emphasized the importance of periodic orbits in the study of dynamics. He said that 'the periodic orbits are the skeleton of dynamics of a given problem'. Indeed, periodic orbits in a Hamiltonian dynamics arise as generators of Floer theory. One can ask how much dynamics will be changed by the change of the Hamiltonian. Of course, it is a complicated problem in general. However, if we restrict our problem on the periodic orbit, there are many available tools in symplectic theory. Floer homology and symplectic homology is invariant under the change of Hamiltonians, respectively. Moreover, the action spectrum of the boundary of a Liouville domain is invariant under Liouville isomorphisms. This stability of the action spectrum was proved in \cite{CFHW} and they used this to define local Floer homology. At first glance, it seems Floer homology and symplectic homology do not give any information under the change of Hamiltonians. However, if we consider the action filtration on the homology, then the homology with action filtration can reflect the change of Hamiltonians because, in general, the change of Hamiltonians does not give a Liouville isomorphism. In addition, 

The study of the symplectic topology has been actively done. This study on the symplectic topology has brought sometimes the progress of the celestial mechanics. As a remarkable example, in \cite{AFFHvK}, they proved the existence of global surfaces of section in the restricted three body problem for some pair of mass ratio and energy $(\mu, c)$. They used the finite energy plane theory in \cite{HWZ} based on the pseudo-holomorphic curve theory. This theory is still developing actively. For example, in the upcoming book \cite{FvK}, they collect many valuable applications of holomorphic curve theory to celestial mechanics. Many valuable symplectic theories like contact homology, symplectic field theory and finite energy foliation has been originated from this theory. From Gromov's nonsqueezing Theorem, the relation between symplectic embedding and the periodic orbit has been emphasized, see the introduction of \cite{FH}. For a systematic approach of the symplectic embeddings, one can use the notion of symplectic capacity. Symplectic capacities are symplectic invariants inspired by the Gromov's work in \cite{Gromov}. Ekeland and Hofer introduced the definition of symplectic capacity for the subsets of $(\mathbb{R}^{2n}, \omega_0)$ in \cite{EH}. This can be generalized to all symplectic manifolds as follows.

\begin{Def}[Symplectic capacity]
A symplectic capacity is a map which associates a symplectic manifold $(M, \omega)$ a number $c(M, \omega) \in (0, +\infty]$ satisfying the following conditions

(1) (Conformality) $c(M, \kappa \omega)=|\kappa|c(M, \omega)$ for $\kappa \ne 0$.

(2) (Monotonicity) If there is a symplectic embedding of $(M_1, \omega_1)$ into $(M_2, \omega_2)$, then 
\\ $c(M_1, \omega_1) \le c(M_2, \omega_2)$.

(3) (Normalization) $c(B^{2n}(1))=c(Z^{2n}(1))=\pi$.

(3') (Nontriviality) $0<c(B^{2n}(1))$ and $c(Z^{2n}(1))<+\infty$.
\\ Here, $B^{2n}(r)$ is the ball in $\mathbb{R}^{2n}$ with radius $r$ and $Z^{2n}(r)$ is the cylinder $B^2(r)\times \mathbb{R}^{2n-2}$ in $\mathbb{R}^{2n}$.
\end{Def}

In this paper, we will not discuss the symplectic capacity for general symplectic manifolds as above. Instead of considering all symplectic manifolds, we focus on a particular class of symplectic manifolds, that is, Liouville domains enclosed by fiberwise star-shaped hypersurfaces in a cotangent bundle space. Let $(N, g)$ be a closed Riemannian manifold. The cotangent bundle $T^* N$ with the canonical symplectic structure $\omega_{can}$ is an open exact symplectic manifold. Let $M \subset T^* N$ be a fiberwise star-shaped domain, namely, $M \cap T^*_q N$ is a star-shaped domain with respect to the origin in $T^*_q N$ for every $q \in N$. We denote by $FSD(N)$ the set of all fiberwise star-shaped domains in $T^* N$. We introduce the following definition.

\begin{Def}[Symplectic capacity for $FSD(N)$]
Let $(N, g)$ be a closed Riemannian manifold. A symplectic capacity for $FSD(N)$ which associates a fiberwise star-shaped domain in $T^* N$ a number $c(M) \in (0, +\infty]$ satisfying the following conditions

(1) (Conformality) $c(k M)=kc(M)$ for all $k \in \mathbb{R}^+$ for $M \in FSD(N)$.

(2) (Monotonicity) $c(M_2) \ge c(M_1)$ if there is a symplectic embedding of $M_1$ into $M_2$ for $M_1, M_2 \in FSD(N)$.

(3) (Nontriviality) $0<c(D^*_g N)<+\infty$ where $D^*_g N:=\{(q, p) \in T^* N | g^*_q(p, p) \le 1\}$.
\\ Here, $kM$ is defined by fiberwise multiplication in each cotangent space.
\end{Def}

With this definition, one cannot discuss arbitrary embeddings of symplectic manifolds in general. This allows us to compare two fiberwise star-shaped domains in the same cotangent bundle. However, if we use this restricted definition, then we can easily obtain infinitely many symplectic capacities for $FSD(N)$ using the spetral invariant of symplectic homology. This is not a new idea, for example see \cite{FS}, \cite{Oh} and \cite{Schlenk}. However, we will get a convenient form for the practical application by the following reinterpretation. Since $M \in FSD(N)$ is a Liouville domain, we can define a symplectic homology of $M$. We have the long exact sequence
\begin{gather*}
\cdots \rightarrow SH_* ^{<b}(M) \xrightarrow{i^b_M} SH_* (M) \xrightarrow{j^b_M} SH_* ^{\ge b} \rightarrow SH_{*-1}^{<b}(M) \xrightarrow{i^b_M} \cdots
\end{gather*}
for the symplectic homology of $M$ for an action filtration. Moreover, we have the isomorphism
\begin{gather*}
\Psi_M : H_* (\Lambda N) \rightarrow SH_* (M) 
\end{gather*}
between the homology of the loop space of $N$ and the symplectic homology of $M$, see \cite{AS}, \cite{SW} and \cite{V2}. With these ingredients, we can define a map
\begin{gather*}
c_N : FSD(N) \times H_*(\Lambda N)^{\times} \rightarrow \mathbb{R},
\\ c_N (M, \alpha):=\inf \{b \in \mathbb{R} \cup \{+ \infty\} | \Psi_M(\alpha) \in \textrm{im}(i_M^b) \}
\end{gather*}
assigning a nonnegative number to the pair of a fiberwise star-shaped domain and a homology class of the free loop space of $N$. We will prove the properties of the map $c_N$.

\begin{ThmA}[Properties of $c_N$]
The map
\begin{gather*}
c_N : FSD(N) \times H_* (\Lambda N)^{\times} \rightarrow \mathbb{R}
\\ \quad \quad \quad \quad (M, \alpha) \mapsto c(M, \alpha)
\end{gather*}
satisfies the following properties.

(1) (Conformality) $c_N (kM, \alpha)=k c_N (M, \alpha)$ for all $k \in \mathbb{R}^+$.

(2) (Monotonicity) $c_N (M_2, \alpha) \ge \kappa_{min}(\Sigma_1, \Sigma_2) c_N (M_1, \alpha)$ for all $M_1, M_2 \in FSD(N)$ where $\Sigma_i=\partial M_i$, $i=1, 2$ and $\kappa_{min}(\Sigma_1, \Sigma_2)=\min_{x \in \Sigma_1} \{\kappa(x)| \kappa(x)x \in \Sigma_2, \kappa(x)>0 \}$.

(3) (Spectrality) $c_N (M, \alpha) \in Spec(\Sigma, \lambda_{can})$ where $\Sigma= \partial M$.
\\ for each $\alpha \in H_* (\Lambda N)^{\times}$.
\end{ThmA}

Whenever we choose a homology class $\alpha \in H_* (\Lambda N)^{\times}$, $c_N(\cdot, \alpha)$ gives a map from $FSD(N)$ to $[0, +\infty]$. By Theorem A, the map $c_N(\dot, \alpha)$ satisfies Conformality of symplectic capacity for $FSD(N)$. Also, with symplectic invariance of symplectic homology, the map $c_N(\cdot, \alpha)$ satisfies also Monotonicity of symplectic capacity for $FSD(N)$. Finally, Spectrality of Theorem A can replace Nontriviality of symplectic capacity provided $c_N(\cdot, \alpha) \ne 0 \iff c_N(D^*_g N, \alpha) \ne 0$.
 
We will apply this symplectic capacity for Liouvillie domains in a cotangent bundle to the rotating Kepler problem and Hill's lunar problem. For this application, of course, one has to be able to find Liouville domains related to these problems. The following Theorem makes this possible.

\begin{ThmFC}[\cite{CFvK} for the rotating Kepler problem, \cite{L} for Hill's lunar problem]
Below the critical energy levels, the energy hypersurfaces of the rotating Kepler problem and Hill's lunar problem can be symplectically embedded into the cotangent bundle of $S^2$ as fiberwise convex hypersurfaces, respectively.
\end{ThmFC}

We will explain this in Section 2.2. We call these fiberwise convex hypersurfaces by the regularized energy hypersurfaces of the rotating Kepler problem and Hill's lunar problem. We denote by $\Sigma_R^c$ and $\Sigma_H^{c'}$ the regularized energy hypersurface of the rotating Kepler problem of energy $-c$ and Hill's lunar problem at energy $-c'$, respectively. Since they are fiberwise convex, they bound Liouville domains, denoted by $M_R^c$ and $M_H^{c'}$. We define increasing sequences
\begin{gather*}
c_R^P:={P+3 \over 2(P+1)^{1 \over 3}},
\\ c_H^P:={2P+8-\sqrt{(P+1)(P+9)} \over 2(P+1)^{1 \over 3}}
\end{gather*}
for $P=1, 2, 3, \cdots$. We define by $-c_R^0=-{3 \over 2}$ and $-c_H^0=-{3^{4 \over 3} \over 2}$ the critical values of the rotating Kepler problem and Hill's lunar problem. We will prove the following Theorem.

\begin{ThmB}
For the fiberwise convex domains $M_R^c$ and $M_H^{c'}$ in $T^* S^2$ defined by the regularized energy hypersurfaces of the rotating Kepler problem and Hill's lunar problem, we have the following inclusions in $T^* S^2$.
\\ (1) $M_H^{c} \subset M_R^{c_R^1}$ for all $c \ge c_H^0$.
\\ (2) $M_H^{c} \subset M_R^{c_R^P}$ if $c \ge c_H^P$ for all $P=2, 3, 4, \cdots$.
\\ (3) $M_R^{c+{1 \over 2c^2}} \subset M_H^c$ for all $c>c_H^0$.
\end{ThmB}

Since fiberwise convexity implies fiberwise star-shapedness, we can apply the symplectic capacity $c_{S^2}$ in Theorem A to  the Liouville domains $M_R^c$ and $M_H^{c'}$. On the other hand, we can compute the symplectic capacity $c_{S^2}(M_R^c, \alpha)$ when $\alpha$ can be uniquely expressed by the retrograde and direct orbits(possibly multiply covered) in the symplectic homology of $M_R^c$. We will denote by $\delta_{R, N}, \delta_{D, N}$ the homology classes determined by the $N$th-iterations of retrograde and direct orbits in the symplectic homology of $M_R^c$ if they are cycles. When $\delta_{R, N}$ and $\delta_{D, N}$ are defined, we define the homology classes $\Delta_{R, N}=\Psi_{M_R^c}^{-1}(\delta_{R, N})$ and $\Delta_{D, N}=\Psi_{M_R^c}^{-1}(\delta_{D, N})$ in $H_*(\Lambda S^2)$ by the isomorphism $\Psi_{M_R^c}: H_* (\Lambda S^2) \xrightarrow{\approx} SH_*(M_R^c)$. More precisely, we will prove the following Theorem.

\begin{ThmC}
For the energy $c \in [c_R^P, c_R^{P+1})$, the homology classes $\Delta_{R, N}$ and $\Delta_{D, N}$ are well-defined for $N=1, 2, \cdots, P$. Moreover, we have the symplectic capacity
\begin{gather*}
c_{S^2}(M_R^c, \Delta_{R, N})=2\pi N L_R(c)=\pi N \sqrt{{3 \over 2c}} \sec \left({1 \over 3} \arccos \left( \left({3 \over 2c} \right)^{3 \over 2}\right)\right),
\\ c_{S^2}(M_R^c, \Delta_{D, N})=-2\pi N L_D(c)=-\pi N \sqrt{{3 \over 2c}} \sec \left({1 \over 3} \arccos \left( \left({3 \over 2c} \right)^{3 \over 2}\right)+{2 \pi \over 3}\right)
\end{gather*}
of $M_R^c$ in $T^* S^2$ with respect to $\Delta_{R, N}$ and $\Delta_{D, N}$ for $N=1, 2, \cdots, P$. Here, $2 \pi L_R (c)$ and $- 2 \pi L_D (c)$ are the actions of the retrograde and direct orbits, respectively.
\end{ThmC}

Theorem C can be proved by computing the action values and index of all periodic orbits of the rotating Kepler problem. In fact, The Conley-Zehnder indices of the rotating Kepler problem were already computed in \cite{AFFvK} and thus we will use this. If we compute the action value in this paper, as a result of this computation, we can prove that the retrograde orbit is always the systole of the regularized energy hypersurface $(\Sigma_R^c, \lambda_{can})$. We leave the computation of systolic volume of the contact manifold $(\Sigma_R^c, \lambda_{can})$ in the Appendix.

Using Theorem A, B and C, we can obtain estimates of the symplectic capacity of $M_H^c$ in $T^* S^2$. We introduce the final goal of this paper.

\begin{ThmD}
For the fiberwise convex domain $M_H^c$, we have the following estimates for the symplectic capacity for $FSD(S^2)$.
\\ (1) The inequalities
\begin{gather*}
2 \pi {-1 +\sqrt{1+8c^3} \over 4c^2} \le c_{S^2}(M_H^c, \Delta_{R, 1})<2^{-{11 \over 6}} \cdot 3^{1 \over 2} \pi \sec \left( {1 \over 3} \arccos(2^{-{5 \over 2}}\cdot 3^{3 \over 2})\right) \approx 2\pi \times 0.490534,
\\ 2 \pi {1 +\sqrt{1+8c^3} \over 4c^2} \le c_{S^2}(M_H^c, \Delta_{D, 1})<-2^{-{11 \over 6}} \cdot 3^{1 \over 2} \pi \sec \left( {1 \over 3} \arccos(2^{-{5 \over 2}}\cdot 3^{3 \over 2})+{2 \pi \over 3}\right) \approx 2\pi \times 0.793701
\end{gather*}
hold for all $c>c_H^0$.
\\ (2) If $c \in [c_H^P, c_H^{P+1})$ for some $P \in \{2, 3, 4, \cdots \}$, then the inequalities
\begin{gather*}
2 \pi N{-1+\sqrt{1+8c^3} \over 4c^2} \le c_{S^2}(M_H^c, \Delta_{R, N}) \le 2\pi N {-(P+1)+\sqrt{(P+1)(P+9)} \over 4(P+1)^{1 \over 3}},
\\ 2 \pi N{1+\sqrt{1+8c^3} \over 4c^2} \le c_{S^2}(M_H^c, \Delta_{D, N}) \le 2\pi N (P+1)^{-{1 \over 3}}
\end{gather*}
hold for all $N=1, 2, \cdots, P$.
\end{ThmD}

The lower estimates are the result of (3) of Theorem B. The upper estimates in (1) and (2) are the result of (1) and (2) of Theorem B, respectively. From (1) of Theorem D, one can say that  there exists at least one periodic orbit whose action is less than $\pi$ in the regularized Hill's lunar problem for any energy level below the critical value $-c_H^0=-{3^{4 \over 3} \over 2}$. Moreover, the periodic orbit whose action is $c_{S^2}(M_H^c, \Delta_{R, 1})$ has Conley-Zehnder index 1 and the periodic orbit whose action is $c_{S^2}(M_H^c, \Delta_{D, 1})$ has Conley-Zehnder index 3. From (2), we can say basically same, it is better to visualize the result, see Figure 1 and 2.
\\ \\
$\bold{Acknowledgements : }$ I thank Urs Frauenfelder for valuable discussions. I also thank colleagues in Augsburg university for many helps and encouragements. This research is supported by DFG-CI 45/6-1: Algebraic Structures on Symplectic Homology and Their Applications.

\begin{figure}
\centering
\includegraphics[]{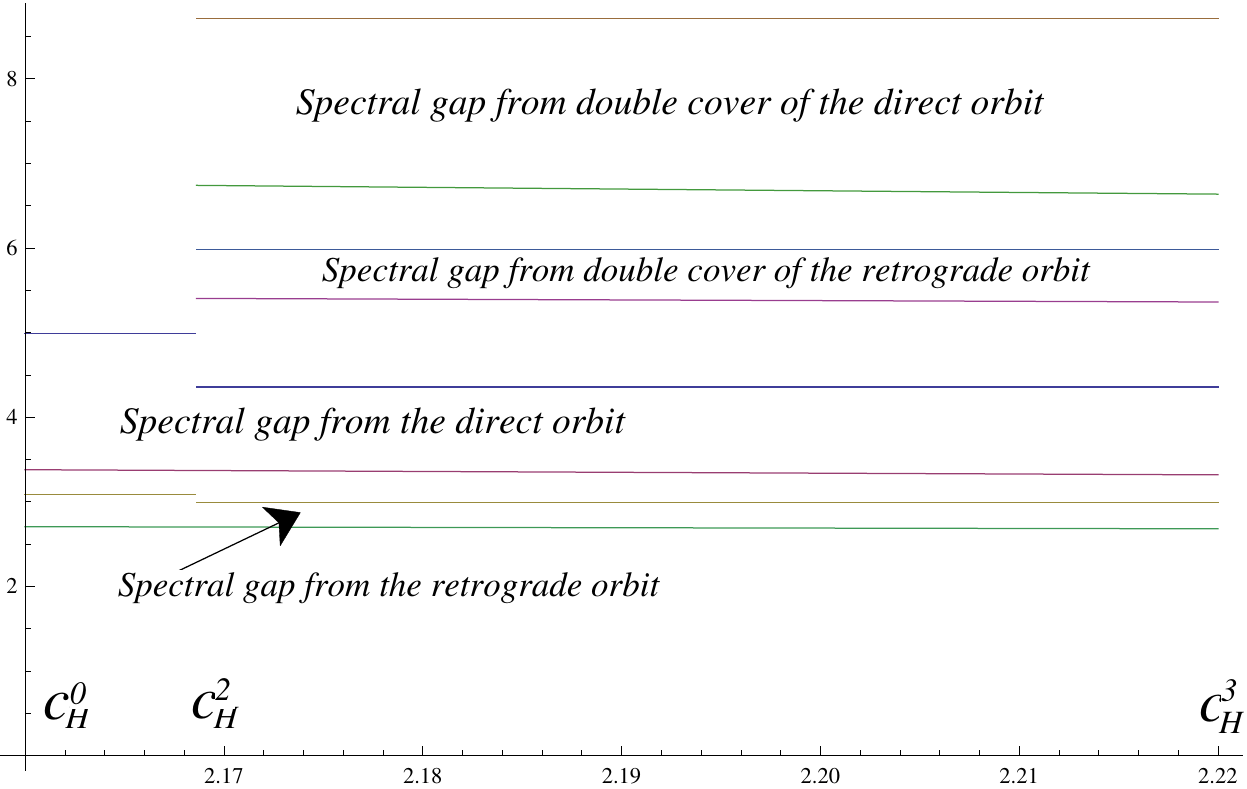}
\caption{Estimates for the action of Hill's lunar problem on $c \in (c_H^0, c_H^3)$}
\end{figure}

\begin{figure}
\centering
\includegraphics[]{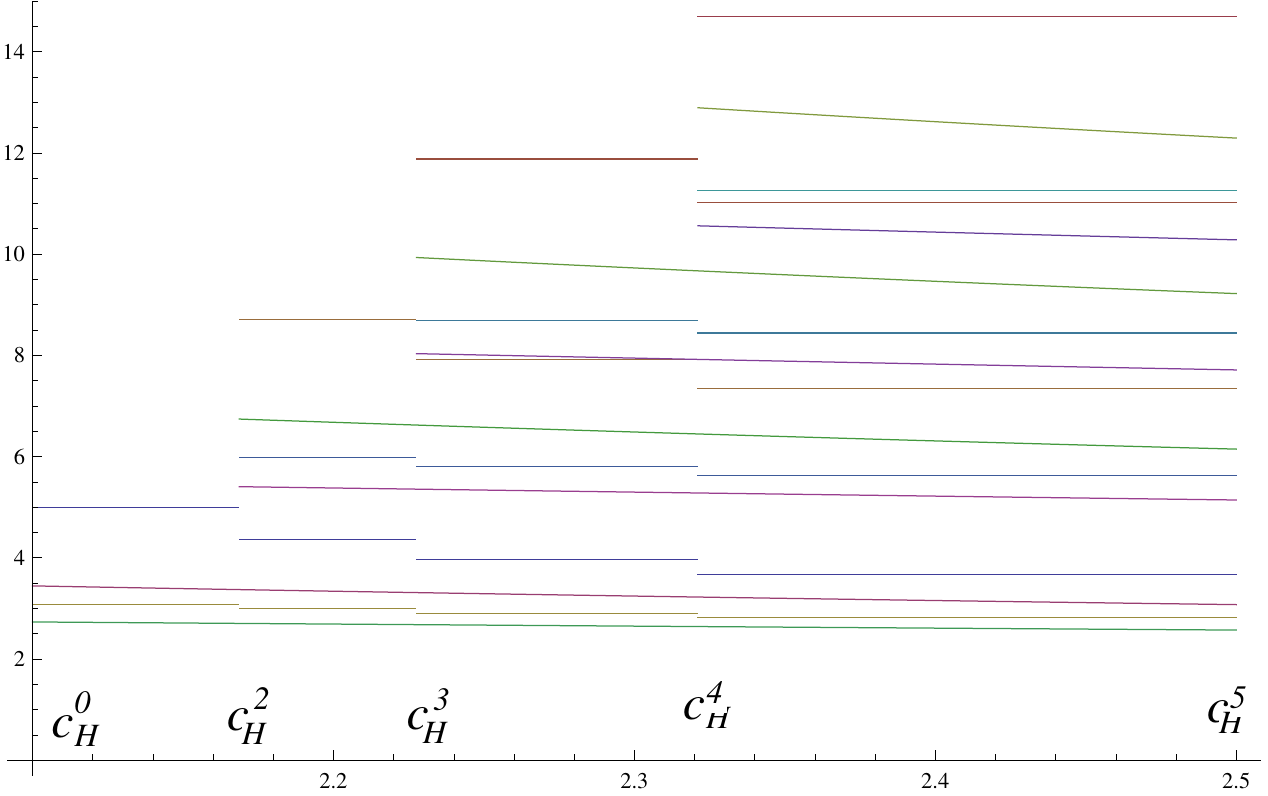}
\caption{Estimates for the action of Hill's lunar problem on $c \in (c_H^0, c_H^5)$. Note that they can be overlapped from third cover.}
\end{figure}

\section{Fiberwise star-shaped hypersurfaces in a cotangent bundle}

\subsection{Geodesic and Hamiltonian equation}
Let $(N, g)$ be a Riemannian $n$-manifold. In the dynamical aspect, one of the most interesting objects is the geodesic. The geodesic equation of $(N, g)$ is a second order differential equation
\begin{gather*}
x : (-\epsilon, \epsilon) \rightarrow U,
\\ {d^2 x^i \over dt^2}+\Gamma^i_{jk} {dx^j \over dt}{dx^k \over dt}=0, i=1,2, \cdots, n
\end{gather*}
for a curve on $N$ in a local coordinate $U$ of $N$ where $\Gamma^i_{jk}={1 \over 2}g^{im}(g_{mj,k}+g_{mk,j}-g_{ij,m})$ are Christoffel symbols. Here we follow the Einstein summation convention. This differential equation is derived from the Euler-Lagrange equations of motion of the energy functional
\begin{gather*}
E : \Omega_{x_0, x_1}(N) \rightarrow \mathbb{R},
\\ E(\gamma)= \int_{0}^{1}{{1 \over 2} g_{\gamma(t)}(\dot{\gamma}(t), \dot{\gamma}(t))} dt
\end{gather*}
on the smooth path space $\Omega_{x_0, x_1}(N)=\{ \gamma \in C^{\infty}([0,1], N) | \gamma(0)=x_0, \gamma(1)=x_1 \}$ connecting two points $x_0, x_1 \in N$. In other words, a geodesic connecting two points $x_0, x_1$ is the stationary point of the energy functional $E$ defined above. One can regard the integrand ${1 \over 2} g_{\gamma(t)}(\dot{\gamma}(t), \dot{\gamma}(t))$ as a Lagrangian $L$ defined on the tangent bundle $TN$ by
\begin{gather*}
L : TN \rightarrow \mathbb{R},
\\ L(q, v)={1 \over 2}g_q(v, v)
\end{gather*}
where $q \in N, v \in T_q N$. Using the Legendre transformation, one can derive the Hamiltonian equation on the cotangent bundle $T^* N$ corresponding to the above Euler-Lagrangian equation and consequently to the geodesic equation. Before we introduce this Hamiltonian equation, we explain the canonical symplectic structure on the cotangent bundle and the Hamiltonian flow associated with a Hamiltonian. Every cotangent bundle has the canonical symplectic structure. The canonical symplectic form $\omega_{can}$ is defined by $\omega_{can}=d \lambda_{can}$ where $\lambda_{can}$ is the Liouville 1-form on $T^* N$.

\begin{Def}
Let $\pi : T^* N \rightarrow N$ be the canonical projection. The Liouville 1-form(or canonical 1-form) $\lambda_{can}$ is defined by
\begin{gather*}
\lambda_{can}(v)=p(\pi_*(v))
\end{gather*}
for $v \in T_x T^* N$ where $x=(q, p) \in T^* N$ with $q=\pi(x) \in M$ and $p \in T^*_q N$.
\end{Def}

In canonical coordinates $(q,p)$, that is, $q$-variables are coordinates on $N$ and $p$-variables are the conjugated momentum, we can express these forms
\begin{eqnarray*}
\lambda_{can}=pdq, \quad \omega_{can}=dp \wedge dq
\end{eqnarray*}
in terms of $q, p$. It is independent of the choice of canonical coordinates. One can easily see that $\omega_{can}$ is closed and nondegenerate. This fact leads us naturally to the definition of symplectic manifolds.

\begin{Def}
A smooth manifold $M$ equipped with a 2-form $\omega$ is called a symplectic manifold if $\omega$ is closed and nondegenerate.
\end{Def}

We will discuss about the Hamiltonian equation of a symplectic manifold for a while. Suppose that $(M, \omega)$ is a symplectic manifold. A Hamiltonian is a function on $\mathbb{R} \times M$. Let $H$ be a Hamiltonian on $M$, namely we have a function
\begin{eqnarray*}
H : \mathbb{R} \times M \rightarrow \mathbb{R}
\end{eqnarray*}
on $\mathbb{R} \times M$. We will write
\begin{gather*}
H_t:=H(t, \cdot): M \rightarrow \mathbb{R}
\end{gather*} 
for notational convenience. We can define the Hamiltonian vector field $X_H^t$ associated to $H$ by
\begin{eqnarray*}
\iota_{X_H^t} \omega=-dH_t
\end{eqnarray*}
and this is uniquely defined by nondegeneracy of $\omega$. 

The Hamiltonian flow $\phi_H^t$ is the flow of the Hamiltonian vector field and so defined by the differential equation
\begin{eqnarray*}
{d \over dt}\phi_H^t(x)=X_H^t(\phi_H^t (x))
\end{eqnarray*}
and $\phi_H^t(x_0)$ is given by solving the initial value problem
\begin{eqnarray*}
\dot{x}(t)=X_H^t(x(t)), \quad x(0)=x_0 \in M
\end{eqnarray*}
We call the above equation Hamiltonian equation. We will call this diffeomorphism $\phi_H^t$ a Hamiltonian diffeomorphism generated by $H$ at time $t$ for each fixed $t$. Hamiltonian diffeomorphisms satisfy some usefule properties. Let us check the following basic properties.

\begin{Thm}
The Hamiltonian diffeomorphism $\phi_H^t$ is a symplectomorphism for each $t$. Namely, $(\phi_H^t)^* \omega=\omega$ holds for each $t \in \mathbb{R}$.
\end{Thm}
\begin{proof}
Since $\phi_H^0$ is the identity, it is enough to see that $(\phi_H^t)^* \omega$ is time independent. Hence we will show that
\begin{gather*}
{d \over dt}(\phi_H^t)^* \omega=0
\end{gather*}
for any $t$. Using Cartan's formula, we have that
\begin{gather*}
{d \over dt}(\phi_H^t)^* \omega=(\phi_H^t)^* L_{X_H^t} \omega=(\phi_H^t)^*(d \iota_{X_H^t} \omega+\iota_{X_H^t}d \omega)
\end{gather*}
and we also know that
\begin{gather*}
\iota_{X_H^t}d \omega=0,
\\ d \iota_{X_H^t} \omega=d(-dH_t)=0
\end{gather*}
from the closedness of $\omega$ and the definition of Hamiltonian vector field. This implies
\begin{gather*}
{d \over dt}(\phi_H^t)^* \omega=0
\end{gather*}
and therefore this proves Theorem 2.1.
\end{proof}

When the Hamiltonian $H$ is time-independent, the energy is conserved. Since we mostly deal with time-independent Hamiltonians in this paper, the following Theorem is important.
\begin{Thm}
If $H$ is time-independent, that is, $H$ is a function on $M$, then the Hamiltonian flow $\phi_H^t$ preserves energy, that is, $H(\phi_H^t(x))=H(x)$ for all $t \in \mathbb{R}$ and $x \in M$.
\end{Thm}
\begin{proof}
The proof can be done by the following equation
\begin{eqnarray*}
{d \over dt} H(\phi_H^t(x))&=&dH(\phi_H^t(x))[{d \over dt}\phi_H^t(x)]=dH(\phi_H^t(x))[X_H(\phi_H^t(x))]
\\ &=&-\omega(\phi_H^t(x))(X_H(\phi_H^t(x)),X_H(\phi_H^t(x)))=0
\end{eqnarray*}
for any $t \in \mathbb{R}$ and $x \in M$. This proves Theorem 2.2.
\end{proof}

Theorem 2.2 tells us that if $H$ is a time-independent Hamiltonian and $x \in H^{-1}(c)$, then $\phi_H^t(x) \in H^{-1}(c)$ for every $t$. In other words, the Hamiltonian vector field $X_H$ is tangential to the energy hypersurface of $H$. Suppose $\Sigma$ is a hypersurface, codimension 1 submanifold, of a symplectic manifold $(M, \omega)$. Then $\Sigma$ induces a canonical line bundle
\begin{gather*}
L_{\Sigma} \rightarrow \Sigma
\end{gather*}
as a subbundle of the tangent bundle $T\Sigma$ by defining the fiber
\begin{gather*}
L_{\Sigma, x}=\{v \in T_x M | \omega(v, w)=0 \textrm{ for all } w \in T_x \Sigma \}
\end{gather*}
as the symplectic complement of $T_x \Sigma$ for each $x \in \Sigma$. Since $\dim T_x \Sigma +\dim T_x \Sigma^{\omega}=\dim T_x M$ and every hyperplane is a coisotropic subspace in the symplectic space, $L_{\Sigma}$ is a line subbundle of $T\Sigma$.

\begin{Lem}
Let $H: M \rightarrow \mathbb{R}$ be a time-independent Hamiltonian on a symplectic manifold $(M, \omega)$. Then the Hamiltonian vector field $X_H$ on a energy hypersurface $H^{-1}(c)$ defines a section of the canonical line bundle $L_{H^{-1}(c)} \rightarrow H^{-1}(c)$.
\end{Lem}
\begin{proof}
We have to show that
\begin{gather*}
X_H(x) \in L_{H^{-1}(c), x}
\end{gather*}
for each $x \in H^{-1}(c)$. By definition, it is enough to see that
\begin{gather*}
\omega(X_H(x), w)=0
\end{gather*}
for all $w \in T_x H^{-1}(c)$. In fact, for any $w \in T_x H^{-1}(c)$ we have
\begin{gather*}
\omega(X_H(x), w)=-dH(x)[w]=0.
\end{gather*}
This proves Lemma 2.3.
\end{proof}

Lemma 2.3 implies that if two Hamiltonians have same regular energy hypersurface than the Hamiltonian flows are same on that energy hypersurface up to reparametrization. For example, if we composite a monotone increasing invertible function $f: \mathbb{R} \rightarrow \mathbb{R}$ to a given Hamiltonian $H$, then $X_{f\circ H}$ is parallel to $X_H$ and so has the same Hamiltonian flow up to reparametrization.

We return to the geodesic problem on a Riemannian manifold $(N, g)$. We want to find the Hamiltonian flow corresponding to the geodesic flow on $(N, g)$. Thus we will derive the Hamiltonian function $H: T^* N \rightarrow \mathbb{R}$ corresponding to the above Lagrangian $L(q, v)={1 \over 2} g_q (v,v)$. We apply the Legendre transformation
\begin{eqnarray*}
H(q,p)=\sup_{v \in {T_q N}}{(\left<p,v\right>-L(q,v))}
\end{eqnarray*}
to $L$ in order to obtain $H$ at $(q, p) \in T^* N$. One can easily see that the supremum is attained at $v(p) \in T_q N$ such that $p=d_v L(q, v(p))=\iota_{v(p)} g_q$. There exists a unique $v(p) \in T_q N$ such that $p=\iota_{v(p)}g_q$ for each $p \in T^*_q N$ by the nondegeneracy of a Riemannian metric. Then we have the Hamiltonian function
\begin{eqnarray*}
H(q, p)=\left<p,v(p) \right> -L(q, v(p))=g_q (v(p), v(p))-{1 \over 2} g_q (v(p), v(p))={1 \over 2}g^*_q(p, p)
\end{eqnarray*}
on $T^* N$ where $g^*$ is a metric on $T^* N$ which is dual to $g$. Intuitively, one can think of a geodesic as a free motion of a particle and so the Hamiltonian corresponding to the geodesic equation has only the kinetic energy term. We summarize the above discussion in the following Theorem.
\begin{Thm}
Let $(N, g)$ be a Riemannian manifold. Then the Hamiltonian flow of the Hamiltonian
\begin{eqnarray*}
H(q, p)={1 \over 2}g_q^*(p,p)
\end{eqnarray*}
on the symplectic manifold $(T^* N, d \lambda_{can})$ is a lift of the geodesic flow of $(N, g)$. Namely, the projection  $\pi(\phi_H^t ((q,p)))$ of a Hamiltonian flow into the base manifold $N$ is the geodesic flow starting at $q$ with tangent vector $v(p) \in T_q N$.
\end{Thm}
In the upshot the geodesic problem on any Riemannian manifold $(N, g)$ can be interpreted as a Hamiltonian dynamics problem on its cotangent bundle $T^* N$. Let us see the following familiar example.

\begin{Ex}
We consider the 2-sphere $(S^2, g_{round})$ with the round metric. Then the Hamiltonian $H_S$ for the geodesic flow on $(S^2, g_{round})$ is given by $H_S(q, p)={1 \over 2}g_{round}^*(p, p)$ for each $(q, p) \in T^* S^2$. We want to see this Hamiltonian in a local coordinate chart. In particular, if we think of the stereographic projection 
\begin{eqnarray*}
\phi: \mathbb{R}^2 \rightarrow S^2-\{ N \}
\end{eqnarray*}
from the north pole $N=(0, 0, 1)$ given by
\begin{gather*}
\phi(x_1, x_2)=({2x_1 \over x_1^2+x_2^2+1}, {2x_2 \over x_1^2+x_2^2+1}, {x_1^2+x_2^2-1 \over x_1^2+x_2^2+1}),
\\ \phi^{-1}(q_1,q_2,q_3)=({q_1 \over 1-q_3}, {q_2 \over 1-q_3}).
\end{gather*}
This induces the canonical local coordinate chart $\Phi : T^* \mathbb{R}^2 \rightarrow T^* (S^2-\{ N \})$ of the cotangent bundle. The Hamiltonian $H$ on this coordinate chart have the expression
\begin{eqnarray*}
\tilde{H}(x, y):=H_S \circ \Phi(x, y)={1 \over 8}(|x|^2+1)^2 |y|^2
\end{eqnarray*}
in terms of $(x, y) \in T^* \mathbb{R}^2$. This Hamiltonian system defined on $T^* \mathbb{R}^2=\mathbb{R}^2 \times \mathbb{R}^2$ is equivalent to the geodsic problem on $(S^2-\{N\}, g_{round})$.
\end{Ex}

Of course, not every Hamiltonian system is a geodesic problem. It is hard to imagine that an arbitrary Hamiltonian mechanics problem can be interpreted as a geodesic problem of some Riemannian manifold. However, Moser in \cite{Moser} found a beautiful connection between the Kepler problem and the geodesic problem on $(S^2, g_{round})$. The Kepler problem is important in celestial mechanics as the most fundamental problem. As one knows, the geodesic problem on the standard unit sphere equipped with the round metric is also one of the most fundamental problems in the geodesic problem. We will discuss this relationship and its generalization in the next section.

\subsection{Moser regularization and generalization}

We will discuss only the planar Kepler problem and will give a sketch of the proof of Theorem in \cite{Moser}. The generalization of the Kepler problem on $\mathbb{R}^n$ is not difficult. One can find a precise proof of the $n$-dimensional problem in \cite{Moser}. 

The planar Kepler problem is a one body problem under the gravitational force toward the origin on the plane. One can see that any two body problem can be decoupled into two one body problems by fixing their center of mass on the origin. The Hamiltonian of the planar Kepler problem is
\begin{gather*}
H_{KP}: T^*(\mathbb{R}^2-\{(0,0)\} ) \rightarrow \mathbb{R},
\\ H_{KP}(q, p)={1 \over 2}|p|^2-{1 \over |q|}
\end{gather*}
with some normalizations of physical constants. We consider the Hamiltonian
\begin{eqnarray*}
K_{KP}^c(q, p)=|q|(H_{KP}(q,p)+c)={1 \over 2}(|p|^2+2c)|q|-1
\end{eqnarray*}
in order to remove the singularity. Note that the energy hypersurfaces are same
\begin{eqnarray*}
(K_{KP}^c)^{-1}(0)=H_{KP}^{-1}(-c)
\end{eqnarray*}
and so their Hamiltonian vector fields $X_{K_{KP}}$ and $X_{H_{KP}}$ are parallel on this common energy hypersurfaces $(K_{KP}^c)^{-1}(0)$. We focus on the case of $c={1 \over 2}$. Other negative energy levels can be achieved simply by rescaling the variables. We consider the following symplectic transformation
\begin{eqnarray*}
&\Psi: (T^* \mathbb{R}^2=\mathbb{R}^2 \times \mathbb{R}^2, dx \wedge dy) \rightarrow (T^* \mathbb{R}^2=\mathbb{R}^2 \times \mathbb{R}^2, dq \wedge dp),&
\\ &\Psi(x, y)=(y, -x)&
\end{eqnarray*}
namely $p=-x, q=y$. Then we define the Hamiltonian 
\begin{eqnarray*}
\tilde{K}(x, y):=K_{KP}^{1 \over 2} \circ \Psi={1 \over 2}(|x|^2+1)|y|-1
\end{eqnarray*}
by applying the above symplectic transformation. We remark that this symplectic transformation plays the role of changing the position and momentum variables in our case. We recall the Hamiltonian $\tilde{H}(x,y)={1 \over 8}(|x|^2+1)^2|y|^2$ on $T^* \mathbb{R}^2$ in Example 2.1.1. Then the energy hypersurfaces $\tilde{K}^{-1}(0)$ and $\tilde{H}^{-1}({1 \over 2})$ are same and they have same Hamiltonian flows up to reparametrization. We know these energy hypersurfaces $\tilde{K}^{-1}(0)$ and $\tilde{H}^{-1}({1 \over 2})$ come from $H_{KP}^{-1}(-{1 \over 2})$ and $H_S^{-1}({1 \over 2})$, respectively, where $H_S(q, p)={1 \over 2}g_{round}^*(p, p)$ for each $(q, p) \in T^* S^2$. We summarize the Moser's result.

\begin{Thm}[Moser]
For a negative energy $c<0$, the energy hypersurface $H_{KP}^{-1}(c)$ can be symplectically embedded into the cotangent bundle $T^* S^2$ as the unit cotangent bundle of $S^2- \{N\}$. Moreover, we can compactify these energy hypersufaces into the unit cotangent bundle of $S^2$ by adding the collision orbits.
\end{Thm}

We summarize what we have done above. The procedure can be simplified by the following composition of maps
\begin{eqnarray*}
H_{KP}^{-1}(-c) \subset (T^* \mathbb{R}^2, \omega_{std}) \xrightarrow[Symp.]{\Psi} (T^* \mathbb{R}^2, \omega_{std}) \xrightarrow[Stereo.]{\Phi} (T^* S^2, \omega)
\end{eqnarray*}
for $c>0$ and the closure of image $\overline{\Phi \circ \Psi(H_{KP}^{-1}(-{1 \over 2}))}$ under the maps was amazingly the unit cotangent bundle $S_1^*S^2$ of $(S^2, g_{round})$. In general, if we choose another energy level, then we have a hypersurface
\begin{gather*}
\overline{\Phi \circ \Psi  (H_{KP}^{-1}(-c))}=\Sigma_K^c \subset T^*S^2
\end{gather*}
of the cotangent bundle over $S^2$. This hypersurface can be interpreted as a unit cotangent bundle of $S^2$ with respect to a Riemannian metric $g_c$. We have many possibilities for a generalization. For example, one can replace $\Psi$ and $\Phi$ by other symplectomorphisms. In this case, we let $T_c: T^* \mathbb{R} \rightarrow T^* \mathbb{R}$ be a linear symplectic map
\begin{gather*}
T_c(q, p)=({q \over \sqrt{2c}}, \sqrt{2c}p).
\end{gather*}
If we replace $\Psi$ by $\Psi \circ T_c$, then we have that
\begin{gather*}
\overline{\Phi \circ \Psi \circ T_c (H_{KP}^{-1}(-c))}=S_{\sqrt{2c}}^* S^2:=\{(q,p) \in T^* S^2 | \sqrt{g_{round}^*(p,p)}=\sqrt{2c}\}
\end{gather*}
for each $c>0$. In fact, this completes the above Theorem.

In this paper, we will discuss another kind of generalization by considering the metric on $S^2$. For this generalization, we recall the definition of a Finsler metric on a smooth manifold.

\begin{Def}
A Finsler manifold is a differentiable manifold $N$ equipped with a Finsler function $F$ on the tangent bundle $TN$. Namely, $F$ satisfies the following conditions.
\\ $\cdot$ $F$ is smooth on $TN \backslash N$. Here, $N$ means the zero section.
\\ $\cdot$ $F((q, v)) \ge 0$ for all $(q, v) \in TN$ and $F((q, v))=0$ if and only if $v=0$.
\\ $\cdot$ $F((q, \lambda v))=\lambda F((q, v))$ for all $\lambda \ge 0$ and $(q, v) \in TN$.
\\ $\cdot$ $F((q, v+w)) \le F((q, v))+F((q, w))$ for all $(q,v), (q,w) \in TN$
\\ We call $F$ a Finsler metric on $N$.
\end{Def}

In general, a Finsler metric $F$ is not an inner product on each tangent space but defines a norm on each tangent space. Let us define the corresponding geometric object.

\begin{Def}
Let $N$ be a differentiable manifold. A hypersurface $\Sigma$, codimension 1 submanifold, of the tangent bundle $TN$ is called fiberwise convex if $\Sigma \cap T_q N$ bounds a strictly convex bounded domain of $T_q N$ which contains the origin for each $q \in N$. 
\end{Def}

One can immediately see that there is a one-to-one correspondence between the set of all Finsler metrics and the set of all fiberwise convex hypersurfaces for any fixed manifold $N$.
\begin{gather*}
\{ \textrm{Finsler metric on } N\} \longleftrightarrow \{ \textrm{Fiberwise convex hypersurface of } TN\},
\\ F \longmapsto F^{-1}(1). \quad \quad \quad \quad \quad
\end{gather*}

\begin{Rem}
We can rewrite the above two definitions for the cotangent bundle $T^*N$ by the exactly same way. Moreover, we also have the one-to-one correspondence between the set of dual Finsler metric on $N$ and the set of fiberwise convex hypersurfaces of $T^* N$.
\end{Rem}

We can extend the idea of the Moser regularization. For a given Hamiltonian
\begin{gather*}
H:T^* \mathbb{R}^2 \rightarrow \mathbb{R}
\end{gather*}
on $T^* \mathbb{R}^2 = \mathbb{R}^2 \times \mathbb{R}^2$. For the embedding
\begin{gather*}
\Sigma^c:=\Phi \circ \Psi (H^{-1}(-c)) \subset T^* S^2
\end{gather*}
of an energy hypersurface under the above maps, if its closure $\overline{\Sigma^c}$ in $T^* S^2$ is a fiberwise convex hypersurface of $T^* S^2$, then the Hamiltonian flow on $H^{-1}(-c)$ can be interpreted as a geodesic flow of the corresponding Finsler metric on $S^2$. In this case, we will say that the Hamiltonian dynamics defined by $H$ is fiberwise convex for energy $-c$. This generalization has been applied particularly to celestial mechanics problems related to Kepler problem. In \cite{CFvK}, they prove fiberwise convexity of the rotating Kepler problem. Fiberwise convextity of Hill's lunar problem was also proved in \cite{L}. One goal in this generalization is determining fiberwise convexity of the restricted three body problem. We still do not know fiberwise convexity of the restricted three body problem. Now, we want to introduce precisely the main ingredients of this paper as examples of this generalization.

\begin{Thm}[Fiberwise convexity of the rotating Kepler problem, \cite{CFvK}]
The bounded component of the regularized rotating Kepler problem is fiberwise convex for all energy below the critical level.
\end{Thm}

\begin{Thm}[Fiberwise convexity of Hill's lunar problem, \cite{L}]
The bounded component of the regularized Hill's lunar problem is fiberwise convex for all energy below the critical level.
\end{Thm}

We introduce the Hamiltonians 
\begin{gather*}
H_R : T^*(\mathbb{R}^2 -\{0\})=(\mathbb{R}^2 -\{0\}) \times \mathbb{R}^2 \rightarrow \mathbb{R}
\\ H_R(q, p)={1 \over 2}|p|^2-{1 \over |q|}+p_1 q_2-p_2 q_1
\end{gather*}
and
\begin{gather*}
H_H : T^*(\mathbb{R}^2 -\{0\})=(\mathbb{R}^2 -\{0\}) \times \mathbb{R}^2 \rightarrow \mathbb{R}
\\ H_H (q, p)={1 \over 2}|p|^2-{1 \over |q|}+p_1 q_2-p_2 q_1-q_1^2+{1 \over 2}q_2^2
\end{gather*}
of the rotating Kepler problem and Hill's lunar problem, respectively. One can easily see that $H_R$ and $H_H$ have one critical value $-c_R^0=-{3 \over 2}$ and $-c_H^0=-{3^{4 \over 3} \over 2}$, respectively. We define the bounded component of each problem
\begin{gather*}
\Sigma_{R}^c:=(\overline{\Phi \circ \Psi(H_{R}^{-1}(-c))})^b,
\\ \Sigma_{H}^{c'}:=(\overline{\Phi \circ \Psi(H_{H}^{-1}(-c'))})^b
\end{gather*}
where the overlines denote the closure in $T^* S^2$ and superscripts $b$ denote the bounded component in $T^* S^2$ for each $c>c_R^0$ and $c'>c_H^0$. We will call $\Sigma_{R}^c$($\Sigma_{H}^c$) the energy hypersurface of the regularized rotating Kepler problem(Hill's lunar problem) at energy $-c$ for each $c>c_R^0$ $(c>c_H^0)$.The above Theorems mean that $\Sigma_{R}^c$ and $\Sigma_H^{c'}$ are fiberwise convex hypersurfaces for each $c>c_R^0$ and $c'>c_H^0$. Therefore, we can regard each of the rotating Kepler problem and Hill's lunar problem as a geodesic problem on $S^2$ equipped with a Finsler metric. One immediate consequence of fiberwise convexity of a Hamiltonian problem is admitting a contact structure of the Hamiltonian problem. Therefore, we can apply Theorem in contact topology. By Eliashberg' work in \cite{Eli2}, there is a unique tight contact sturcture up to isotopy. From the criterion due to Eliashberg and Gromov \cite{Eli1}, \cite{Gromov}, any symplectically fillable contact 3-manifold is tight. Because regularized energy hypersurfaces $\Sigma_R^c$ and $\Sigma_H^{c'}$ are symplectically fillable and diffeomorphic to $\mathbb{R}P^3$, we have the following Corollary.

\begin{Cor}
The bounded component of the regularized rotating Kepler problem and the regularized Hill's lunar problem has a contact structure for the energy level below each critical value. Moreover, these contact structures are the unique tight contact structure on $\mathbb{R}P^3$ up to contact isotopy.
\end{Cor}

Actually, what we need for the contact structure is fiberwise star-shapedness. If one can prove that a Hamiltonian $H$ defines a fiberwise star-shaped hypersurface $\Sigma$ in $T^* S^2$ along this procedure, then we can think of the Hamiltonian flow as a Reeb flow of the contact manifold $(\Sigma, \lambda_{can})$ up to reparametrization. Moreover, the inside $M$ of $\Sigma$ in $T^* S^2$ defines a Liouville domain. Because  the tools in this paper can be applied to any Liouville domain defined by a fiberwise star-shaped hypersurface in a cotangent bundle, it is worthwhile to mention the contact structure of the restricted three body problem in \cite{AFvKP}.

\begin{Thm}[Albers-Frauenfelder-van Koert-Paternain]
For a energy $c$ below the first critical value, two bounded components $\Sigma_E^c$ and  $\Sigma_M^c$ of the regularized restricted three body problem in $T^* S^2$ admit compatible contact forms, respectively. Moreover, there exists $\epsilon>0$ such that for $-c \in (H(L_1), H(L_1)+\epsilon)$ the bounded component $\Sigma_{E, M}^c$ admits a compatible contact form $\lambda$.
\end{Thm}

In \cite{AFvKP}, they opened the possibility of using contact topology in order to understand the dynamics of the restricted three body problem.

\begin{Cor}[Albers-Frauenfelder-van Koert-Paternain]
For $-c <H(L_1)$, the contact structures of $(\Sigma_E^c, \ker \lambda_{can})$ and $(\Sigma_M^c, \ker  \lambda_{can})$ are the tight contact structure on $\mathbb{R}P^3$ up to contact isotopy. Moreover, for $-c \in (H(L_1), H(L_1)+\epsilon)$, the contact structure of $(\Sigma_{E, M}^c, \ker \lambda)$ is the tight contact structure on $\mathbb{R}P^3 \# \mathbb{R}P^3$
\end{Cor}

 A challenging problem is to get a suitable action spectrum estimate or systole bound for the restricted three body problem using the methods in this paper.

\section{Conley-Zehnder indices and action spectrum of the rotating Kepler problem}

In symplectic topology, it is important to know the Conley-Zehnder indices of periodic orbits. These indices play important roles in finite energy foliations, Floer homology and so on. In a general Hamiltonian problem, these indices are hard to compute. Moreover, we do not know where and how many periodic orbits are in many Hamiltonian problem. However, in the rotating Kepler problem, the Conley-Zehnder indices of all periodic orbits for energies below the critical value were completely determined in \cite{AFFvK}. We will introduce the result briefly and we will compute the action value of each orbits. This will lead us to understand the chain complex structure in the symplectic homology of the Liouville domains determined by the rotating Kepler problem.

\subsection{Conley-Zehnder indices of the rotating Kepler problem}

In this section, we will recall the result in \cite{AFFvK}. They use the Conley-Zehnder index defined in \cite{HWZ}. Because the rotating Kepler problem is  time-independent, there is the always present degeneracy if we use the definition in Section 4.1. In \cite{HWZ}, they use the restirction to the contact plane of hypersurface and so according to this definition, the retrograde and direct orbits are generically nondegenerate. In \cite{AFFvK}, they compute directly the indices of the retrograde and direct orbits using the suitably chosen trivialization of the contact structure. For the noncircular orbits $T_{k, l}$, the Conley-Zehnder indices are computed by using the fiberwise convexity of the regularized rotating Kepler problem. Because one can interpret the periodic orbits as critical points of the energy functional associated to a Finsler metric, the Conley-Zehnder index agrees with the Morse index of the energy functional. Then one can use the local invariance of Morse homology to determine the Morse index of $T_{k,l}$ at the bifurcatiion point. For example, if one has a degenerate orbit of $S^1$-family with the Conley-Zehnder index $k$, then this will become the nondegenerate obits of Conley-Zehnder index $k$ and $k+1$ after suitable perturbation in Section 4.1.

We introduce the notations. We start with the Hamiltonian of the Kepler problem
\begin{gather*}
H_{KP}: T^*(\mathbb{R}^2-\{(0,0)\} ) \rightarrow \mathbb{R},
\\ H_{KP}(q, p)={1 \over 2}|p|^2-{1 \over |q|}
\end{gather*}
and we know this has angular momentum integral
\begin{gather*}
L:=q_1p_2-q_2p_1
\end{gather*}
because $H_{KP}$ is invariant under rotations around $0$. The Hamiltonian $H_{KP}$ is an integral as well because it is time-independent. We will denote $H_{KP}$ by $E$ for the notational convenience. As we have seen in Moser regularization, every orbit of the Hamiltonian equation for $H_{KP}$ is periodic orbit, including the collision orbit after regularization, for negative energy. In fact, we know the orbits are either ellipses of eccentricity $\epsilon:=\sqrt{2EL^2+1}$ or collision orbits by Kepler's laws of planetary motion. Moreover, we have the equality 
\begin{gather*}
T^2=-{\pi^2 \over 2 E^3}
\end{gather*}
for the period $T$ of the ellipse. The rotating Kepler problem is the Kepler problem in a rotating coordinate system. The Hamiltonian of the rotating Kepler problem is given by
\begin{gather*}
H_{RKP}=E+L={1 \over 2}|p|^2-{1 \over |q|}+q_1p_2-q_2p_1
\end{gather*}
in our convention. This has the unique critical value $-c_R^0:=-{3 \over 2}$. We are interested in the energy hypersurfaces below this critical value. We can easily see that $E, L$ are integrals of the rotating Kepler problem and invariant under the Hamiltonian flow. Even though every orbit is periodic in the Kepler problem, not every orbit is periodic in the rotating Kepler problem. In fact, the Hamiltonian flow of the rotating Kepler problem is given by the composition
\begin{gather*}
\phi_{L}^t \circ \phi_E^t
\end{gather*}
of two Hamiltonian flows where $\phi_{L}^t$ is the rotation 
\begin{gather*}
\phi_{L}^t : \mathbb{R}^4=\mathbb{R}^2 \times \mathbb{R}^2 \rightarrow \mathbb{R}^4=\mathbb{R}^2 \times \mathbb{R}^2
\\ \phi_{L}^t=\begin{pmatrix} \cos t & -\sin t \\ \sin t & \cos t \end{pmatrix} \oplus \begin{pmatrix} \cos t & -\sin t \\ \sin t & \cos t \end{pmatrix}
\end{gather*}
generated by the Hamiltonian $L$ and $\phi_{E}^t$ is the Hamiltonian flow generated by the Hamiltonian $E$. Therefore a periodic orbit in the Kepler problem should satisfy a resonance condition in general in order to be a periodic orbit in the rotating Kepler problem.

First, the circular orbits in the Kepler problem give the circular orbits in the rotating Kepler problem and always give the periodic orbits. By the direction of the rotation of circular orbits in the Kepler problem, we have two types of the circular orbits in the rotating Kepler problem. If we consider the opposite direction of the coordinate rotation, then we have the retrograde orbit and denote by $\gamma_{R}$. If we consider the same direction of rotation for orbit with the coordinate rotation, then we have the direct orbit and denote by $\gamma_{D}$. The circular orbits have the eccentricity $0=\sqrt{2EL^2+1}$. If we fix an energy hypersurface $H_{RKP}^{-1}(-c)$, then we have the equation
\begin{gather*}
0=2E(-c-E)^2+1
\end{gather*}
of the value $E$ for the circular orbits. There exist two zeros less than $-{1 \over 2}$ for each $c>{3 \over 2}$. The smaller zero corresponds to the retrograde orbit and the other zero corresponds to the direct orbit.

Second, an ellipse orbit with positive eccentricity in the Kepler problem gives a periodic orbit in the rotating Kepler problem if and only if the period is a rational multiple of $2\pi$. If $T_R=2\pi l$ for some $l \in \mathbb{N}$ and the orbit is a $k$-fold cover of ellipses in the inertial coordinate, then we we call this periodic orbit a $k$-fold covered ellipse in an $l$-fold covered coordinate system and denote it by $\gamma_{k, l}$. In the circular orbit case, there exist a retrograde orbit and a direct orbit for each $c>{3 \over 2}$, up to reparametrization. On the other hand, $\gamma_{k,l}$ does not exist always. We discuss the energy values where $\gamma_{k,l}$ exists for each $k,l$. From the definition of $\gamma_{k,l}$, the period of underlyng ellipse in the Kepler problem is $T={2\pi l \over k}$. Using Kepler's law $T^2=-{\pi^2 \over 2E^3}$, we have
\begin{gather*}
{4\pi^2 l^2 \over k^2}=-{\pi^2 \over 2E^3}
\end{gather*}
and the energy level $E_{k, l}$ of this underlying ellipse of $\gamma_{k,l}$ is
\begin{gather*}
E_{k, l}=-{1 \over 2}({k \over l})^{2 \over 3}
\end{gather*} 
for each $k,l \in \mathbb{N}$. In fact, we only consider the energy $E<{-1 \over 2}$ and so we will assume $k>l$. From the eccentricity equation, $\gamma_{k, l}$ can exist only when the inequality
\begin{gather*}
0<2E_{k,l}(c+E_{k,l})^2+1
\end{gather*}
holds. We solve this inequality for $c$. Then we have the energy range
\begin{gather*}
c_{k,l}^- <c<c_{k,l}^+
\end{gather*}
for $\gamma_{k,l }$ where
\begin{gather*}
c_{k,l}^- :=-E_{k,l}-\sqrt{1 \over {-2E_{k,l}}},
\\ 
c_{k,l}^+ :=-E_{k,l}+\sqrt{1 \over {-2E_{k,l}}}.
\end{gather*}
At $c=c_{k,l}^-$, the eccentricity is $0$ and $L=-E_{k,l}-c_{k,l}^+=\sqrt{1 \over {-2E_{k,l}}}>0$. This means that $\gamma_{k,l}$ is the multiple cover of the retrograde orbit. In fact, the periodic orbit $\gamma_{k,l}$ degenerates to $k+l$-fold cover of the retrograde orbit at $c=c_{k,l}^-$. Similarly, the periodic orbit $\gamma_{k,l}$ degenerates to $k-l$-fold cover of the direct orbit at $c=c_{k,l}^+$. Using direct computation with suitably chosen trivialization, Morse index theory with fiberwise convexity of the rotating Kepler problem and bifurcation argument, they determined all Conley-Zehnder indices of above orbits.

\begin{prop}[Albers-Fish-Frauenfelder-van Koert, \cite{AFFvK}]
We define the $N$-th iteration of $\gamma_{R}$ and $\gamma_D$ by $\gamma_{R, N}$ and $\gamma_{D, N}$, respectively.
The Conley-Zehnder indices of $\gamma_{R, N}$ and $\gamma_{D, N}$ are given by
\begin{gather*}
\mu_{CZ}(\gamma_{R,N})=1+2 \max \{ k \in \mathbb{Z} | k{2 \pi \over (-2E)^{3 \over 2}}<NS_{R} \}
\end{gather*}
and
\begin{gather*}
\mu_{CZ}(\gamma_{D,N})=1+2 \max \{ k \in \mathbb{Z} | k{2 \pi \over (-2E)^{3 \over 2}}<NS_{D} \}
\end{gather*}
for $NS_{R}, NS_D \notin \mathbb{Z} {2 \pi \over (-2E)^{3 \over 2}}$ where $S_{R}={2\pi \over {(-2E)^{3 \over 2} + 1}}$ and $S_{D}={2\pi \over {(-2E)^{3 \over 2} - 1}}$ are the periods of $\gamma_{R}$ and $\gamma_{D}$, respectively. Moreover, the Conley-Zehnder index of $T_{k, l}$ is $2k-1$ for each $k>l \ge 1$.
\end{prop}
From the above computation of Conley-Zehnder indices of all periodic orbits, in \cite{AFFvK}, they proved the dynamically convexity and therefore there exists a global disk-like surfaces of sections for each energy hypersurface of the rotating Kepler problem after the Levi-Civita transformation using the following remarkable statement in \cite{HWZ}.

\begin{Def}[Hofer-Wysocki-Zehnder]
Let $(\Sigma, \xi= \ker \lambda)$ be a contact 3-manifold. The contact form $\lambda$ is called dynamically convex if $c_1(\xi)$ vanishes on $\pi_2(N)$ and $\mu_{CZ}(\gamma) \ge 3$ for any contractible Reeb periodic orbit $\gamma$.
\end{Def}

\begin{Thm}[Hofer-Wysocki-Zehnder]
Let $\lambda$ be a dynamically convex contact form on $S^3$. Then there exists a disk-like global surfaces of section for Reeb vector field.
\end{Thm}

Because, in general, one do not know all Reeb orbits of a contact three manifold, it is hard to determine dynamical convexity. However, in \cite{HWZ}, they gave a useful criterion for dynamical convexity.

\begin{Thm}[Hofer-Wysocki-Zehnder]
A strictly convex regular energy hypersurface $\Sigma$ of $\mathbb{R}^4$ with the canonical contact form $\lambda_{can}$ is dynamically convex. 
\end{Thm}

As an application of above Theorems, one can see the following result for the restricted three body problem in \cite{AFFHvK}.

\begin{Thm}[Albers-Fish-Frauenfelder-Hofer-van Koert]
Given $c>{3 \over 2}$, there exists $\mu_0=\mu_0 (c) \in [0,1)$ such that for all $\mu_0<\mu<1$ there exists a disk-like global surface of section for the hypersurface of the Levi-Civita regularized restricted three body problem of mass ratio $\mu$ with its Reeb vector field.
\end{Thm}

In other words, they proved that for such pairs $(\mu, c)$, the Levi-Civita regularized energy hypersurfaces are strictly convex. On the other hand, in \cite{AFFvK}, they also proved the fail of strict convexity for energy hypersurfaces of the rotating Kepler problem after Levi-Civita transformation. Thus Theorem 3.3 cannot be used for the rotating Kepler problem. In this point of view, one can ask whether the rotating Kepler problem has the convex embedding or not. If there is convex embedding, then this provides another proof of dynamical convexity. If there is no convex embedding, then it can be one example showing the gap of strict convexity and dynamical convexity. Because one is topological and geometric property and the other is symplectic property, it is worthwhile to find such a gap. At this moment, it is still open question.

\subsection{Spectrum of the rotating Kepler problem}

Another important ingredient of symplectic homology is the periods of Reeb periodic orbits. Let $(\Sigma, \xi=\ker \lambda)$ be a co-oriented contact manifold. We define the Reeb vector field $R_{\lambda}$
\begin{gather*}
\lambda(R_{\lambda})=1, \quad \iota_{R_{\lambda}} d \lambda=0
\end{gather*}
associated with the contact form $\lambda$. The set of all positive periods of closed Reeb orbits is called the action spectrum. We will denote this by
\begin{gather*}
Spec(\Sigma, \lambda)
\end{gather*}
One can compute the period of the closed Reeb orbit $\gamma$ by the integration $\mathcal{A}(\gamma)=\int_{\gamma} \lambda$. We have seen that $(\Sigma_R^c,  \lambda_{can})$ is a contact manifold for each $c>c_R^0={3 \over 2}$. We will compute the period of every closed Reeb orbit in $(\Sigma_R^c, \lambda_{can})$. One can compute the period of the closed Reeb orbit $\gamma$ by the integration $\int_{\gamma} \lambda$.

First, we will compute the period of the retrograde circular orbit $\gamma_R$. Let $r$ be the distance from the origin for the circular orbit of the Kepler problem. In $q$-coordinate, the circular orbit can be parametrized as follows
\begin{gather*}
q_{K}^r (t)=(r \cos(\omega t), r \sin(\omega t)).
\end{gather*}
We have to determine the frequency $\omega$. From the Hamiltonian equation $\dot{q}={\partial E \over \partial p}$, we have
\begin{gather*}
p_{K}^r (t)=\dot{q_{K}^r }(t)=(-r\omega \sin(\omega t), r \omega \cos (\omega t))
\end{gather*}
and from the Hamiltonian equation $\dot{p}=-{\partial E \over \partial q}$, we have
\begin{gather*}
\dot{p_K^r }(t)=(-r\omega^2 \cos(\omega t), -r \omega^2 \sin (\omega t))=-{q \over |q|^3}=(-r^{-2} \cos (\omega t), -r^{-2} \sin (\omega t)).
\end{gather*}
This implies $\omega=r^{-{3 \over 2}}$. For this circular periodic orbit, the energy is given by $E=-{1 \over 2r}$. The corresponding retrograde orbit in the rotating Kepler problem has the following parametrization
\begin{gather*}
q_R^r (t)=(r \cos((r^{-{3 \over 2}}+1)t), r \sin((r^{-{3 \over 2}}+1)t))
\end{gather*}
on the $q$-coordinate. From the Hamiltonian equation $\dot{q}={\partial H_{RKP} \over \partial p}$, we have
\begin{gather*}
\begin{cases} \dot{q_1}=p_1-q_2 \\ \dot{q_2}=p_2+q_1 \end{cases} 
\Rightarrow 
\begin{cases} p_1=\dot{q_1}+q_2 \\ p_2=\dot{q_2}-q_1 \end{cases}.
\end{gather*}
This implies
\begin{gather*}
p_R^r (t)=(-r^{-{1 \over 2}} \sin ((r^{-{3 \over 2}}+1)t), r^{-{1 \over 2}} \cos ((r^{-{3 \over 2}}+1)t))
\end{gather*}
We define $\gamma_R^r (t):=(q_R^r (t), p_R^r (t))$ and compute the integral 
\begin{eqnarray*}
\mathcal{A}(\gamma_R^r )&=&\int_{\Phi \circ \Psi(\gamma_R^r )} \lambda_{can}
\\&=& \int_{\gamma_R^r} (\Phi \circ \Psi)^{*} \lambda_{can}
\\&=& \int_{\gamma_R^r} -q dp
\\&=& \int_{0}^{2\pi \over {r^{-{3 \over 2}}+1}} (r^{-1}+r^{1 \over 2}) dt= 2\pi r^{1 \over 2}
\end{eqnarray*}
where $\Phi, \Psi$ are symplectomorphisms defined in Section 2. We can express the energy $-c$
\begin{gather*}
-c=H_{RKP}(\gamma_R^r (t))=E+L=-{1 \over 2r}+r^{1 \over 2} \iff c={1 \over 2r}-r^{1 \over 2}
\end{gather*}
for this retrograde orbit in terms of $r$. In sum, the retrograde orbit $\gamma_R^r$ of radius $r$ has the action
\begin{gather*}
\mathcal{A}(\gamma_R^r)=2\pi r^{1 \over 2}
\end{gather*}
and it is on the energy hypersurface
\begin{gather*}
\gamma_R^r \subset H_{RKP}^{-1}(-{1 \over 2r}+r^{1 \over 2})
\end{gather*}
of energy $-{1 \over 2r}+r^{1 \over 2}$.

We can similarly compute the action and energy for the direct orbit of radius $r$. Let $\gamma_D^r (t)=(q_D^r(t), p_D^r (t))$ be the direct orbit of radius $r$ in the rotating Kepler problem. Then we have
\begin{gather*}
q_D^r(t)=(r \cos ((-r^{-{3 \over 2}}+1)t), r \sin ((-r^{-{3 \over 2}}+1)t)),
\\ p_D^r (t)=(r^{-{1 \over 2}} \sin ((-r^{-{3 \over 2}}+1)t), -r^{-{1 \over 2}} \cos ((-r^{-{3 \over 2}}+1)t))
\end{gather*}
by similar computation in above. We can compute the action
\begin{gather*}
\mathcal{A}(\gamma_D^r)=2\pi r^{1 \over 2}
\end{gather*}
and this direct orbit is on the energy level
\begin{gather*}
-c=H_{RKP}(\gamma_D^r(t))=E+L=-{1 \over 2r}-r^{1 \over 2}.
\end{gather*}

Before we go to the non-circular orbits, we want to express the action values of the retrograde and direct orbits in terms of $L$. In the retrograde orbit case, we have $L=r^{1 \over 2}$ and we have $-c=E+L=-{1 \over 2L^2}+L$. Therefore, the action of retrograde orbit $\gamma_R^c$ on $H_{RKP}^{-1}(c)$ is given by
\begin{gather*}
\mathcal{A}(\gamma_R^c)=2 \pi L_R (c)
\end{gather*}
where $L_R (c)$ is the positive zero of an equation
\begin{gather*}
0=-2(x+c)x^2+1 \iff c={1 \over 2x^2}-x
\end{gather*}
for $x$. In the direct orbit case, we have $L=-r^{1 \over 2}$. We also have $-c=E+L=-{1 \over 2L^2}+L$ and $r<1$. Therefore, the action of direct orbit $\gamma_D^c$ on $H_{RKP}^{-1}(c)$ is given by
\begin{gather*}
\mathcal{A}(\gamma_D^c)=-2 \pi L_D (c)
\end{gather*}
where $-1<L_D (c)<0$ is the larger negative zero of the equation
\begin{gather*}
c={1 \over 2x^2}-x
\end{gather*}
for $x$. We define the function $f(x):={1 \over 2x^2}-x$.

\begin{figure}
\centering
\includegraphics[]{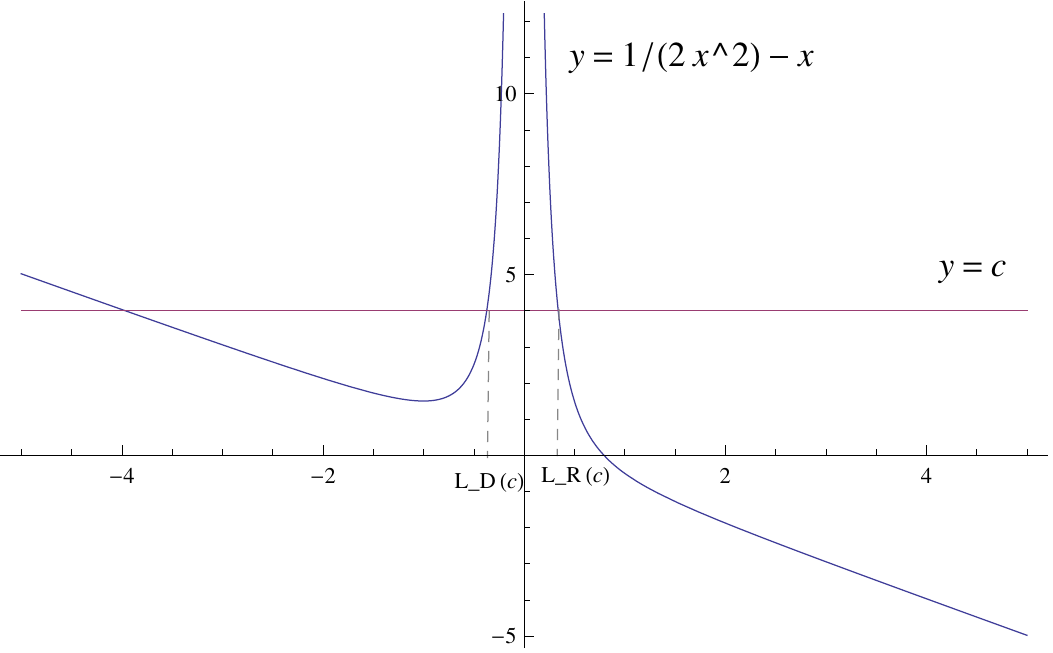}
\caption{Definition of $L_R (c)$ and $L_D(c)$}
\end{figure}

Finally, we have to compute the action of $\gamma_{k,l}$. We recall that $\gamma_{k,l}$ denotes a $k$-fold covered ellipse in an $l$-fold covered coordinate system and thus the period of $\gamma_{k,l}$ is $T_{k, l}=2\pi l$ and the energy of the underlying ellipse is $E_{k,l}=-{1 \over 2}({k \over l})^{2 \over 3}$. In the Kepler problem, every simple periodic orbit in a fixed energy has the same action value because every simple periodic orbit corresponds to the great circle in the standard $S^2$ with round metric by Moser regularization. In fact, the action value of any simple periodic orbit $\gamma^c$ on $E^{-1}(-c)$ is given by 
\begin{gather*}
\mathcal{A}_{KP}(\gamma^c)=2\pi (2c)^{-{1 \over 2}}
\end{gather*}
for each $c>0$. Define
\begin{gather*}
\lambda:=(\Phi \circ \Psi)^* \lambda_{can}=-q dp
\end{gather*}
for the next computation. We compute the action of $\gamma_{k, l}$ as follows.
\begin{eqnarray*}
\mathcal{A}(\gamma_{k,l})&=&\int_{\Phi \circ \Psi(\gamma_{k, l})} \lambda_{can}=\int_{\gamma_{k,l}} \lambda
\\ &=& \int_{0}^{T_{k,l}} \lambda(\dot{\gamma_{k,l}}(t)) dt
\\ &=& \int_{0}^{T_{k,l}} \lambda(X_{H_{RKP}}(\gamma_{k,l}(t))) dt
\\ &=& \int_{0}^{T_{k,l}} \lambda(X_E (\gamma_{k,l}(t))+X_L (\gamma_{k,l}(t))) dt
\\ &=& \int_{0}^{T_{k,l}} \lambda(X_E (\gamma_{k,l}(t))) dt + \int_{0}^{T_{k,l}} \lambda(X_L (\gamma_{k,l}(t))) dt
\\ &=& k(2\pi (-2E_{k,l})^{-{1 \over 2}})+\int_{0}^{2\pi l} L (\gamma_{k,l}(t)) dt
\\ &=&2\pi k \sqrt{1 \over {-2E_{k,l}}}+2\pi lL
\end{eqnarray*}
If we consider the periodic orbit $\gamma_{k,l}^c$ on $H_{RKP}^{-1}(-c)$, then we have
\begin{gather*}
-c=E_{k, l}+L
\end{gather*}
and thus we have
\begin{gather*}
\mathcal{A}(\gamma_{k,l}^c)=2\pi k \sqrt{1 \over {-2E_{k,l}}}+2 \pi l (-c-E_{k,l})=2\pi(-lc+{3 \over 2}k^{2 \over 3} l^{1 \over 3})
\end{gather*}
for every $c \in (c_{k,l}^-, c_{k,l}^+)$. We have seen that 
\begin{gather*}
c_{k,l}^- :=-E_{k,l}-\sqrt{1 \over {-2E_{k,l}}},
\\
c_{k,l}^+ :=-E_{k,l}+\sqrt{1 \over {-2E_{k,l}}}.
\end{gather*}
and so
\begin{gather*}
c_{k,l}^- ={1 \over 2}({k \over l})^{2 \over 3}-({l \over k})^{1 \over 3}, \quad c_{k,l}^+ ={1 \over 2}({k \over l})^{2 \over 3}+({l \over k})^{1 \over 3}.
\end{gather*}
We note that
\begin{gather*}
c_{k,l}^- =f(({l \over k})^{1 \over 3}), \quad c_{k,l}^+ =f(-({l \over k})^{1 \over 3}).
\end{gather*}
and this implies that
\begin{gather*}
{1 \over 2}({k \over l})^{2 \over 3}-({l \over k})^{1 \over 3}<c<{1 \over 2}({k \over l})^{2 \over 3}+({l \over k})^{1 \over 3}
\\ \iff L_R(c)^3<{l \over k}<-L_D (c)^3
\end{gather*}
where $L_R (c)>0$ and $-1<L_D (c)<0$ are zeros of $c=f(x)={1 \over 2x^2}-x$. We have proved the following Proposition.

\begin{prop}
Let $Spec(\Sigma_R^c, \lambda_{can})$ be the set of actions of the energy hypersurfaces of regularized the rotating Kepler problem at energy $-c$. Then we have
\begin{eqnarray*}
Spec(\Sigma_{R}^c, \lambda_{can})&=&2\pi L_R (c) \mathbb{N} \cup (-2\pi L_D (c)) \mathbb{N} 
\\ &\cup& \{ 2\pi (-lc+{3 \over 2} k^{2 \over 3} l^{1 \over 3}) | {l \over k} \in (L_R (c)^3, -L_D (c)^3), k>l \textrm{ and } k,l \in \mathbb{N} \}
\end{eqnarray*}
for each $c>{3 \over 2}$. The values $2\pi L_R(c)$ and $-2 \pi L_D(c)$ are the actions of the retrograde and direct orbit, respectively, where 
\begin{gather*}
L_R (c)>0 \textrm{ and } -1<L_D (c)<0
\end{gather*}
are zeros of $c=f(x)={1 \over 2x^2}-x$.
\end{prop}

We can have explicit formulas
\begin{gather*}
L_R (c)={1 \over 2}\sqrt{{3 \over 2c}} \sec \left({1 \over 3} \arccos \left( \left({3 \over 2c} \right)^{3 \over 2}\right)\right),
\\ L_D (c)={1 \over 2}\sqrt{{3 \over 2c}} \sec \left({1 \over 3} \arccos \left( \left({3 \over 2c} \right)^{3 \over 2}\right)+{2 \pi \over 3}\right)
\end{gather*}
for the zeros of $c={1 \over 2x^2}-x$ using trigonometric identity.

\begin{figure}
\centering
\includegraphics[]{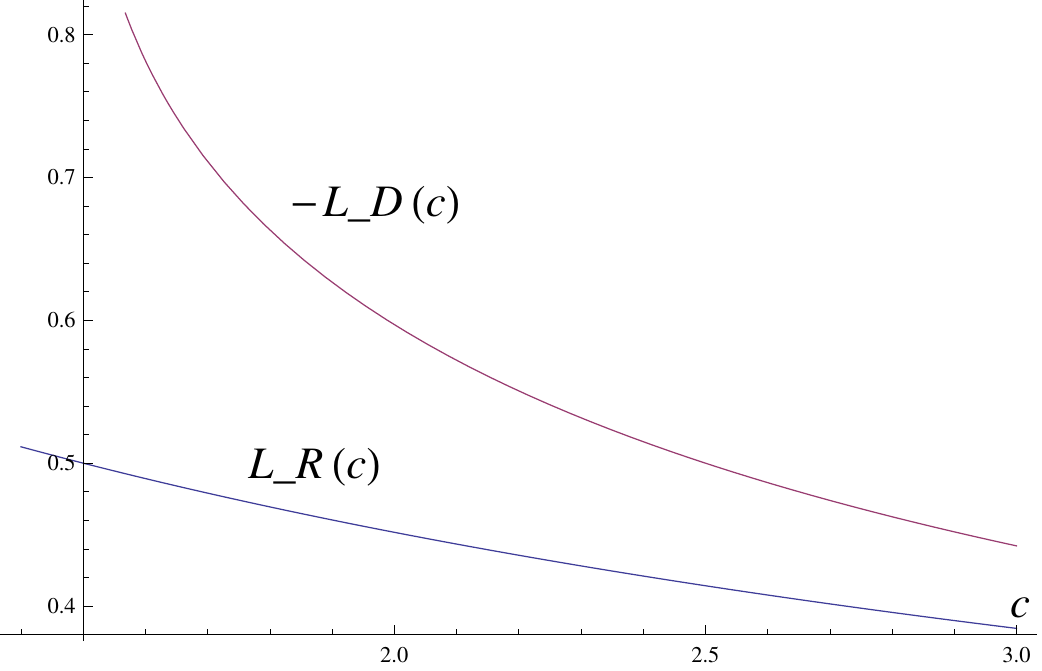}
\caption{The graphs of $L_R (c)$ and $-L_D(c)$ with variable $c$}
\end{figure}

As one can expect and one can see in Figure 4, it is easy to see that the retrograde orbit is the smallest action orbit in the rotating Kepler problem. Even though it is not the mainstream of this paper, it is worthwhile to mention about the systolic volume of the contact manifold $(\Sigma_R^c, \lambda_{can})$. We will see the systolic volume of $(\Sigma_R^c, \lambda_{can})$ for each $c$ in Appendix.

\section{Symplectic capacity of fiberwise star-shaped domains in cotangent bundle}

\subsection{Symplectic homology of Liouville domain}

In this paper, we will use the symplectic homology of cotangent bundles. However, we can define slightly more generally the symplectic homology of Liouville domains without any difference in difficulty. Thus we will define the symplectic homology for Liouville domains. More generally, one can define the symplectic homology for a symplectic manifold $(M, \omega)$ with contact type boundary under the following assumptions.

$\bold{(\Omega)}$: $[\omega]$ vanishes on $\pi_2(M)$.

$\bold{(C)}$: The first Chern class $c_1(M)$ vanishes on $\pi_2(M)$.
\\ One can see this general construction under these assumptions in \cite{CFH} and \cite{V1}. In our case, assumption $\bold{(\Omega)}$ always hold by exactness of symplectic form. Throughout this paper, we will assume that our Liouville domain $(M, \omega)$ satisfies assumption $\bold{(C)}$. This is necessary to define a integer-valued Conley-Zehnder index. Let us recall the definition of Liouville domain.

\begin{Def}
A Liouville domain is a compact symplectic manifold $(M, \omega=d \lambda)$ with a boundary $\partial M$ and a vector field $Y$ satisfying the following conditions.

(1) $L_{Y} \omega= \omega$ or equivalently $\lambda=\iota_{Y} \omega$,

(2) $Y$ transverse to $\partial M$ and pointing outward.
\\ We call the 1-form $\lambda$ the Liouville form and the vector field $Y$ the Liouville vector field.
\end{Def}

We have famous examples of the Liouville domain which satisfy $\bold{(C)}$. In particular, we have to keep in mind the second example throughout this paper.

\begin{Ex}[Star-shaped domain in $\mathbb{R}^{2n}$]
If we take the unit ball $B^{2n}_1(0)=\{x \in \mathbb{R}^{2n} | |x|^2 \le 1\}$ in $\mathbb{R}^{2n}$ with the symplectic form
\begin{gather*}
\omega_{can} = d \lambda_{can} \quad \textrm{ where } \quad \lambda_{can}={1 \over 2}(p dq-q dp)
\end{gather*}
is the canonical Liouville form, then a vector field 
\begin{gather*}
Y_{can}(q,p)={1 \over 2} q{\partial \over \partial q}+{1 \over 2} p{\partial \over \partial p}
\end{gather*}
is the Liouville vector field. The vector field $Y_{can}$ is radial and so transverse to the unit sphere $\partial B^{2n}_1(0)=S^{2n-1}$ pointing outward. Thus $(B^{2n}_1(0), \omega_{can}=d \lambda_{can})$ is a Liouville domain. More generally, if we take a domain $D \in \mathbb{R}^{2n}$ whose boundary $S=\partial D$ is transversal to $Y_{can}$, then $(D, \omega_{can}= d \lambda_{can})$ is a Liouville domain. The condition to have $Y_{can}$-transversal boundary  is the star-shapedness of $D$ with respect to the origin.
\end{Ex}

\begin{Ex}[Fiberwise star-shaped domain in $T^*N$]
Let $(N, g)$ be an orientable Riemannian manifold. If we take the unit disk cotangent bundle $D^*_g N=\{ x \in T^* N | g^*(x, x) \le 1\}$ in $T^* N$ with the canonical symplectic form
\begin{gather*}
\omega_{can} = d \lambda_{can}
\end{gather*}
where $\lambda_{can}$ is the canonical 1-form defined in Section 2.1. We have seen that
\begin{gather*}
\lambda_{can}=p dq, \quad \omega_{can}=dp \wedge dq
\end{gather*}
in any canonical local coordinate system of $T^*N$. Thus a vector field $Y_{can}=p {\partial \over \partial p}$ is the Liouville vector field. The vector field $Y_{can}$ is radial on each fiber and so transverse to the unit sphere cotangent bundle $\partial D^*_g N=S^*_g N$ pointing outward. Therefore, $(D^*_g N, \omega_{can}= d \lambda_{can})$ is a Liouville domain for any Riemannian metric $g$. In fact, we do not need the metric in order to define $\omega_{can}, \lambda_{can}$ and $Y_{can}$ and if we choose a domain $M \in T^* N$ whose boundary $\Sigma=\partial M$ is transversal to $Y_{can}$, then $(M, \omega_{can}= d \lambda_{can})$ is a Liouville domain. As in the above Example, the condition to have $Y_{can}$-transversal boundary is the fiberwise star-shapedness of $M$.
\end{Ex}
 
Let $(M, \omega= d\lambda)$ be a Liouville domain with the Liouville vector field $Y$. We define the completion of $M$ by attaching the symplectization cylinder $[1, \infty) \times \partial M$ along $\partial M$ identified with $\{1\} \times \partial M$. Namely, the completion $(\hat{M},\hat{\omega})$ is
\begin{gather*}
\hat{M}=M \cup_{\{1\} \times \partial M} [1, \infty) \times \partial M,
\\ \hat{\omega}=\begin{cases} \omega & \textrm{ on } M \\ d(r\lambda) & \textrm{ on } [1, \infty) \times \partial M \end{cases}
, \quad \hat{\lambda}=\begin{cases} \lambda & \textrm{ on } M \\ r\lambda & \textrm{ on } [1, \infty) \times \partial M \end{cases}
\end{gather*}
where $r$ is the coordinate for the first component $[1, \infty)$ of symplectization cylinder.

Symplectic homology is obtained by taking a limit on a carefully chosen family of Floer homology on $\hat{M}$. First, we will define the Floer homology for a 1-periodic Hamiltonian and later we will specify the type of Hamiltonian that we use for the symplectic homology. Throughout this paper, we will use $\mathbb{Z}_2$-coefficient to avoid the orientation argument. However, our discussion in this section is still valid in general for $\mathbb{Z}$-coefficient by considering the coherent orientation discussed in \cite{FH} and \cite{CFH}.

Let $(\hat{M}, \hat{\omega}=d \hat{\lambda})$  be the completion of a Liouville domain $(M, \omega)$. We choose a time-dependent and 1-periodic Hamiltonian $H:S^1 \times \hat{M} \rightarrow \mathbb{R}$, $S^1=\mathbb{R} / \mathbb{Z}$. We define $H_t(x)=H(t, x)$ for notational convenience. Recall that we have defined the Hamiltonian vector field $X_H^t$ by $\iota_{X_H^t} \hat{\omega}=-dH_t$ and Hamiltonian flow $\phi_H^t$ as the flow of $X_H^t$. We define the action functional 
\begin{gather*}
\mathcal{A}_H(x)=\int_{S^1} x^* \hat{\lambda}-\int_0^1 H(t, x(t)) dt
\end{gather*}
associated to $H$ on the free loop space $\mathcal{L}\hat{M}:=C^{\infty}(S^1, \hat{M})$ of $\hat{M}$.  We want to formulate Morse homology on $\mathcal{L}\hat{M}$ using the action functional $\mathcal{A}_H$ as a Morse function. However, we do not know if $\mathcal{A}_H$ is a Morse function, namely nondegenerate at every critical point. The corresponding concept is the nondegeneracy of Hamiltonians and it is a generic condition as in usual Morse homology theory. We observe the critical point of $\mathcal{A}_H$. We compute the differential of $\mathcal{A}_H$
\begin{eqnarray*}
d \mathcal{A}_H(x)(\hat{v})&=&\int_{0}^{1} {-\hat{\omega}(\dot{x}(t), \hat{v}(t))-dH_t(\hat{v}(t))} dt
\\ &=&\int_{0}^{1} {-\hat{\omega}(\dot{x}(t), \hat{v}(t))+\hat{\omega}(X_H ^t(x(t)), \hat{v}(t))} dt
\\ &=&\int_{0}^{1} {\hat{\omega}(\hat{v}(t), \dot{x}(t)-X_H ^t (x(t)))} dt
\end{eqnarray*}
for a tangent vector $\hat{v} \in T_{x} \mathcal{L}\hat{M}$ at $x \in \mathcal{L}\hat{M}$ where $\dot{ }={d \over dt}$. Here, we interpret the tangent vector $\hat{v} \in T_{x} \mathcal{L}\hat{M}$ as a section of pull-back bundle $x^* T\hat{M}$, namely $\hat{v}(t) \in T_{x(t)}\hat{M}$. From this computation, we know that the loop $x$ is a critical point of $\mathcal{A}_H$ if and only if 
\begin{gather*}
\dot{x}(t)-X_H ^t (x(t)))=0 \quad \textrm{for all }\quad t \in S^1
\end{gather*}
that is, $x$ is a 1-periodic orbit of the Hamiltonian vector field $X_H^t$ of $H$. Precisely, we have 
\begin{gather*}
Crit(\mathcal{A}_H)=\{x \in \mathcal{L} \hat{M} | \dot{x}(t)=X_H ^t (x(t))\}
\end{gather*}
and we will denote this set of all 1-periodic orbits of Hamiltonian vector field $X_H^t$ by $\mathcal{P}_H$.

\begin{Def}
A 1-periodic orbit $x \in \mathcal{P}_H$ is called nondegenerate if the linearized Hamiltonian flow of time 1 map
\begin{gather*}
d\phi_H^1(x(0)): T_{x(0)} \hat{M} \rightarrow T_{x(0)}\hat{M}
\end{gather*}
at $x(0)$ has no eigenvalue 1, that is, if
\begin{gather*}
\det (I-d\phi_H^1(x(0))) \ne 0.
\end{gather*}
We call a Hamiltonian $H \in C^{\infty}(S^1 \times \hat{M})$ nondegenerate if every $x \in \mathcal{P}_H$ is nondegenerate.
\end{Def}

Nondegeneracy is a generic condition and we will assume our Hamiltonian $H$ is nondegenerate. If a Hamiltonian $H$ is nondegenerate, then we have well-defined Conley-Zehnder indices for all $x \in \mathcal{P}_H$. We state briefly the definition of the Conley-Zehnder index $\mu_{CZ}(x)$ of a 1-periodic orbit $x \in \mathcal{P}_H$. We denote by $Sp(2n)$ the group of $2n \times 2n$ symplectic matrices and we define its subset $Sp^* (2n)$ of all $2n \times 2n$ symplectic matrices which do not have 1 as an eigenvalue. We note that $Sp^* (2n)$ is open, dense and has codimension 1 in $Sp(2n)$. We have a Maslov type index on the set

\begin{gather*}
SP(2n)=\{\Psi:[0,1] \rightarrow Sp(2n) | \Psi(0)=I, \Psi(1) \in Sp^*(2n)\}
\end{gather*}
of paths in $Sp(2n)$. We recall the Conley-Zehnder index for a path of symplectic matrices as defined in \cite{S}.

\begin{Thm}
For each $n \in \mathbb{N}$, there is a unique map
\begin{gather*}
\mu^n_{CZ} : SP(2n) \rightarrow \mathbb{R}
\end{gather*}
satisfying the following properties.

$\bold{(Naturality)}$ For any path $\Psi : [0, 1] \rightarrow Sp(2n)$, $\mu_{CZ}(\Psi^{-1} \Phi \Psi)=\mu_{CZ}(\Phi)$

$\bold{(Homotopy)}$ If $\Phi_1$ and $\Phi_2$ are homotopic in $SP(2n)$, then $\mu_{CZ}(\Phi_1)=\mu_{CZ}(\Phi_2)$.

$\bold{(Zero)}$ If $\Phi(s)$ has no eigenvalue on the unit circle for $s>0$, then $\mu_{CZ}(\Phi)=0$.

$\bold{(Product)}$ For $n_1+n_2=n$ and $\Phi_1 \in SP(2n_1), \Phi_2 \in SP(2n_2)$, we can regard $\Phi_1 \oplus \Phi_2$ as an element of $SP(2n)$. Then $\mu_{CZ}(\Phi_1 \oplus \Phi_2)=\mu_{CZ}(\Phi_1)+\mu_{CZ}(\Phi_2)$.

$\bold{(Loop)}$ If $\Psi :[0,1 ] \rightarrow Sp(2n)$ is a loop, then $\mu_{CZ}(\Psi \Phi)=\mu_{CZ}(\Phi)+2m(\Psi)$.

$\bold{(Signature)}$ If $S$ is a symmetric $2n \times 2n$ matrix with $||S||_{op} <2\pi$ and $\Phi(t)=\exp{tJ_0 S}$, then $\mu_{CZ}(\Phi)={1 \over 2} \textrm{sign}(S)$.
\\
In fact, the map $\mu^n_{CZ}$ is uniquely determined by $\bold{(Homotopy)}$, $\bold{(Loop)}$ and $\bold{(Signature)}$ properties.
\end{Thm}

For the Conley-Zehnder index of a 1-periodic orbit $x \in \mathcal{P}_H$, if $x$ is contractible then we take a symplectic filling $\bar{x}: D \rightarrow \hat{M}$ and we take a symplectic trivialization
\begin{gather*}
\bar{\Gamma} : \bar{x}^* T\hat{M} \rightarrow D \times \mathbb{R}^{2n}
\end{gather*}
for a symplectic vector bundle $\bar{x}^* T\hat{M} \rightarrow D$. This trivialization induces a trivialization
\begin{gather*}
\Gamma : x^* T\hat{M} \rightarrow S^1 \times \mathbb{R}^{2n}
\end{gather*}
of the subbundle $x^* T \hat{M} \rightarrow S^1$ by restriction. We obtain a path of symplectic matrices 
\begin{gather*}
\Phi^{\Gamma}_{x}(t)=\Gamma(t) d\phi_H^t(x(0)) \Gamma(0)^{-1} \in Sp(2n), \quad t \in [0, 1]
\end{gather*}
from the linearized Hamiltonian flow $d\phi_H^t$. Nondegeneracy of $H$ implies $\Phi^{\Gamma}_{x} \in SP(2n)$. We define the Conley-Zehnder index of $x$ with respect to $\bar{x}, \bar{\Gamma}$ by
\begin{gather*}
\mu_{CZ}(x;\bar{x}, \bar{\Gamma}):=\mu_{CZ}(\Phi^{\Gamma}_{x}).
\end{gather*}
By the condition $\bold{(C)}$ on the first Chern class, in fact, it is independent of the choices of $\bar{x}$ and $\bar{\Gamma}$ and so we will denote simply by $\mu_{CZ}(x):=\mu_{CZ}(x;\bar{x}, \bar{\Gamma})$. For the Conley-Zehnder index of a noncontractible 1-periodic orbit, we choose a representative $x_c$ and a trivialization $\Gamma_c : x_c^* T\hat{M} \rightarrow S^1 \times \mathbb{R}^{2n}$ for each $0 \ne c \in H_1(\hat{M}; \mathbb{Z})$. For a given $x \in \mathcal{P}_H$ with $[x]=c$, we extend the trivialization $\Gamma_c$ along the 2-cycle connecting $x_c$ and $x$. This induces an trivialization $\Gamma : x^* T\hat{M} \rightarrow S^1 \times \mathbb{R}^{2n}$. Then we can define the Conley-Zehnder index of $x$ as before.

One important ingredient of Morse homology is a Riemannian metric. We want to define the metric on $\mathcal{L}\hat{M}$. For this, we recall the definition of $\omega$-compatible almost complex structure on $M$.

\begin{Def}
Let $(M, \omega)$ be a symplectic manifold(possibly with boundary). We call a section $J$ of $\Gamma(End(TM))$ an almost complex structure on $M$ if $J(x)^2=-id |_{T_x M}$ for all $x \in M$. We call an almost complex structure $J$ on the symplectic manifold $(M, \omega)$ is $\omega$-compatible if $\omega(\cdot, J \cdot)$ defines a Riemannian metric on $M$. We denote the space of all $\omega$-compatible almost structure by $\mathcal{J}_{\omega}(M)$.
\end{Def}

Because we need a particular type of $\hat{\omega}$-compatible almost complex structure on $\hat{M}$ in order to define symplectic homology.

\begin{Def}
Let $(\hat{M}, \hat{\omega}=d\hat{\lambda})$ be a completion of a Liouville domain $(M, \omega=d\lambda)$. An $\hat{\omega}$-compatible almost complex structure $J$ is called SFT-like if it satisfies the following conditions
\\ (1) $J$ preserves the contact hyperplane $\xi=\ker \lambda|_{T\partial M}$ on $(\partial M, \lambda)$.
\\ (2) $JY=R$ and $JR=-Y$ on $\partial M$ where $Y$ is the Liouville vector field and $R$ is the Reeb vector field.
\\ (3) $J$ is invariant under the flow of the Liouville vector field $Y$ in the cylindrical end $[1, \infty) \times \partial M$.
We denote by $\mathcal{J}^{SFT}_{\hat{\omega}}(\hat{M})$ the set of all SFT-like $\hat{\omega}$-compatible almost complex structure on $\hat{M}$.
\end{Def} 

The space $\mathcal{J}^{SFT}_{\hat{\omega}}(\hat{M})$ is nonempty and contractible. One can see the proof of this fact, for example, in \cite{McSal1}. We choose an SFT-like $\hat{\omega}$-compatible structure $J \in \mathcal{J}^{SFT}_{\hat{\omega}}(\hat{M})$. From the definition, one can define the metric from $J$. We denote this metric by $\left< v_1, v_2 \right> _{J}:=\hat{\omega}(v_1, Jv_2)$ for $v_1, v_2 \in T_p \hat{M}$. We consider 1-periodic $\omega$-compatible almost complex structure $J:=\{J_t\}_{t \in S^1}$, that is, $J_t \in \mathcal{J}^{SFT}_{\hat{\omega}}(\hat{M})$ for all $t \in S^1$. With this, we induce a metic on $\mathcal{L}\hat{M}$ by $L^2$-metic using $\left< \cdot, \cdot \right> _{J}$. Let $x \in \mathcal{L}\hat{M}$ be a loop in $\hat{M}$. One can think a vector of the tangent space $T_x \mathcal{L}\hat{M}$ as a vector field along $x$, that is, we identify $\hat{v} \in T_x \mathcal{L}\hat{M}$ with a section $\hat{v} \in \Gamma(x^*T\hat{M})$ of the pull-back bundle of tangent bundle $T\hat{M}$ via $x: S^1 \rightarrow M$. With this identification, we define a metric on $\mathcal{L}\hat{M}$ as follows. Given $\hat{v_1}, \hat{v_2} \in T_{x} \mathcal{L}\hat{M}=\Gamma(x^* T\hat{M})$, we define
\begin{eqnarray*}
\left< \left< \hat{v_1}, \hat{v_2} \right> \right>_{J}:=\int_{0}^{1} \left<\hat{v_1}(x(t)), \hat{v_2}(x(t)) \right> _{J_t} dt=\int_{0}^{1} {\omega_{x(t)}(\hat{v_1}(x(t)), J_t \hat{v_2}(x(t)))} dt.
\end{eqnarray*}

We can deduce the gradient flow line equation for $\mathcal{A}_H$ using above computations. Since we have
\begin{eqnarray*}
d \mathcal{A}_H(x)(\hat{v})&=&\int_{0}^{1} {\hat{\omega}(\hat{v}(x(t)), \dot{x}(t)-X_H ^t (x(t)))} dt
\\ &=& \left< \left< \hat{v}, -J(\dot{x}-X_H ^t) \right> \right>_{J}
\end{eqnarray*}
for any $\hat{v} \in T_x \mathcal{L} \hat{M}$ and $x \in \mathcal{L} \hat{M}$, we have the gradient vector
\begin{gather*}
\nabla \mathcal{A}_H (x)=-J(\dot{x}-X_H ^t)
\\ \textrm{precisely, } \quad \nabla \mathcal{A}_H (x)(x(t))=-J_t (x(t))(\dot{x}(t)-X_H ^t (x(t))) \in T_{x(t)}\hat{M}
\end{gather*}
This induces the gradient flow line 
\begin{gather*}
u : \mathbb{R} \rightarrow \mathcal{L}\hat{M},
\\ {du\over ds}=\nabla \mathcal{A}_H (u(s))
\end{gather*}
of $\mathcal{A}_H$ on $\mathcal{L}\hat{M}$. This is an ODE on an infinite dimensional space. Using the identification $C^{\infty}(\mathbb{R}, \mathcal{L}\hat{M})=C^{\infty}(\mathbb{R} \times S^1, \hat{M})$, we can rewrite this ODE to a PDE on $\hat{M}$. Namely, the gradient flow line $u : \mathbb{R} \times S^1 \rightarrow \hat{M}$ satisfies the perturbed Cauchy-Riemann equation
\begin{gather*}
{\partial u \over \partial s}(s, t)=\nabla \mathcal{A}_H (u(s, t))=-J_t (u(s, t))({\partial u \over \partial t}(s, t)-X_H ^t (u(s, t)))
\\ \iff \partial_s u+J_t(u)(\partial_t u-X_H^t(u))=0.
\end{gather*}

As in the Morse homology,  we will define the boundary map by counting the gradient flow line. Given $x^{\pm} \in \mathcal{P}_H=Crit(\mathcal{A}_H)$, we denote by $\widehat{\mathcal{M}}(x^{-}, x^{+})$ the space of gradient flow lines from $x^{-}$ to $x^{+}$, that is, 
\begin{gather*}
\widehat{\mathcal{M}}(x^{-}, x^{+})=\{u : \mathbb{R} \times S^1 \rightarrow \hat{M}| \partial_s u+J_t(u)(\partial_t u-X_H^t(u))=0, \lim_{s\rightarrow \pm \infty} u(s, t)=x^{\pm}\}.
\end{gather*}
We have $\mathbb{R}$-action on $\mathbb{R} \times S^1$ and we can obtain the unparametrized moduli space by taking quotient by this $\mathbb{R}$-action on $\widehat{\mathcal{M}}(x^{-}, x^{+})$. This quotient is called the moduli space of Floer trajectories and is denoted by
\begin{gather*}
\mathcal{M}(x^{-}, x^{+}):=\widehat{\mathcal{M}}(x^{-}, x^{+}) / \mathbb{R}.
\end{gather*}

Assume now that all elements in $\mathcal{P}_H$ and the gradient trajectories between them are contained in the compact subset of $\hat{M}$. This will be achieved by taking $H$ with suitable assumptions and we will introduce these assumptions later. For a generic $J \in S^1 \times \mathcal{J}^{SFT}_{\hat{\omega}}(\hat{M})$, the moduli space $\mathcal{M}(x^{-}, x^{+})$ is a smooth manifold of dimension $\mu_{CZ}(x^+)-\mu_{CZ}(x^-)-1$ for each $x^-, x^+ \in \mathcal{P}_H$. We define the Floer chain group for $H$
\begin{gather*}
CF_k^{<a}(H):=\mathbb{Z}_2 \left<x \in \mathcal{P}_H | \mu_{CZ}(x)=k, \mathcal{A}_H(x)<a\right>
\end{gather*}
as the $\mathbb{Z}_2$-module generated by the 1-periodic orbits of index $k$ and action less than $a$ for $k \in \mathbb{Z}$ and $a \in \mathbb{R} \cup \{\pm \infty\}$. We abbreviate $CF_k^{<+\infty}(H)=CF_k (H)$. We also define the filtered chain complex
\begin{gather*}
CF_k^{[a,b)}(H):=CF_k^{<b} / CF_k^{<a}
\end{gather*}
for $a<b \in \mathbb{R} \cup \{\pm \infty\}$ and define a boundary map
\begin{gather*}
\partial^{[a,b)}: CF_k^{[a,b)}(H) \rightarrow CF_{k-1}^{[a,b)}(H)
\end{gather*}
on it by
\begin{gather*}
\partial^{[a, b)}(x):=\sum_{\begin{matrix}y \in \mathcal{P}_H, \\ \mu_{CZ}(y)=k-1, \\ a\le \mathcal{A}_H(y)<b \end{matrix}}{\#_{\mathbb{Z}_2}\mathcal{M}(y, x)y}
\end{gather*}
If we have compactness for the moduli spaces, then $\partial^{[a,b)}$ is well defined and indeed a boundary map, that is, it satisfies $\partial^{[a,b)} \circ \partial^{[a,b)}=0$. Under the compactness assumption, we can define the filtered Floer homology groups
\begin{gather*}
FH^{[a,b)}_*(H)=\ker \partial^{[a,b)} / \textrm{im} \partial^{[a,b)}
\end{gather*}
for $a<b \in \mathbb{R} \cup \{\pm \infty\}$. From a short exact sequence of chain complexes
\begin{gather*}
0 \rightarrow CF_*^{[a,b)}(H) \rightarrow CF_*^{[a,c)}(H) \rightarrow CF_*^{[b,c)}(H) \rightarrow 0,
\end{gather*}
we have a long exact sequence of the filtered Floer homology groups
\begin{gather*}
\cdots \rightarrow FH^{[a,b)}_*(H) \rightarrow FH^{[a,c)}_*(H) \rightarrow FH^{[b,c)}_*(H) \rightarrow FH^{[a,b)}_{*-1}(H) \rightarrow \cdots
\end{gather*}

A standard argument in Floer homology theory says that $FH^{[a,b)}_*(H)$ is independent of the choice of $J$. However, $FH^{[a,b)}_*(H)$ depends on the choice of the Hamiltonian. Moreover, $FH^{[a,b)}_*(H)$ cannot be defined for an arbitrary Hamiltonian due to the compactness. We have to specify the Hamiltonians which guarantee compactness results.

\begin{Def}
We call a smooth Hamiltonian $H: S^1 \times \hat{M} \rightarrow \mathbb{R}$ admissible if it satisfies the following conditions
\\ (1) $H$ is nondegenerate.
\\ (2) $H|_{S^1 \times M} \le 0$
\\ (3) $\lim_{r \rightarrow \infty} H(\cdot, r, x)=ar+b$ on the symplectic cylinder $(r, x) \in [1, +\infty) \times \partial M$ for some $a, b \in \mathbb{R}$ such that $0<a \notin Spec(\partial M, \lambda)$. We denote by $Ad(M)$ the set of all admissible Hamiltonian on $\hat{M}$.
\end{Def}

For an admissible Hamiltonian $H \in Ad(M)$, there is a $S^1$-family of SFT-like $\hat{\omega}$-compatible almost complex structure $J$ such that the moduli space $\mathcal{M}(x^-, x^+; H, J)$ is a smooth manifold for each $x^-, x^+ \in \mathcal{P}_H$. Moreover, in fact, the set of all such $S^1$-family of SFT-like $\hat{\omega}$-compatible almost complex structure forms a Baire set in $C^{\infty}(S^1, \mathcal{J}_{\hat{\omega}}^{SFT}(\hat{M}))$. We call such pair $(H, J) \in Ad(M) \times C^{\infty}(S^1, \mathcal{J}_{\hat{\omega}}^{SFT}(\hat{M}))$ an admissible pair. We denote by $\mathcal{N}_{reg}(M)$ the set of all admissible pairs. For an admissible pair $(H, J) \in \mathcal{N}_{reg}(M)$, we can define the filtered Floer homology
\begin{gather*}
FH_*^{[a, b)}(H)
\end{gather*}
for $a<b \in \mathbb{R} \cup \{\pm \infty\}$. Moreover, if we have two admissible pairs $(H_0, J_0), (H_1, J_1) \in \mathcal{N}_{reg}(M)$ such that $H_0(x) \le H_1(x)$ for every $x \in \hat{M}$, then we can take a monotone homotopy, say $(L,J)$, between them satisfying
\begin{gather*}
L : \mathbb{R} \times S^1 \times \hat{M} \rightarrow \mathbb{R}, \quad L_s \in Ad(M), \quad {\partial L \over \partial s} \ge 0,
\\ L(s, t, x)=
\begin{cases}
H_0(t, x) & \textrm{ if } \quad s \le -s_0
\\ H_1(t, x) & \textrm{ if }  \quad s \ge s_0 
\end{cases}
\end{gather*}
where $L_s(t, x):=L(s, t, x)$ and
\begin{gather*}
J : \mathbb{R} \times S^1 \rightarrow \mathcal{J}_{\hat{\omega}}^{SFT}(\hat{M})
\\ J(s, t)=
\begin{cases}
J_0(t) & \textrm{ if } \quad s \le -s_0
\\ J_1(t) & \textrm{ if }  \quad s \ge s_0 
\end{cases}
\end{gather*}
for some large $s_0 \in \mathbb{R}$. Using this pair $(L,J)$, we can define moduli spaces
\begin{eqnarray*}
\mathcal{M}(x, y; L, J)&:=&\{u : \mathbb{R} \times S^1 \rightarrow \hat{M}| \partial_s u+J(s, t)(u)(\partial_t u-X_L(s, t, u))=0, 
\\ & & \lim_{s\rightarrow -\infty} u(s, *)=x, \lim_{s \rightarrow +\infty} u(s, *)=y\}.
\end{eqnarray*}
for each $x \in \mathcal{P}_{H_0}, y \in \mathcal{P}_{H_1}$. For a generic $(L, J)$, the moduli space $\mathcal{M}(x, y; L, J)$ is a smooth manifold of dimension $\mu_{CZ}(y)-\mu_{CZ}(y)$. If we consider the degree 0 map
\begin{gather*}
\phi^{(L, J)}: CF_k^{[a, b)}(H_0) \rightarrow CF_k^{[a, b)}(H_1)
\end{gather*}
by defining
\begin{gather*}
\phi^{(L, J)}(x):=\sum_{\begin{matrix}y \in \mathcal{P}_{H_1}, \\ \mu_{CZ}(y)=k, \\ a\le \mathcal{A}_H(y)<b \end{matrix}}{\#_{\mathbb{Z}_2}\mathcal{M}(x, y; L, J)y}.
\end{gather*}
Then this is a chain map between $CF_*(H_0)$ and $CF_*(H_1)$. Thus $\phi^(L, J)$ induces a natural map
\begin{gather*}
\phi_{(H_0, H_1)}^{(L, J)} : FH_*^{[a, b)}(H_0) \rightarrow FH_*^{[a, b)}(H_1)
\end{gather*}
on the filtered Floer homology. This is well-defined by the compactness of the moduli spaces and the monotone property is used for this compactness. In fact, one can prove that this map is independent of the choice of $L$ by considering the homotopy of homotopies and therefore we denote $\phi_{(H_0, H_1)}^{(L,J)}$ by $\phi_{(H_0, H_1)}$ and we have a direct system
\begin{gather*}
(\mathcal{N}_{reg}(M), \le) \xrightarrow{FH^{[a, b)}} \mathcal{G}Ab
\end{gather*}
where $(\mathcal{N}_{reg}(M), \le)$ is a directed set with the induced partial order from $Ad(M)$, namely $(H_0, J_0) \le (H_1, J_1) \iff H_0(t, x) \le H_1(t, x)$ for all $t \in S^1, x \in \hat{M}$ and $\mathcal{G}Ab$ is the category of graded abelian groups. We define the symplectic homology
\begin{gather*}
SH_*^{[a, b)}(M, \omega):=\lim_{\longrightarrow} FH_*^{[a, b)}(H)
\end{gather*}
of a Liouville domain $(M, \omega=d \lambda)$ with filtration $[a, b)$. From the naturality, we have a long exact sequence of symplectic homology
\begin{gather*}
\cdots \rightarrow SH^{[a,b)}_*(M) \rightarrow SH^{[a,c)}_*(M) \rightarrow SH^{[b,c)}_*(M) \rightarrow SH^{[a,b)}_{*-1}(M) \rightarrow \cdots
\end{gather*}
for each $a<b<c \in \mathbb{R} \cup \{ \pm \infty \}$. In particular, we obtain the following long exact sequence
\begin{gather*}
\cdots \rightarrow SH_* ^{<b}(M) \xrightarrow{i^b_M} SH_* (M) \xrightarrow{j^b_M} SH_* ^{\ge b} \rightarrow SH_{*-1}^{<b}(M) \xrightarrow{i^b_M} \cdots
\end{gather*}
by taking $a=-\infty, c=+\infty$ for each $b \in \mathbb{R}$. This will play an important role to define capacity in the next Section. By definition of direct limit, we have the canonical map
\begin{gather*}
\phi_H^{[a, b)} : FH_*^{[a, b)}(H) \rightarrow SH_*^{[a, b)}(M)
\end{gather*}
for each $(H, J) \in \mathcal{N}_{reg}(M)$ and these canonical maps satisfy the following universal property.
\begin{displaymath}
\xymatrix{
 FH_*^{[a, b)}(H_i) \ar[rr]^{\phi_{(H_i, H_j)}^{[a, b)}} \ar[dr]_{\phi_{H_i}^{[a, b)}} \ar[ddddr]_{\psi_{H_i}} & & FH_*^{[a, b)}(H_j) \ar[dl]^{\phi_{H_j}^{[a, b)}}  \ar[ddddl]^{\psi_{H_j}}
\\
 & SH_*^{[a, b)}(M) \ar@{-->}[ddd]_{\exists! \psi_{M}} &
\\
\\
\\
& X_* &
}
\end{displaymath}

Suppose that $(\hat{M}, \hat{\omega}=d \hat{\lambda})$ is an open exact symplectic manifold. We assume that there exist two Liouville domain $(M_1, \lambda_1) \subset (M_2, \lambda_2) \subset (\hat{M}, \hat{\lambda})$ such that we can identify $\hat{M_1}=\hat{M_2}=\hat{M}$. Then we have $Ad(M_2) \subset Ad(M_1)$ and so this induces a map 
\begin{gather*}
\phi_{M_2, M_2}^{[a, b)} : SH_*^{[a, b)}(M_2) \rightarrow SH_*^{[a, b)}(M_1) 
\end{gather*}
on symplectic homology of $M_1$ and $M_2$. We call this map the monotone morphism.

\begin{Ex}
Let $M_1 \subset M_2$ be compact star-shaped domains in $(\mathbb{R}^{2n}, \omega_{can}=d \lambda_{can})$. Then we can regard $\hat{M_1}=\hat{M_2}=\mathbb{R}^{2n}$ and therefore we have the monotone morphism
\begin{gather*}
\phi_{M_2, M_2}^{[a, b)} : SH_*^{[a, b)}(M_2) \rightarrow SH_*^{[a, b)}(M_1) 
\end{gather*}
on the symplectic homology. In \cite{FH}, they define monotone morphisms more generally for symplectic embeddings and in \cite{FHW} they use this morphism in order to study symplectic embeddings of ellipsoids in $\mathbb{R}^{2n}$ and to classify polydisks in $\mathbb{R}^{2n}$ symplectically. Moreover, they construct a symplectic capacity for domains in $\mathbb{R}^{2n}$.
\end{Ex}

\begin{Ex}
Let $M_1 \subset M_2$ be fiberwise star-shaped domains in $(T^*N, \omega_{can}=d \lambda_{can})$. Then we have that $\hat{M_1}=\hat{M_2}=T^* N$.  Thus we have the monotone morphism
\begin{gather*}
\phi_{M_2, M_2}^{[a, b)} : SH_*^{[a, b)}(M_2) \rightarrow SH_*^{[a, b)}(M_1) 
\end{gather*}
on the symplectic homology. Observing this monotone morphism, we will define a symplectic capacity for fiberwise star-shaped domains in a cotangent bundle in the next Section. 
\end{Ex}

We have defined the symplectic homology for a Liouville domain $(M, \omega=d \lambda)$. However, it is hard to see directly the computation of symplectic homology, the generators of the symplectic homology and so on. Because $Ad(M)$ is too big, we can consider a simpler set instead of $Ad(M)$. In $Ad(M)$, we allow only nondegenerate Hamiltonians and so we consider only the time-dependent Hamiltonian(A time-independent Hamiltonian is automatically a degenerate Hamiltonian due to $S^1$-action of each 1-periodic orbit). If one uses the perturbation argument in \cite{FHW}, then it is possible to consider the time-independent Hamiltonian by requiring transversal nondegeneracy, that is, there is no eigenvalue 1 of linearized Hamiltonian flow for a 1-periodic orbit when we restrict to the contact hyperplane. The Conley-Zehnder index defined above will be replaced by the transversal Conley-Zehnder index obtained by restricting the linearized flow to the contact plane. Moreover, we do not need to take a smooth Hamiltonian if we use the remark about $C^0$-Hamiltonian in \cite{V1}. Hence we will assume that $Ad(M)$ contains transversely nondegenerate time-independent $C^0$-Hamiltonians which satisfy the original conditions as well. Following the argument in \cite{V1}, we will use the following family of time-independent Hamiltonians
\begin{gather*}
K_M^c(x)=\begin{cases} 0 & \textrm{ if } x \in M \\ c(r-1) & \textrm{ if } x=(r, p) \in [1, \infty) \times \partial M \end{cases}
\end{gather*}
on $\hat{M}$ for a Liouville domain $(M, \omega=d \lambda)$ and for $0<c \notin Spec(\partial M, \lambda)$. We note that the family $\{K_M^c\}_{c \in \mathbb{R}^+}\backslash Spec(\partial M, \lambda)$ of functions is cofinal in $Ad(M)$, that is, for any $H \in Ad(M)$ there exists $c \in \mathbb{R}^+$ such that
\begin{gather*}
K_M^c \ge H.
\end{gather*}
This implies that
\begin{gather*}
SH_*^{[a, b)}(M)=\lim_{\underset{c}{\longrightarrow}}{FH_*^{[a, b)}(K_M^c)}.
\end{gather*}

\begin{figure}
\centering
\includegraphics[]{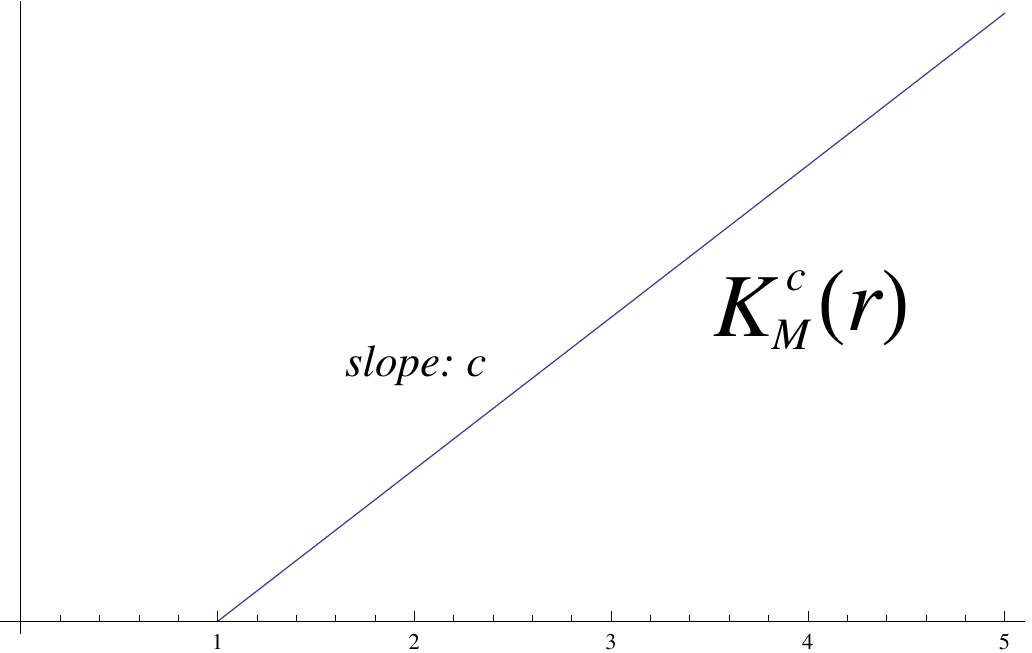}
\caption{$K_M^c : \hat{M} \rightarrow \mathbb{R}$}
\end{figure}

Let $H:\hat{M} \rightarrow \mathbb{R}$ be a time-independent Hamiltonian. We assume that $H$ is $C^2$-small in $M$ and $H(r, x)=h(r)$ on $(r, x) \in [1, \infty) \times \partial M$. Then 1-periodic orbits in $M$ are all constant orbit on the critical points of $H$. We observe the symplectization cylinder part. Because we have
\begin{gather*}
dH(r, x)=dh(r)={dh \over dr}(r) dr \implies X_H(r, x)={dh \over dr}(r)R_{\lambda}(r, x)
\end{gather*}
for $(r, x) \in [1, \infty) \times \partial M$ where $R_{\lambda}(r, x)=(T_r)_*(R_{\lambda}(x))$ for the trivial map $T_r : \partial M \rightarrow \{ r\} \times \partial M$. Let $x :S^1 \rightarrow [1, \infty) \times \partial M$ be a 1-periodic orbit of $H$. Then $x$ lies on a level set, say $\{r \} \times \partial M$. Thus $\dot{x}(t)={dh \over dr}(r)R_{\lambda}(r, x(t))$ and so $x$ is a copy of ${dh \over dr}(r)$-periodic Reeb orbit. Moreover, we have the action value
\begin{eqnarray*}
\mathcal{A}_H(x)&=&\int_{S^1} x^* \hat{\lambda}-\int_{0}^{1}H(x(t)) dt
\\ &=& \int_{0}^{1}{\hat{\lambda}({dh \over dr}(r)R_{\lambda}(r, x))}-\int_{0}^{1}h(r) dt
\\ &=& r{dh \over dr}(r)-h(r)
\end{eqnarray*}
of $x$ in terms of $r, h$. Let us discuss 1-periodic orbits of the Hamiltonian $K_M^c$. We assume that $c \notin Spec(\partial M, \lambda)$ and denote $K_M^c(x, r)=k_M^c(r)$ on the cylinder. In the function $k_M^c$, we can think that every slope between $0$ and $c$ appears exactly once and arbitrarily close to $1$. This implies that the 1-periodic orbits of $K_M^c$ have one-to-one correspondence with the periodic Reeb orbits of period $T \in (0, c)$ in $(\partial M, \lambda)$. Moreover, the action value of a 1-periodic orbit is given by it corresponding Reeb period $T$.

We shall introduce the symplectic homology for our previous examples. Our first example is the star-shaped domain in $\mathbb{R}^{2n}$. It is the simplest example for symplectic homology. In particular, computations of symplectic homologies with any action filtration for ellipsoids and polydisks was done in \cite{FHW}.

\begin{Ex}[Symplectic homology for ellipsoid in \cite{FHW}]
Let $r=(r_1, r_2, \cdots, r_n)$ be an n-tuple of positive real numbers such that $r_1 \le r_2 \le \cdots \le r_n$. We define an open ellipsoid
\begin{gather*}
E(r):=\left\{z \in \mathbb{C}^n | \sum_{k=1}^{n}{\left|  {z_k \over r_k}\right|^2 } < 1 \right\}
\end{gather*}
in $\mathbb{C}^n$. We define the set
\begin{gather*}
\sigma(r):=\left\{ k\pi r_j^2 | k \in \mathbb{N}, j\in \left\{1, 2, \cdots, n\right\}\right\}=\{d_1 \le d_2 \le \cdots \}
\end{gather*}
that allows repeated elements according to the multiplicity. For every $d \in \mathbb{R}\cup\{+\infty\}$, we define a chain complex
\begin{eqnarray*}
C^d(r)&=&0 \quad \textrm{for} \quad d \le 0
\\ C^d(r)&=& (\mathbb{Z}_2, n) \quad \textrm{for} \quad 0 < d \le d_1
\\ C^d(r)&=& (\mathbb{Z}_2, n) \oplus (\mathbb{Z}_2, n+1) \oplus \cdots \oplus (\mathbb{Z}_2, n+2m(d,r)) \quad \textrm{for} \quad d_1 \le d < +\infty
\\ C^{+\infty}(r)&=& \bigoplus_{l=0}^{+\infty}(\mathbb{Z}_2, n+l)
\end{eqnarray*}
where the right component denotes the grade and $m(d,r):=\sup\{l | d_l< d\}$. We also define its quotient
\begin{gather*}
C^{[a, b)}(r):=C^b(r) /C^a(r)
\end{gather*}
The boundary map
\begin{gather*}
\cdots \xrightarrow{id} (\mathbb{Z}_2, n+2m) \xrightarrow{0} (\mathbb{Z}_2, n+2m-1) \xrightarrow{id} \cdots \xrightarrow{id} (\mathbb{Z}_2, n+2) \xrightarrow{0} (\mathbb{Z}_2, n+1) \xrightarrow{id} (\mathbb{Z}_2, n) \xrightarrow{0} 0
\end{gather*}
of infinite chain complex gives the boundary map for each $C^d(r)$ or $C^{[a, b)}$ by restriction. The following result was proved in \cite{FHW}
\begin{gather*}
SH_*^{[a, b)}(E(r))=H_*(C^{[a, b)}(r), \partial^{[a, b)}(r))
\end{gather*}
In particular, we have $SH_*(E(r))=0$.
\end{Ex}

In \cite{FHW}, they answered many embedding problems between two ellipsoids from the information of periodic Reeb orbits because we know every Reeb periodic orbits on $\partial E(r)$. In this paper, we will work in the opposite way. Namely, we will obtain information of Reeb periodic orbit from the embedding relations. We shall see the symplectic homology for our another example. This computation was done in \cite{AS}, \cite{SW} and \cite{V2} independently. We will follow the proof of Abbondandolo-Schwarz.

\begin{Ex}[Symplectic homology for cotangent bundle in \cite{AS}]
Let $M$ be a fiberwise star-shaped domain in $(T^* N, \omega_{can}=d \lambda_{can})$. Then we have the following result.
\begin{ThmAS}[Abbondandolo-Schwarz \cite{AS}, Salamon-Weber \cite{SW}, Viterbo \cite{V2}]
The symplectic homology $SH_* (M)$ is isomorphic to the homology $H_*(\Lambda N)$ of the free loop space of $N$.
\end{ThmAS}
We will give a sketch of proof for this result. In \cite{AS}, they regarded a symplectic homology as a Floer homology on the cotangent bundle and defined special conditions for Hamiltonian $H : S^1 \times T^* N \rightarrow \mathbb{R}$ as follows.

$\bold{(H1)}$: $dH(t, q, p)[Y_{can}]-H(t,q,p) \ge c_0|p|^2-h_1$ for some constants $c_0>0, c_1 \ge 0$.

$\bold{(H2)}$: $|\nabla_q H(t, q, p) \le c_2(1+|p|^2), |\nabla_p H(t, q, p) \le c_2(1+|p|^2)$ for some constant $c_2 \ge 0$
\\
Let $Qd(T^*N)$ be the set of Hamiltonians which satisfy the above conditions $\bold{(H1)}$ and $\bold{(H2)}$. A difficulty of this extension of function is the compactness of moduli spaces. Instead of using the maximum principle, they observe directly the Cauchy-Riemann operator and they get an $L^{\infty}$ estimate. The conditions $\bold{(H1)}$ and $\bold{(H2)}$ allow a Hamiltonian $H$ to have a Lagrangian $L$ satisfying

$\bold{(L1)}$: $\nabla_{vv} L(t, q, v) \ge d_0 I$ for some constant $d_0>0$.

$\bold{(L2)}$: $|\nabla_{qq}L(t, q, v)| \le d_1 (1+|v|^2), \nabla_{qv} L(t,q,v)| \le d_1(1+|v|)$ and $|\nabla_{vv}L(t, q, v)| \le d_1$ for some constant $d_1 \ge 0$.
\\ by the Legendre transformation. Using this Lagrangian $L$, one can consider the Lagrangian action functional
\begin{gather*}
\mathcal{E}_L(x)= \int_{0}^{1} L(t, x(t), \dot{x}(t)) dt
\end{gather*}
on the free loop space $x \in \Lambda N:=W^{1, 2}(S^1, N)$ of $N$. They developed the Morse homology on $\Lambda^1 N$ using $\mathcal{E}_L$ and defined an isomorphism
\begin{gather*}
\Theta : (CM_*(\mathcal{E}_L), \partial_*(\mathcal{E}_L, g)) \rightarrow (CF_*(H), \partial_*(H, J))
\end{gather*}
on the chain levels where $g$ is a Morse-Smale Riemannian metric on $\Lambda N$. This proves the isomorphism between the symplectic homology $SH_*(M)$ and the Morse homology $H_*(\Lambda N)$. This will play an important role to define symplectic capacity using a min-max argument.
\end{Ex}

We shall finish this section with one more example. It is a case of Example 3.1.6. We will use this to apply the symplectic capacity defined in this paper to Hill's lunar problem.

\begin{Ex}[Symplectic homology for $T^* S^2$]
We know that
\begin{gather*}
SH_*(M) \cong H_*(\Lambda S^2)
\end{gather*}
for any fiberwise star-shaped domain $M \in T^* S^2$. We know the homology of $\Lambda S^2$ from the result in string topology, see \cite{CJY} including general computations for loop homologies of spheres and projective spaces.
\begin{gather*}
H_*(\Lambda S^2; \mathbb{Z}_2)=
\begin{cases}
\mathbb{Z}_2 & \textrm{if} \quad *=0, 1
\\ \mathbb{Z}_2 \oplus \mathbb{Z}_2 & \textrm{if} \quad * \ge 2
\end{cases}
\end{gather*}
This can be proved also by Morse homology argument on $(S^2, g_{round})$. Because we know the symplectic homology $SH_*(M)$ from the loop homology. We will determine the chain complex and the boundary map of the Liouville domain $M_R^c$ defined by the rotating Kepler problem. We recall the Conley-Zehnder indices and action spectrum of the rotating Kepler problem from the previous Section. The action values of the $N$th-iterated retrograde and direct orbits are given by
\begin{gather*}
\mathcal{A}(\gamma_{R,N})=2\pi L_R (c)N \quad \textrm{ and } \quad \mathcal{A}(\gamma_{D,N})=-2\pi L_D(c)N,
\end{gather*}
respectively, where $0<L_R(c)$ and $-1<L_D(c)<0$ are zeros of $c={1 \over 2x^2}-x$. From Figure 3, we know that $L_D(c)$ and $L_R(c)$ go to zero as $c$ goes to $+\infty$. The Conley-Zehnder indices of these orbits are given by
\begin{gather*}
\mu_{CZ}(\gamma_{R,N})=1+2 \max \{ k \in \mathbb{Z} | k{2 \pi \over (-2E)^{3 \over 2}}<NS_{R} \}
\end{gather*}
and
\begin{gather*}
\mu_{CZ}(\gamma_{D,N})=1+2 \max \{ k \in \mathbb{Z} | k{2 \pi \over (-2E)^{3 \over 2}}<NS_{D} \}
\end{gather*}
for $NS_{R}, NS_D \notin \mathbb{Z} {2 \pi \over (-2E)^{3 \over 2}}$ where $S_{R}={2\pi \over {(-2E)^{3 \over 2} + 1}}$ and $S_{D}={2\pi \over {(-2E)^{3 \over 2} - 1}}$. We have
\begin{gather*}
\max \{ k \in \mathbb{Z} | k{2 \pi \over (-2E)^{3 \over 2}}<NS_{R} \}=\left\lfloor N{S_{R} \over {{2 \pi \over (-2E)^{3 \over 2}}}}  \right\rfloor =\left\lfloor N{(-2E)^{3 \over 2} \over (-2E)^{3 \over 2}+1}  \right\rfloor,
\\ \max \{ k \in \mathbb{Z} | k{2 \pi \over (-2E)^{3 \over 2}}<NS_{D} \}=\left\lfloor N{S_{D} \over {{2 \pi \over (-2E)^{3 \over 2}}}}  \right\rfloor =\left\lfloor N{(-2E)^{3 \over 2} \over (-2E)^{3 \over 2}-1}  \right\rfloor
\end{gather*}
We denote the $E$ of the retrograde and direct orbit on the energy $-c$ by $E_R(c)$ and $E_D (c)$, respectively. Using the relations $E_R(c)=-{1 \over 2L_{R}(c)^2}$ and $E_D(c)=-{1 \over 2L_{D}(c)^2}$, we define
\begin{gather*}
\alpha_R(c):={(-2E_R (c))^{3 \over 2} \over (-2E_R (c))^{3 \over 2}+1}={1 \over 1+L_R(c)^3}
\\ \alpha_D(c):={(-2E_D (c))^{3 \over 2} \over (-2E_D (c))^{3 \over 2}-1}={1 \over 1+L_D(c)^3}
\end{gather*}
and we have the Conley-Zehnder indices
\begin{gather*}
\mu_{CZ}(\gamma_{R,N}^c)=1+2 \left \lfloor N\alpha_R (c) \right\rfloor,\quad \mu_{CZ}(\gamma_{D,N}^c)=1+2 \left \lfloor N\alpha_D (c) \right\rfloor
\end{gather*}
of the retrograde and direct orbit on energy $-c$. Note that $\alpha_R(c)<1$ and $\alpha_D(c)>1$ go to 1 as $c$ goes to $+\infty$. For any large integer $P \in \mathbb{N}$, we consider a sufficiently large $c$ such that
\begin{gather*}
{P \over P+1} < \alpha_R(c) <1 \iff 0<L_R(c)^3<{1 \over P},
\\ 1<\alpha_D(c)<{P+1 \over P} \iff 0<-L_D(c)^3<{1 \over P+1}
\end{gather*}
equivalently $0<-L_D(c)^3<{1 \over P+1}$.
Then we have
\begin{gather*}
\mu_{CZ}(\gamma_{R, N}^c)=2N-1 \quad \textrm{for} \quad N=1, 2, \cdots, P+1,
\\ \mu_{CZ}(\gamma_{D,N}^c)=2N+1 \quad \textrm{for} \quad N=1, 2, \cdots, P
\end{gather*}
for a such $c$. The last condition implies that there is no $T_{k,l}$ satisfying $k \le P+1$ and particularly $T_{k, l}$ whose action is below the $P$-th iteration of the retrograde and direct orbit does not appear on energy level $-c$. The periodic orbit $\gamma_{R, N}^c$(resp, $\gamma_{D,N}^c$) gives two generators, say $\overline{\gamma_{R, N}^c}, \underline{\gamma_{R, N}^c}$(resp, $\overline{\gamma_{D,N}^c}, \underline{\gamma_{D,N}^c}$), on the chain level by perturbation. The boundary map between these orbits should be $0$-map, because the number of generators in chain level coincides with the dimension of resulting homology in each grade less that $2P+2$. Moreover, we have two generators for constant orbit for a suitable Morse function inside $\Sigma_R^c$. In sum, we have that
\begin{gather*}
CF_*(K_{M_R^c}^b)=
\begin{cases}
\mathbb{Z}_2 & \textrm{if} \quad *=0, 1
\\ \mathbb{Z}_2 \oplus \mathbb{Z}_2 & \textrm{if} \quad *=2, 3, \cdots, 2P+1
\end{cases}
\end{gather*}
for any sufficiently large $b$. We know that all boundary maps are $0$-maps. Therefore, we have
\begin{gather*}
FH_*(K_{M_R^c}^b)=
\begin{cases}
\mathbb{Z}_2 & \textrm{if} \quad *=0, 1
\\ \mathbb{Z}_2 \oplus \mathbb{Z}_2 & \textrm{if} \quad *=2, 3, \cdots, 2P
\end{cases}
\end{gather*}
for any sufficiently large $b$ and so
\begin{gather*}
SH_*(M_R^c)=
\begin{cases}
\mathbb{Z}_2 & \textrm{if} \quad *=0, 1
\\ \mathbb{Z}_2 \oplus \mathbb{Z}_2 & \textrm{if} \quad *=2, 3, \cdots, 2P.
\end{cases}
\end{gather*}
for every $c$ satisfying $0<-L_D(c)^3<{1 \over P+1}$. Moreover, the representative of each homology class is unique for each homology class with grade less than $2P+1$.
\end{Ex}

\begin{Rem}
In the above computation, there is an important remark that we will use in order to compute the symplectic capacity for $\Sigma_R^c$. If $c$ satisfies
\begin{gather*}
0<-L_D(c)^3<{1 \over P+1}
\end{gather*}
for $P \in \mathbb{N}$, then we know the retrograde and direct orbits, up to $P$-th iterations, are generators of the symplectic homology $SH(M_R^c)$. In particular, if we consider $P=1$, then c has to satisfy $-L_D(c)^3<{1 \over 2}$ and equivalently $c>2^{2 \over 3}$(the birth point of Hekuba orbit $\gamma_{2,1}$). This implies that if $c>2^{2 \over 3}$ then the retrograde and direct orbit are generators of $SH(M_R^c)$. More generally, for any $P \in \mathbb{N}$, if $-L_D (c)^3<{1 \over P+1}$, then $N$-th iterations of the retrograde and direct orbits are generators of $SH_* (M_R^c)$. It is easy to see that
\begin{gather*}
-L_D (c)^3<{1 \over P+1} \iff {1 \over 2}(P+1)^{2 \over 3}+(P+1)^{-{1 \over 3}} < c.
\end{gather*} 
\end{Rem}

We will use this to get representatives of homology classes for $H_* (\Lambda S^2)$ in order to obtain symplectic capacities for the Liouville domain $M_R^c$ enclosed by regularized energy hypersurfaces of the rotating Kepler problem $\Sigma_R^c$.

\subsection{Symplectic capacity in cotangent bundle}

Let $N$ be a closed manifold. The cotangent bundle $(T^* N, d \lambda_{can})$ is an exact symplectic manifold. We define a symplectic capacity for fiberwise star-shaped domain in $T^* N$. Let $M$ be a fiberwise star-shaped domain. Then $(M, \omega=d \lambda_{can}|_M)$ is a Liouville domain as we discussed in Example 3.1.2. We note that $[\omega]|_{\pi_2 (M)}=0$, $c_1(M)|_{\pi_2 (M)} =0$ and the symplectic completion $\hat{M}$ can be regarded as $T^* N$. We have seen that the symplectic homology for $(M, \omega_{can})$ is isomorphic to the homology of $H_* (\Lambda N)$.  We will denote this isomorphism by 
\begin{gather*}
\Psi_M : H_* (\Lambda N) \rightarrow SH_* (M) 
\end{gather*}
for each fiberwise star-shaped domain $M \subset T^* N$. We recall the long exact sequence of symplectic homology for action filtration. For $b \in \mathbb{R} \cup \{+\infty\}$, we have
\begin{gather*}
\cdots \rightarrow SH_* ^{<b}(M) \xrightarrow{i^b_M} SH_* (M) \xrightarrow{j^b_M} SH_* ^{\ge b} \rightarrow SH_{*-1}^{<b}(M) \xrightarrow{i^b_M} \cdots
\end{gather*}

Using this long exact sequence, we assign a constant in the following way.

\begin{Def}
In the above setup, we define
\begin{gather*}
c_N(M, \alpha) := \inf \{b \in \mathbb{R} \cup \{ +\infty \} | \Psi_M (\alpha) \in \textrm{im}(i^b_M) \}
\end{gather*}
for a homology class $0 \ne \alpha \in H_* (\Lambda N)$. This constant $c_N(M, \alpha)$ is called the spectral invariant of $\alpha$ in the symplectic homology of $M$.
\end{Def}

One can see immediately that we have another description of the spectral invariant $c_N$. Let us define a constant
\begin{gather*}
c_N'(M, \alpha) := \sup \{b \in \mathbb{R} \cup \{+\infty \} | j^b_M(\Psi_M (\alpha))\ne 0 \}
\end{gather*}
for a while. For any $\epsilon>0$, there exist $b \in [c_N(M, \alpha), c_N(M, \alpha)+\epsilon)$ and $\sigma \in SH_*^{>b}(M)$ such that $\Psi_M(\alpha)=i_M^b(\sigma)$. Then we have $j_M^b(\Psi_M (\alpha))=j_M^b \circ i_M^b(\sigma)=0$ by exactness. This implies  that $c_N '(M, \alpha) \le b$ and so $c_N'(M, \alpha) \le c_N(M, \alpha)$ because $\epsilon$ is arbitrary. On the other hand, for any $b > c_N'(M, \alpha)$, we have $j_M^b(\Psi_M (\alpha))=0$. Then we have $\Psi_M (\alpha) \in \ker j_M^b=\textrm{im} i_M^b$. This implies that $c_N(M, \alpha) \le b$ and so $c_N(M, \alpha) \le c_N'(M, \alpha)$. This proves $c_N(M, \alpha)=c_N'(M, \alpha)$. Thus we will denote this common value by $c_N(M, \alpha)$. Because we have a constant whenever we have a fiberwise star-shaped domain and a homology class of the free loop space $\Lambda N$, we can think of $c_N$ as a map
\begin{gather*}
c_N: FSD(N) \times H_* (\Lambda N)^{\times} \rightarrow \mathbb{R}
\end{gather*}
where $FSD(N)$ is the set of all fiberwise star-shaped domains on $T^* N$ and $H_* (\Lambda N)^{\times}=H_*(\Lambda N) \backslash \{ 0 \}$. We will prove the following properties of $c_N$.

\begin{ThmA}[Properties of $c_N$]
The map
\begin{gather*}
c_N : FSD(N) \times H_* (\Lambda N)^{\times} \rightarrow \mathbb{R}
\\ \quad \quad \quad \quad (M, \alpha) \mapsto c(M, \alpha)
\end{gather*}
satisfies the following properties.

(1) (Conformality) $c_N (kM, \alpha)=k c_N (M, \alpha)$ for all $k \in \mathbb{R}^+$.

(2) (Monotonicity) $c_N (M_2, \alpha) \ge \kappa_{min}(\Sigma_1, \Sigma_2) c_N (M_1, \alpha)$ for all $M_1, M_2 \in FSD(N)$ where $\Sigma_i=\partial M_i$, $i=1, 2$ and $\kappa_{min}(\Sigma_1, \Sigma_2)=\min_{x \in \Sigma_1} \{\kappa(x)| \kappa(x)x \in \Sigma_2, \kappa(x)>0 \}$.

(3) (Spectrality) $c_N (M, \alpha) \in Spec(\Sigma, \lambda_{can})$ where $\Sigma= \partial M$.
\\ for each $\alpha \in H_* (\Lambda N)^{\times}$.
\end{ThmA}

In Theorem A, $kM$ in (1) denotes the Liouville domain obtained by multiplying $k$ on each fiber of $M$ as a scalar multiplication in each cotangent space. We define $\kappa_{min}(\Sigma_1, \Sigma_2)$ in (2) by
\begin{gather*}
\kappa_{min}(\Sigma_1, \Sigma_2)=\min_{x \in \Sigma_1} \{\kappa(x)| \kappa(x)x \in \Sigma_2, \kappa(x)>0 \}
\end{gather*}
and we define similarly
\begin{gather*}
\kappa_{max}(\Sigma_1, \Sigma_2)=\max_{x \in \Sigma_1} \{\kappa(x)| \kappa(x)x \in \Sigma_2, \kappa(x)>0 \}
\end{gather*}
for all fiberwise star-shaped hypersurfaces $\Sigma_1, \Sigma_2$. Clearly, these numbers are positive. Finally, We denote by $\mathcal{P}(\Sigma, \lambda_{can})$ the set of all Reeb periodic orbits of the contact manifold $(\Sigma, \lambda_{can})$. As we discussed in the previous section, we can think of the Reeb orbit as a generator of symplectic homology. We denote by $Spec(\Sigma, \lambda_{can}) \subset \mathbb{R}$ the set of all nonnegative Reeb periods of the contact manifold $(\Sigma, \lambda_{can})$. The period of a Reeb periodic orbit can be regarded as an action value of the Reeb orbit in symplectic homology.

We will prove Theorem A in this Section. For the proof, we need the following Lemmas. 

\begin{Lem}
Let $M$ be a fiberwise star-shaped domain in $T^* N$. If $b \in \mathbb{R}^+ \backslash Spec(\Sigma, \lambda_{can})$, then we have an isomorphism
\begin{gather*}
SH_*^{<b}(M) \simeq FH_*(K_M^b)
\end{gather*}
between the symplectic homology of action less than $b$ and the Floer homology with Hamiltonian $K_M^b: \hat{M}=T^* N \rightarrow \mathbb{R}$. The Hamiltonian $K_M^b$ is given by 
\begin{gather*}
K_M^b(x)=
\begin{cases}
0 & \textrm{if} \quad x \in M,
\\ b(r-1) & \textrm{ if } \quad x=(r, p) \in [0, +\infty) \times \Sigma
\end{cases}
\end{gather*}
\end{Lem}

\begin{proof}
By definition of the symplectic homology of $M$, we have
\begin{gather*}
SH_*^{<b}(M)=\varinjlim_{H \in Ad(M)} FH_*^{<b}(H).
\end{gather*}
Since the action functional $\mathcal{A}_{K_M^b}$ has no critical value larger than $b$, we have
\begin{gather*}
FH_*(K_M^b) \simeq FH_*^{<b}(K_M^b) \simeq FH_*^{<b}(K_M^c)
\end{gather*}
for all $c \ge b$. Since the set of functions $\{ K_M^c | c \ge b \}$ is cofinal in $Ad(M)$, we have
\begin{eqnarray*}
SH_*^{<b}(M)&=&\varinjlim_{c \ge b} FH_*^{<b}(K_M^c)
\\ &=& FH_*(K_M^b).
\end{eqnarray*}
This proves Lemma 4.2.
\end{proof}

Throughout of this Section, we will assume that $b \in \mathbb{R}^+ \backslash Spec(\Sigma, \lambda_{can})$. Because it is known that $Spec(\Sigma, \lambda_{can})$ is discrete for a generic $\Sigma$. 
\begin{Lem}
The following diagram
\begin{displaymath}
\xymatrix{
 & FH_*(K_M^b) \ar[d]_{\phi_{K_M^b}} \ar[dl]_{\phi_{K_M^b}^{<b}}^{\simeq} 
\\
SH_*^{<b}(M) \ar[r]_{i_M^b} & SH_* (M)
}
\end{displaymath}
commutes where $\phi$ is the canonical inclusion in the direct system from a Floer homology of $M$ to the symplectic homology of $M$.
\end{Lem}

\begin{proof}
For an admissible Hamiltonian $H \in Ad(M)$, we have the commutative diagram
\begin{displaymath}
\xymatrix{
\cdots \ar[r] & FH_*^{<b}(H) \ar[r] \ar[d]_{\phi_{H}^{<b}}  & FH_*(H) \ar[r] \ar[d]_{\phi_H} & FH_*^{ \ge b}(H) \ar[r] \ar[d]_{\phi_{H}^{ \ge b}}  & FH_{*-1}^{<b}(H) \ar[r] \ar[d] & \cdots
\\
\cdots \ar[r] & SH_*^{<b}(M) \ar[r]_{i_M^b} & SH_*(M) \ar[r]_{j_M^b} & SH_*^{\ge b}(M) \ar[r]_{\partial} & SH_{*-1}^{<b}(M) \ar[r] & \cdots
}
\end{displaymath}
We focus on the first square of the above commutative diagram and we replace $H$ by $K_M^b$. Then we have the commutative diagram
\begin{gather*}
\xymatrix{
 FH_*^{<b}(K_M^b) \ar[r]_{\simeq} \ar[d]_{\phi_{K_M^b}^{<b}}  & FH_*(K_M^b) \ \ar[d]_{\phi_{K_M^b}} 
\\
 SH_*^{<b}(M) \ar[r]_{i_M^b} & SH_*(M) 
}
\end{gather*}
with an isomorphism on the upper and right sides by Lemma 4.2. If we identify two Floer homology groups in the first row, then we get the desired commutative diagram. This proves Lemma 4.3.
\end{proof}

\begin{Rem}
By virtue of Lemma 4.2 and 4.3, we can identify the induced map $SH_*^{<b}(M) \xrightarrow{i_M^b}  SH^*(M)$ on the symplectic homology with the canonical map $FH_* (K_M^b) \xrightarrow{\phi_{K_M^b}} SH_*(M)$ of the direct system.
\end{Rem}

First, we will prove (3) of Theorem A. This will be done by proving the following Lemma.

\begin{Lem}
For each $M \in FSD(N)$ and $\alpha \in H_* (\Lambda N)$,
\begin{gather*}
c_N(M, \alpha)=\min_{\sum_{x \in \mathcal{P}(\Sigma, \lambda)} c_x x \in \Psi_M(\alpha)} \max \left\{ \int_{S^1} x^* \lambda_{can} | c_x \ne 0 \right\}
\end{gather*}
\end{Lem}
\begin{proof}
Let us denote that
\begin{gather*}
\bar{c}(M, \alpha)=\min_{\sum_{x \in \mathcal{P}(\Sigma, \lambda)} c_x x \in \Psi_M(\alpha)} \max \left\{ \int_{S^1} x^* \lambda_{can} | c_x \ne 0 \right\}
\end{gather*}
for a moment. We want to show that $c_N(\Sigma, \alpha)=\bar{c}(\Sigma, \alpha)$. Let $\sigma=\sum_{x \in \mathcal{P}(\Sigma, \lambda)} c_x x \in \Psi_M(\alpha)$ be a representative of symplectic homology such that $\max \left\{ \int_{S^1} x^* \lambda_{can} | c_x \ne 0 \right\}=\bar{c}(\Sigma, \alpha)=:\bar{c}$. For any $\epsilon>0$, if we take $b=\bar{c}+ \epsilon$, then $[\sigma] \in FH_* (K_M^b)$ since every generator of action below $b$ in chain complex and $\partial \sigma=0$ as well. By the choice of $\sigma$, we have $i_M^b(\sigma)=\Psi_M(\alpha)$ and this implies $c_N(M, \alpha) \le b$ and so $c_N(M, \alpha) \le \bar{c}$ because $\epsilon$ is arbitrary.

Conversely, we suppose that $b< \bar{c}$ and $\Psi_M(\alpha) \in \textrm{im}(i_M^b)$. Then there exists $\sigma \in FH_*(K_M^b)$ such that $i_M^b (\sigma)=\Psi_M (\alpha)$. Since $i_M^b(\sigma)$ consists of the Reeb orbits whose periods are less than or equal to $b$. This implies $\bar{c} \le b$. This contradicts the assumption. Therefore, the inequality $c_N(M, \alpha) \ge \bar{c}$. This completes the proof of Lemma 4.4.
\end{proof}

\begin{Rem}
One can regard $c_N(M, \alpha)$ as the spectral invariant corresponding to $\alpha$ for the Floer homology of Hamiltonian $K_M^b: T^* N \rightarrow \mathbb{R}$ of sufficiently large $b$. In that reason, we call $c_N(M, \alpha)$ by the spectral invariant of $\alpha$ in the symplectic homology of $M$.
\end{Rem}

We consider two fiberwise star-shaped domains $M_1, M_2 \in FSD(N)$ in $T^*N$. We have defined
\begin{gather*}
\kappa_{min}(\Sigma_1, \Sigma_2)=\min_{x \in \Sigma_1} \{\kappa(x)| \kappa(x)x \in \Sigma_2, \kappa(x)>0 \}
\end{gather*}
for $\Sigma_i=\partial M_i$ ($i= 1,2$) and If we assume that $\kappa_{min}(\Sigma_1, \Sigma_2) \ge 1$, that is $M_1 \subset M_2$ and abbreviate $\kappa_{min}:=\kappa_{min}(\Sigma_1, \Sigma_2)$. Then one can easily see that
\begin{gather*}
K_{M_2}^{b \kappa_{min}}(x) \le K_{M_1}^b (x)
\end{gather*}
for all $x \in T^* N=\hat{M_1}=\hat{M_2}$. Then we have the monotone homomorphism
\begin{gather*}
\phi_{(K_{M_1}^b, K_{M_2}^{b \kappa_{min}})} : FH_* (K_{M_2}^{b \kappa_{min}}) \rightarrow FH_*(K_{M_1}^b)
\end{gather*}
between Floer homologies. Using this morphism, we have the following Lemma.

\begin{Lem}
Let $M_1, M_2$ be fiberwise star-shaped domain in $T^* N$. Suppose that $M_1 \subset M_2$. Then the following diagram commutes.
\begin{gather*}
\xymatrix{
 FH_*(K_{M_2}^{b \kappa_{min}}) \ar[rr]_{\phi_{K_{M_2}^{b \kappa_{min}}}} \ar[dd]_{\phi_{(K_{M_1}^b, K_{M_2}^{b \kappa_{min}})}} & & SH_*(M_2) &  
\\
 & & & & \ar[ull]_{\Psi_{M_2}} \ar[dll]_{\Psi_{M_1}} H_* (\Lambda N)  
\\
 FH_*(K_{M_1}^b) \ar[rr]_{\phi_{K_{M_1}^b}} & & SH_*(M_1) &
}
\end{gather*}
\end{Lem}
\begin{proof}
We recall the isomorphism between $\Psi_M : H_*(\Lambda N) \rightarrow SH_*(M)$ in \cite{AS}. They constructed the isomorphism
\begin{gather*}
\Theta^{AS}_H: FH_*(H) \rightarrow HM_*(\mathcal{E}_L)
\end{gather*}
between a Floer homology of a quadratic Hamiltonian $H \in Qd(T^*N)$ and a Morse homology of Lagrangian action functional for $L=\mathcal{L}(H)$, Legendre transformation of $H$. After this construction, $\Psi_M$ can be obtained by identifying $SH_*(M)$ with $FH_*(\mathcal{A}_H)$ and $HM_*(\mathcal{E}_L)$ with $H_*(\Lambda N)$. We can take a quadratic Hamiltonian $H \in Qd(T^* N)$ on $T^* N$ such that $H \ge K_{M_2}^{b \lambda_{min}}$ and $H \ge K_{M_1}^b$. For example, we fix a metric $g$ on $N$ and take sufficiently large $s$ such that $H(q, p)=s|p|_g^2$ satisfies the above inequalities. Then we have
\begin{eqnarray*}
\phi_{(H, K_{M_1}^b)} \circ \phi_{(K_{M_1}^b, K_{M_2}^{b \kappa_{min}})}=\phi_{(H, K_{M_2}^{b \kappa_{min}})}
\end{eqnarray*}
by the naturality of monotone homomorphism. Thus we have 
\begin{eqnarray*}
\Theta^{AS}_{H} \circ \phi_{(H, K_{M_1}^b)} \circ \phi_{(K_{M_1}^b, K_{M_2}^{b \kappa_{min}})}=\Theta^{AS}_{H} \circ \phi_{(H, K_{M_2}^{b \kappa_{min}})}
\end{eqnarray*}
This implies that
\begin{gather*}
\Psi_{M_1}^{-1} \circ \phi_{K_{M_1}^b} \circ \phi_{(K_{M_1}^b, K_{M_2}^{b \kappa_{min}})}=\Psi_{M_2}^{-1} \circ \phi_{K_{M_2}^{b \kappa_{min}}}
\end{gather*}
from the following commutative diagram
\begin{gather*}
\xymatrix{
 SH_*(M) \ar[d]_{\cong} & H_*(\Lambda N) \ar[l]_{\Psi_M} \ar[d]_{\cong}
\\
 FH_*(H) \ar[r]_{\Theta^{AS}_H} & HM_*(\mathcal{E}_L) 
}
\end{gather*}
This proves Lemma 4.5.
\end{proof}

Lemma 4.5 implies the following crucial fact.
\begin{eqnarray*}
& & \Psi_{M_2}(\alpha) \in \textrm{im}(\phi_{K_{M_2}^{b \kappa_{min}}})
\\ &\iff& \Psi_{M_2}(\alpha) \in \textrm{im}(\Psi_{M_2} \circ \Psi_{M_1}^{-1} \circ \phi_{K_{M_1}^b} \circ \phi_{(K_{M_1}^b, K_{M_2}^{b \kappa_{min}})})
\\ &\iff& \alpha \in \textrm{im}(\Psi_{M_1}^{-1} \circ \phi_{K_{M_1}^b} \circ \phi_{(K_{M_1}^b, K_{M_2}^{b \kappa_{min}})})
\\ &\iff& \Psi_{M_1}(\alpha) \in \textrm{im}(\phi_{K_{M_1}^b} \circ \phi_{(K_{M_1}^b, K_{M_2}^{b \kappa_{min}})})
\\ &\Longrightarrow&  \Psi_{M_1}(\alpha) \in \textrm{im}(\phi_{K_{M_1}^b})
\end{eqnarray*}
for any $\alpha \in H_*(\Lambda N)^{\times}$ and $b \in \mathbb{R}$. In sum, we have
\begin{eqnarray*}
\Psi_{M_2}(\alpha) \in \textrm{im}(\phi_{K_{M_2}^{b \kappa_{min}}}) \Longrightarrow  \Psi_{M_1}(\alpha) \in \textrm{im}(\phi_{K_{M_1}^b})
\end{eqnarray*}
for any $\alpha \in H_*(\Lambda N)^{\times}$, $b \in \mathbb{R}$ and $M_1, M_2 \in FSD(N)$. Therefore, we have proved the following Lemma.

\begin{Lem}
Let $M_1, M_2$ be fiberwise star-shaped domains in $T^* N$. Suppose that $\kappa_{min}(\Sigma_1, \Sigma_2) \ge 1$ for $\Sigma_i=\partial M_i$($i=1, 2$). Then the inequality
\begin{gather*}
c_N(M_2, \alpha) \ge \kappa_{min}(\Sigma_1, \Sigma_2) c_N(M_1, \alpha)
\end{gather*}
holds for each $\alpha \in H_* (\Lambda N)$. In particular, the above inequality implies
\begin{gather*}
c_N(M_2, \alpha) \ge c_N(M_1, \alpha)
\end{gather*}
provided $\kappa_{min}(\Sigma_1, \Sigma_2) \ge 1$, that is $M_1 \subset M_2$.
\end{Lem}

We want to extend the above Lemma for any $\kappa_{min}(\Sigma_1, \Sigma_2)  \in \mathbb{R}^+$. We need the following Lemma in order to define a contactomorphism between two fiberwise star-shaped hypersurfaces.

\begin{Lem}
Let $M_1, M_2$ be fiberwise star-shaped domains in $T^* N$. For $\Sigma_1=\partial M_1$ and $\Sigma_2=\partial M_2$, we define a function
\begin{gather*}
f_{\Sigma_1}^{\Sigma_2} : \Sigma_1 \rightarrow \mathbb{R}^+ \quad \textrm{ by } \quad f_{\Sigma_1}^{\Sigma_2}(x)\cdot x \in \Sigma_2
\end{gather*}
on $\Sigma_1$. In local coordinates $x=(q, p)$, $f_{\Sigma_1}^{\Sigma_2}(x)\cdot x=f_{\Sigma_1}^{\Sigma_2}(x)\cdot (q, p)=(q, f_{\Sigma_1}^{\Sigma_2}(x)p)$ is the scalar multiplication on the cotangent space. We define a map
\begin{gather*}
F_{\Sigma_1}^{\Sigma_2} : \Sigma_1 \rightarrow \Sigma_2 \quad \textrm{ by } \quad x \mapsto f_{\Sigma_1}^{\Sigma_2}(x) \cdot x
\end{gather*}
Then the map $F_{\Sigma_1}^{\Sigma_2}$ is a contactomorphism between $(\Sigma_1, \xi_{can})$ and $(\Sigma_2, \xi_{can})$ where $\xi_{can}=\ker \lambda_{can}$. More precisely, one can compute the pull-back of the Liouville 1-form $\lambda_{can}$ as follows.
\begin{gather*}
(F_{\Sigma_1}^{\Sigma_2})^* \lambda_{can} = f_{\Sigma_1}^{\Sigma_2} \cdot \lambda_{can}
\end{gather*}

\end{Lem}
\begin{proof}
It suffices to prove the last statement. We recall the canonical 1-form $\lambda_{can}=p dq$ in the local coordinate $x=(q, p)$. We will directly compute the evaluation of the pull-back form $(F_{\Sigma_1}^{\Sigma_2})^* \lambda_{can}(x)$ for an arbitrary tangent vector $h \in T_x T^* N$ for $x \in \Sigma_1$. Assume that $h=h_q+h_p$ where $h_q \in <{\partial \over \partial q_1}, {\partial \over \partial q_2}, \cdots, {\partial \over \partial q_n}>$ and $h_p \in <{\partial \over \partial p_1}, {\partial \over \partial p_2}, \cdots, {\partial \over \partial p_n}>$. For a notational convenience, we denote $F:=F_{\Sigma_1}^{\Sigma_2}$ and $f:=f_{\Sigma_1}^{\Sigma_2}$ in this proof.
\begin{eqnarray*}
F^* \lambda_{can}(x)(h)&=& \lambda_{can}(F(x))(DF(h))
\\ &=& \lambda_{can}((q, f(x)p))(h_q+f(x)h_p+df(x)(h)\cdot p{\partial \over \partial p})
\\ &=& \lambda_{can}((q, f(x)p))(h_q)
\\ &=& \lambda_{can}((q, f(x)p))(h)
\\ &=& f(x)\lambda_{can}((q, p))(h)
\end{eqnarray*}
Therefore, we have $(F_{\Sigma_1}^{\Sigma_2})^* \lambda_{can} = f_{\Sigma_1}^{\Sigma_2} \cdot \lambda_{can}$ and this proves Lemma 4.7.
\end{proof}

\begin{Rem}
Lemma 4.7 implies that the map
\begin{gather*}
F_{\Sigma_1}^{\Sigma_2} : (\Sigma_1, f_{\Sigma_1}^{\Sigma_2} \cdot \lambda_{can}) \rightarrow (\Sigma_2, \lambda_{can})
\end{gather*}
is a strict contactomorphism for all pair of fiberwise star-shaped hypersurfaces $\Sigma_1, \Sigma_2$. In particular, if $\Sigma_2=k \Sigma_1$ for some $k>0$, then we have a strict contactomorphism
\begin{gather*}
F : (\Sigma_1, k \lambda_{can}) \rightarrow (\Sigma_2, \lambda_{can})
\end{gather*}
and this extends to a symplectomorphism between two Liouville domains $(M_1, k\omega_{can}), (M_2, \omega_{can})$ enclosed by $\Sigma_1, \Sigma_2$, respectively. This implies the conformality of $c_N$ as follows.
\begin{gather*}
c_N(k M_1, \alpha)=c_N(M_2, \alpha)=kc_N(M_1, \alpha)
\end{gather*}
This proves (1) of Theorem A.
\end{Rem}

We can prove (2) of Theorem A by combining Lemma 4.6 and (1) of Theorem A. Let $M_1, M_2$ be fiberwise star-shaped domains in $T^* N$. We denote $k=\kappa_{min}(\Sigma_1, \Sigma_2)$ where $\Sigma_i=\partial M_i$(i=1,2). If $k \ge 1$, then we already have that $c_N(M_2, \alpha) \ge \kappa_{min}(\Sigma_1, \Sigma_2)c_N(M_1, \alpha)$ from Lemma 4.6. Suppose that $0<k<1$. If we consider $k M_1$ instead of $M_1$, then $\kappa_{min}(k \Sigma_1, \Sigma_2)=1$. Hence we can apply Lemma 4.6 for the pair $(k M_1, M_2)$ and so we have
\begin{gather*}
c_N(M_2, \alpha) \ge c_N(k M_1, \alpha)=kc_N(M_1, \alpha)=\kappa_{min}(\Sigma_1, \Sigma_2) c_N(M_1, \alpha)
\end{gather*}
using (1) of Theorem A. This proves (2) of Theorem A. This completes the proof of Theorem A.

Until now, we have proved Theorem A. Therefore, as we mentioned in the introduction, the spectral invariant $c_N(\cdot, \alpha)$ of $\alpha$ can play the role of symplectic capacity for $FSD(N)$ provided $c_N(\cdot, \alpha) \ne 0$. Moreover, by Spectrality of Theorem A, the spectral invariant $c_N(M, \alpha)$ of $\alpha$ in the symplectic homology of $M$ should be one of the Reeb period. We will use the spectral invariant $c_{S^2}$ to obtain estimates action values of Hill's lunar problem in the next two sections.

\section{Embedding of Hill's lunar problem and the rotating Kepler problem}

In this section, we will prove Theorem B. We recall the contact structures of the regularized rotating Kepler problem and Hill's lunar problem in Section 2.2.
Let $H_R : T^*(\mathbb{R}^2-\{(0,0)\}) \rightarrow \mathbb{R}$ be the Hamiltonian of the rotating Kepler problem
\begin{eqnarray*}
H_R (q,p)={1 \over 2}|p|^2-{1 \over |q|}+p_1q_2-p_2q_1
\end{eqnarray*}
and let $H_H : T^*(\mathbb{R}^2-\{(0,0)\}) \rightarrow \mathbb{R}$ be the Hamiltonian of Hill's lunar problem
\begin{eqnarray*}
H_H (q,p)={1 \over 2}|p|^2-{1 \over |q|}+p_1q_2-p_2q_1-q_1^2+{1 \over 2}q_2^2.
\end{eqnarray*}
In \cite{CFvK}, they proved fiberwise convexity of the rotating Kepler problem below the critical energy level $-c_R^0=-{3 \over 2}$ and so we can think of the energy hypersurfaces $H_R ^{-1}(-c)$ of an energy $-c<-c_R^0$ as a hypersurface $\Sigma_R^c$ in $T^* S^2$ after switching the roles of position($q$) and momentum($p$).
The bounded component of the closure $\Sigma_R^c=\overline{\Phi \circ \Psi (H_R^{-1}(-c))}^b$ of the energy hypersurfaces are fiberwise convex in $T^* S^2$ for all $c>c_R^0$. Let $M_R^c$ be the bounded region in $T^* S^2$ such that its boundary $\partial M_R^c$ is the hypersurface $\Sigma_R^c$. Then $M_R^c$ is a fiberwise convex Liouville domain with the restriction of the canonical symplectic structure $\omega_{can}=d \lambda_{can}$ and the canonical Liouville vector field $Y_{can}$ of $T^* S^2$. In the local coordinates, the canonical 1-form $\lambda_{can}$ and the Liouville vector field $Y_{can}$ can be written as
\begin{eqnarray*}
\lambda_{can}=pdq, \quad Y_{can}=p{\partial \over \partial p}
\end{eqnarray*}
where $q$ are the coordinates for base manifold $S^2$ and $p$ are the dual coordinates for cotangent spaces.
Similarly, we can define the Liouville domain $M_H^c$ determined by the regularized energy hypersurface $\Sigma_H^c$ of Hill's lunar problem for each $c>c_H^0={3^{4 \over 3} \over 2}$.
We will discuss their inclusions among these Liouville domains. One can easily see the inclusions between different energy hypersurfaces of the same problem. For $c_1>c_2>c_R^0$ and $c_1 '>c_2 '>c_H^0$, we have
\begin{eqnarray*}
M_R^{c_1} \subset M_R^{c_2} \quad \textrm{ and } \quad M_H^{c_1 '} \subset M_H^{c_2 '}
\end{eqnarray*}
Namely, the energy hypersurface is getting smaller as the energy goes down. Now we want to know the inclusions between $M_R^c$ and $M_H^{c'}$. We want to investigate some data of both problems.

Let the map $\pi : T^* (\mathbb{R}^2-(0,0)) \rightarrow (\mathbb{R}^2-(0,0))$, $\pi(q, p)=q$ be the obvious projection onto the $q$-coordinates. For $c>c_R^0$, we define by 
\begin{gather*}
\mathfrak{R}_R^c:=\overline{\bigcup_{d>c}\pi(H_R^{-1}(-d))^b}
\end{gather*}
Hill's region of the rotating Kepler problem of energy $-c$. Here, superscript $b$ means the bounded component. Moreover, we define by
\begin{gather*}
\mathfrak{R}_R:=\overline{\bigcup_{c>c_R^0} \mathfrak{R}_R^c}
\end{gather*}
Hill's region of the rotating Kepler problem. We define the Hill's regions of Hill's lunar problem by $\mathfrak{R}_H^{c}$ and $\mathfrak{R}_H$ similarly.

\begin{Lem}
For $c>c_R^0$ and $c'>c_H^0$, the Hill's regions are given by
\begin{gather*}
\mathfrak{R}_R^c=\{(q_1, q_2) \in \mathbb{R}^2 | {1 \over |q|}+{1 \over 2}|q|^2<c, |q|<1\},
\\ \mathfrak{R}_H^c=\{(q_1, q_2) \in \mathbb{R}^2 | {1 \over |q|}+{3 \over 2}q_1^2<c', |q_1|<3^{-1 \over 3}, |q_2|<2\cdot3^{-4 \over 3}\}.
\end{gather*}
\end{Lem}
\begin{proof}
See \cite{L}.
\end{proof}

The goal of this paper is obtaing the estimates for symplectic capacities of $M_H^c$ in $T^*S^2$ using symplectic capacities of $M_R^c$ in $T^* S^2$ and inclusions. Thus it is important to show inclusion relations between $M_H^{c'}$ and $M_R^c$. We construct the following Proposition in order to check easily.

\begin{figure}
\centering
\includegraphics[]{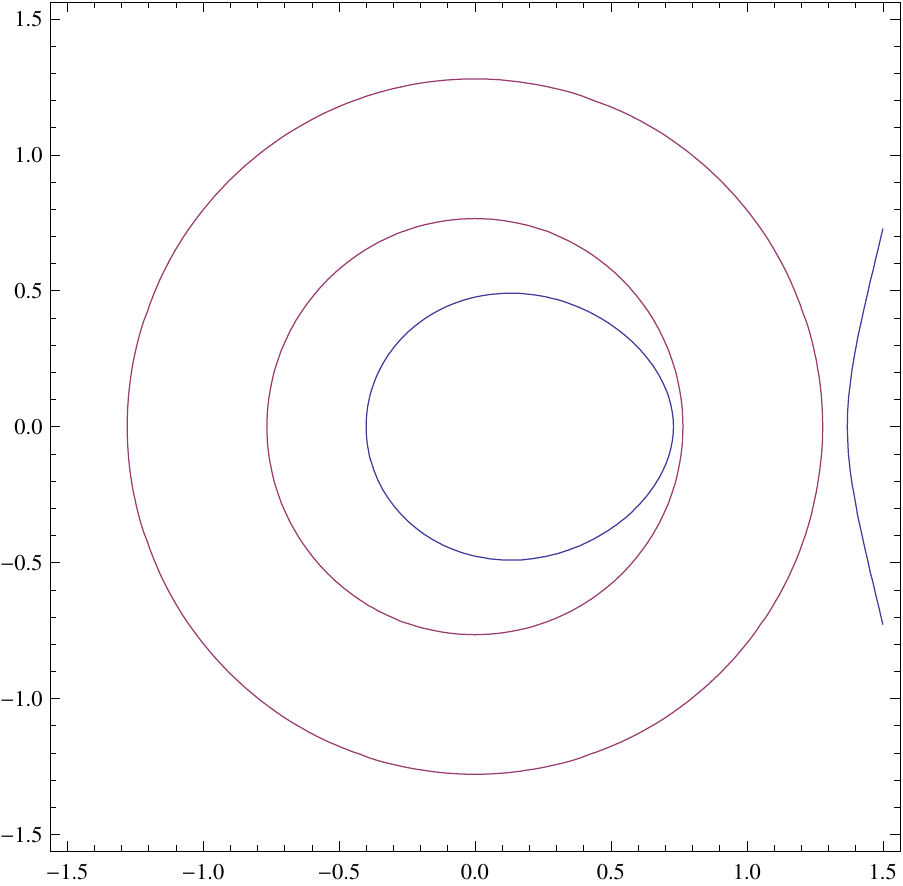}
\caption{$\mathfrak{R}_R^{1.6}$ and $\sigma_{R, (0,1)}^{1.6}$}
\end{figure}

\begin{figure}
\centering
\includegraphics[]{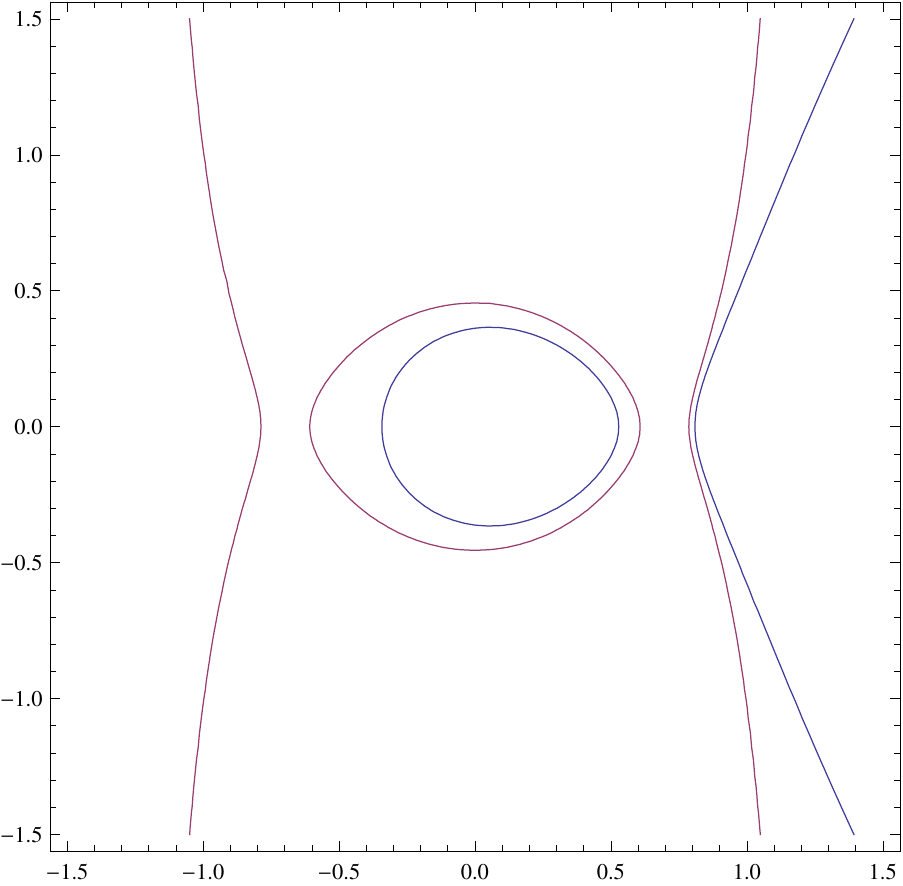}
\caption{$\mathfrak{R}_H^{2.2}$ and $\sigma_{H, (0,1)}^{2.2}$}
\end{figure}

\begin{prop}
We have the following criteria for inclusions.

(1) $M_H^{c'} \subset M_R^c$ if and only if $H_R(q, p)+c \le 0$ for all $(q, p) \in H_H^{-1}(-c')$ with $q \in \mathfrak{R}_H^{c'}$.

(2) $M_H^{c'} \subset M_R^c$ if $H_H(q, p)+c' \ge 0$ for all $(q, p) \in H_R^{-1}(-c)$ with $q \in \mathfrak{R}_R^{c}$.

(3) $M_R^c \subset M_H^{c'}$ if and only if $H_H(q, p)+c' \le 0$ for all $(q, p) \in H_R^{-1}(-c')$ with $q \in \mathfrak{R}_R^{c}$

(4) $M_R^{c} \subset M_R^{c'}$ if $H_R(q, p)+c \ge 0$ for all $(q, p) \in H_H^{-1}(-c')$ with $q \in \mathfrak{R}_H^{c'}$.
\\ for every $c>c_R^0$ and $c'>c_H^0$.
\end{prop}
\begin{proof}
For a fixed $p \in \mathbb{R}^2$, we define the function $H_{R, p} :(\mathbb{R}^2-(0,0)) \rightarrow \mathbb{R}$ by $H_{R,p}(q):=H_R(q, p)$. Then for any $c>c_R^0$, the curve $H_{R, p}^{-1}(-c)$ has one bounded component. We will denote this bounded component by $\sigma_{R, p}^c$. Since we know that the rotating Kepler problem is fiberwise convex, the closed curve $\sigma_{R, p}^c$ bounds a strictly convex domain, say $D_{R, p}^c$, containing the origin and $\sigma_{R, p}^c \subset \mathfrak{R}_R^c$ for all $p$. Following symplectomorphisms, $\Phi \circ \Psi(\sigma_{R, p}^c)$ becomes a fiber of $\Sigma_R^c$ at $p$ and thus $\Phi \circ \Psi(\sigma_{R, p}^c) \subset T^*_{\phi(p)} S^2$. We can define the fiber $\Phi \circ \Psi(\sigma_{H, p}^{c'})$ of $\Sigma_H^{c'}$  and the strictrly convex domain $D_{H, p}^{c'}$ enclosed by $\sigma_{H, p}^{c'}$ analogously.

Proof of (1): The inclusion $M_H^{c'} \subset M_R^c$ holds if and only if the fiber $\Phi \circ \Psi(\sigma_{H, p}^{c'})$ of $\Sigma_H^{c'}$ is contained inside the fiber $\Phi \circ \Psi(\sigma_{R, p}^c)$ of $\Sigma_R^c$ for every $p \in \mathbb{R}^2$. Because inclusion relation is preserved by $\Phi \circ \Psi$, we have 
\begin{gather*}
M_H^{c'} \subset M_R^c \iff D_{H, p}^{c'} \subset D_{R, p}^c \iff \sigma_{H, p}^{c'} \subset D_{R, p}^c
\end{gather*}
for every $p \in \mathbb{R}^2$. Assume we have $\sigma_{H, p}^{c'} \subset D_{R, p}^c$, then the following holds 
\begin{gather*}
q \in \sigma_{H, p}^{c'} \implies H_{R, p}(q) \le -c \implies H_R(q, p)+c \le 0
\end{gather*}
for every $p \in \mathbb{R}^2$, because $H_{R, p}$ is less than $-c$ on $D_{R, p}^c$. Since we know that $\sigma_{H, p}^{c'}$ is a closed curve including the origin on its inside, if the inequality $H_R(q, p)+c \le 0$ holds for every $q \in \sigma_{H, p}^{c'}$, then $q$ has to on $\mathfrak{R}_R^c$. Therefore, the converse is also true. This proves (1).

Proof of (2): By similar argument in the Proof of (1), we have
\begin{eqnarray*}
& &M_H^{c'} \subset M_R^c
\\ &\iff& \sigma_{R,p}^c \subset \mathbb{R}^2 \backslash D_{H, p}^{c'} \textrm{ for all } p
\\ &\Longleftarrow& H_{H, p}(q) \ge -c' \textrm{ for every } q \in \sigma_{R,p}^c  \textrm{ for all } p
\\ &\iff& H_H (q, p)+c' \ge 0 \textrm{ for all } (q, p) \in H_R^{-1}(-c) \textrm{ such that } q \in \mathfrak{R}_R^c
\end{eqnarray*}
This proves (2).

(3) and (4) can be proved analogously. This proves Proposition 5.2.

\end{proof}

Using the above Proposition, we will prove the inclusions in Theorem B. First, we observe the following Theorem.

\begin{Thm}
Every energy hypersurface of the regularized Hill's lunar problem below the critical value can be embedded in $M_R^{2^{2 \over 3}}$.
\end{Thm}
\begin{proof}
It is enough to show that $\Sigma_H^{c_H^0} \subset M_R^{2^{2 \over 3}}$. Assume that $\Phi \circ \Psi (\bar{q}, \bar{p}) \in \Sigma_H^{c_H^0}$ and so
\begin{eqnarray*}
{1 \over 2}|\bar{p}|^2-{1 \over |\bar{q}|}+\bar{p_1}\bar{q_2}-\bar{p_2}\bar{q_1}-\bar{q_1}^2+{1 \over 2}\bar{q_2}^2+c_H^0=0,  \quad \bar{q} \in \mathfrak{R}_H.
\end{eqnarray*}
We compute the value of $H_R^{2^{2 \over 3}}:=H_R+2^{2 \over 3}$. Then we have
\begin{eqnarray*}
H_R^{2^{2 \over 3}}(\bar{q}, \bar{p})&=&{1 \over 2}|\bar{p}|^2-{1 \over |\bar{q}|}+\bar{p_1}\bar{q_2}-\bar{p_2}\bar{q_1}+2^{2 \over 3}
\\ &=& \bar{q_1}^2-{1 \over 2}\bar{q_2}^2-c_H^0+2^{3 \over 2}
\\ &\le& \bar{q_1}^2+2^{2 \over 3}-c_H^0
\\ &\le& 3^{-{2 \over 3}}+2^{2 \over 3}-{3^{4 \over 3} \over 2}<0.
\end{eqnarray*}
The last $\le$ holds because $(\bar{q}, \bar{p}) \in \mathfrak{R}_H$. Above inequality implies that $(\bar{q}, \bar{p}) \in M_R^{2^{2 \over 3}}$ by Proposition 5.2. This completes the proof of Theorem.
\end{proof}
We note that the energy level $-2^{2 \over 3}$ is the bifurcating point of Hekuba orbit. This is important because it is hard to say about generators of symplectic homology when the energy level is between the critical value $c_R^0$ and $-2^{2 \over 3}$ as we discussed in Example 4.1.7 and Remark 4.1. We will discuss this more precisely in the next Section. From the computation in the proof of Theorem 5.2, one can immediately see the following Corollary.

\begin{Cor}
We have the embedding
\begin{gather*}
M_H^c \subset M_R^{c-3^{-{2 \over 3}}}
\end{gather*}
for any $c>c_H^0$.
\end{Cor}

As we observed in Example 4.1.7 and Remark 4.1, when the condition
\begin{gather*}
0<-L_D (c)^3<{1 \over P+1}
\end{gather*}
holds, the iterations of the retrograde and direct orbits represent generators in $SH_*(M_R^c)$ up to $P$-th iteration. We note that
\begin{gather*}
0<-L_D (c)^3<{1 \over P+1} \iff {P+3 \over 2(P+1)^{1 \over 3}}={1 \over 2}(P+1)^{2 \over 3}+(P+1)^{-{1 \over 3}}<c
\end{gather*}
and we define 
\begin{gather*}
c_R^P:={P+3 \over 2(P+1)^{1 \over 3}}
\end{gather*}
for each $P \in \mathbb{N}$. Theorem 5.3 tells us that $M_H^c<M_R^{c_R^1}$ for all $c>c_H^0$, this is (1) of Theorem B. Moreover, we can use the homology classes of the retrograde and direct orbits. We define
\begin{gather*}
c_H^P:={2P+8-\sqrt{(P+1)(P+9)} \over 2(P+1)^{1 \over 3}}
\end{gather*}
for each $P \in \mathbb{N}_{\ge 2}$. Then we have the following Theorem.

\begin{Thm}
For the constants $c_R^P$ and $c_H^P$ defined above, we have the following inclusion
\begin{gather*}
M_H^{c_H^P} \subset M_R^{c_R^P}
\end{gather*}
for each $P\in \{2, 3, 4, \cdots \}$.
\end{Thm}
\begin{proof}
It is enough to show that $\Sigma_H^{c_H^P} \subset M_R^{c_R^P}$. Assume that $\Phi \circ \Psi (\bar{q}, \bar{p}) \in \Sigma_H^{c_H^P}$. That is,
\begin{gather*}
{1 \over 2}|\bar{p}|^2-{1 \over |\bar{q}|}+\bar{p_1}\bar{q_2}-\bar{p_2}\bar{q_1}-\bar{q_1}^2+{1 \over 2}\bar{q_2}^2+c_H^P=0,  \quad \bar{q}  \in \mathfrak{R}_H^{c_H^P}.
\end{gather*}
We insert $(\bar{q}, \bar{p})$ in $H_R^{c_R^P}:=H_R+c_R^P$, then we have
\begin{eqnarray*}
H_R^{c_R^P}(\bar{q}, \bar{p})&=&{1 \over 2}|\bar{p}|^2-{1 \over |\bar{q}|}+\bar{p_1}\bar{q_2}-\bar{p_2}\bar{q_1}+c_R^P
\\ &=& \bar{q_1}^2-{1 \over 2}\bar{q_2}^2-c_H^P+c_R^P
\\ &\le& \bar{q_1}^2+c_R^P-c_H^P
\end{eqnarray*}
and we want to prove the last term less that or equal to $0$. It suffices to prove the following Claim.
\\Claim: If $\bar{q} \in \mathfrak{R}_H^{c_H^P}$, then $\bar{q_1}^2 \le c_H^P-c_R^P$ for any $P \in \{2, 3, \cdots \}$.
\begin{proof}[Proof of Claim]
For $\bar{q} \in \mathfrak{R}_H^{c_H^P}$, $|\bar{q_1}|$ attains its maximum, say $\bar{q_1}^M$, when $\bar{q_2}=0$. It suffices to prove that
\begin{gather*}
(\bar{q_1}^M)^2 \le c_H^P-c_R^P
\end{gather*}
On the other hand, we know that $\bar{q_1}^M$ is the smaller positive zero of the equation
\begin{gather*}
{3 \over 2}x^2+{1 \over x}=c_H^P={2P+8-\sqrt{(P+1)(P+9)} \over 2(P+1)^{1 \over 3}}
\end{gather*}
by Lemma 5.1. We solve the above equation and obtain
\begin{gather*}
\bar{q_1}^M={\sqrt{P+9}-\sqrt{P+1} \over 2 (P+1)^{1 \over 6}}
\end{gather*}
and so in fact we get
\begin{gather*}
(\bar{q_1}^M)^2=c_H^P-c_R^P
\end{gather*}
This proves the Claim.
\end{proof}
Claim implies that
\begin{gather*}
H_R^{c_R^P}(\bar{q}, \bar{p}) \le 0
\end{gather*}
By Proposition 5.2, this proves Theorem 5.5.
\end{proof}

We have proved (2) of Theorem B. Since $M_H^c$ shrinks as $c$ increases, using Theorem 5.3 and Theorem 5.5, we formulate the inclusion for any $c>c_H^0$ in the following Corollary.

\begin{Cor}
For any $c> c_H^0$, we have the following inclusions
\begin{gather*}
\begin{cases}
M_H^c \subset M_R^{c_R^1} & \textrm{if} \quad c \in (c_H^0, c_H^2),
\\ M_H^c \subset M_R^{c_R^P} & \textrm{if} \quad c \in [c_H^P, c_H^{P+1}) \textrm{ for } P =2, 3, 4, \cdots
\end{cases}
\end{gather*}
\end{Cor}

\begin{figure}
\centering
\includegraphics[]{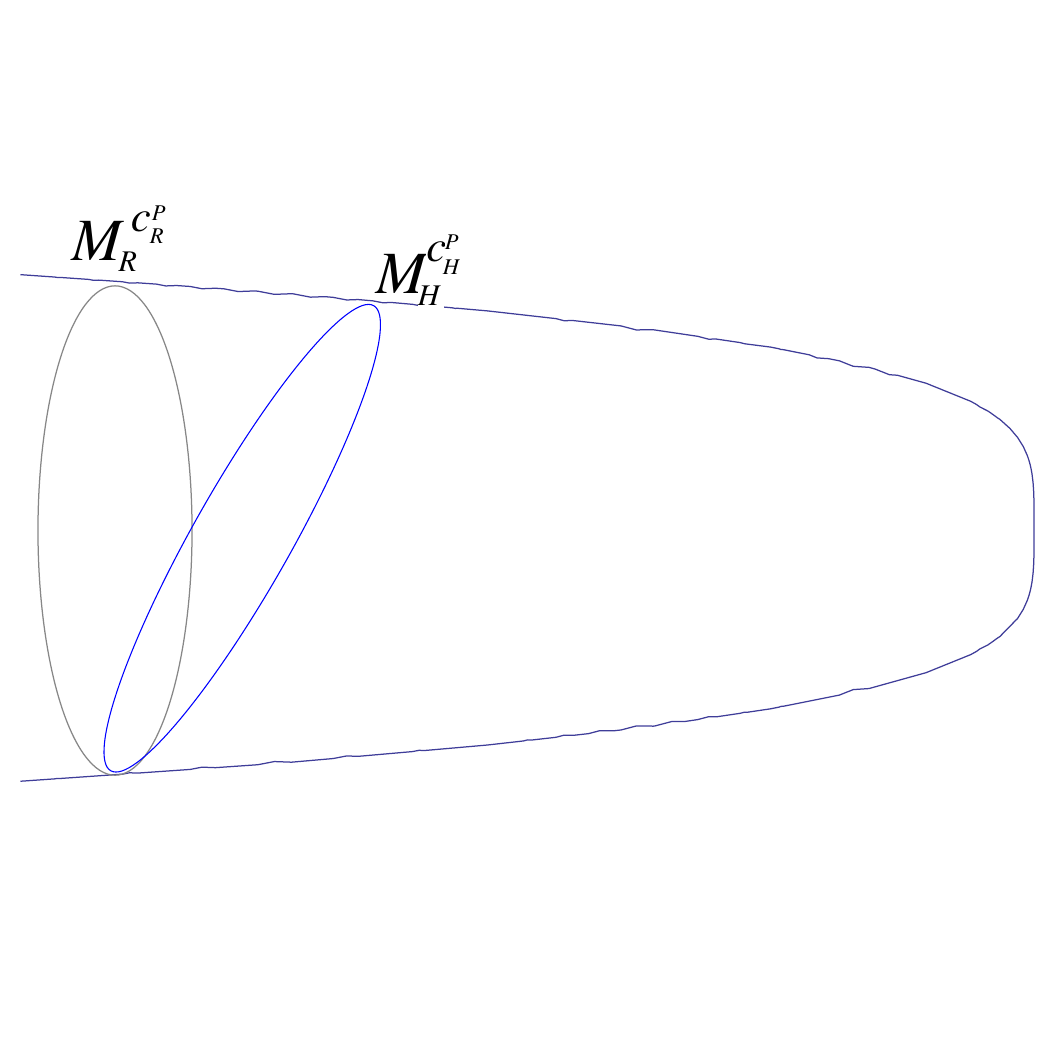}
\caption{The inclusion $M_H^{c_H^P} \subset M_R^{C_R^P}$ in $T^* S^2$}
\end{figure}

We also have embeddings of opposite direction. Namely, the Liouville domain determined by the rotating Kepler problem can be embedded in the Lioville domain determined by Hill's lunar problem.

\begin{prop}
We have the embedding
\begin{gather*}
M_R^{c+{1 \over 2c^2}} \subset M_H^c
\end{gather*}
for each $c>c_H^0$.
\end{prop}
\begin{proof}
It suffices to prove that $\Sigma_R^{c+{1 \over 2c^2}} \subset M_H^c$. Suppose that $(\bar{q}, \bar{p}) \in \Sigma_R^{c+{1 \over 2c^2}}$ and so
\begin{gather*}
{1 \over 2}|\bar{p}|^2-{1 \over |\bar{q}|}+\bar{p_1}\bar{q_2}-\bar{p_2}\bar{q_1}+c+{1 \over 2c^2}=0.
\end{gather*}
We evaluate $H_H^c:= H_H+c$ at this $(\bar{q}, \bar{p})$. Then we have
\begin{eqnarray*}
H_H^c(\bar{q}, \bar{p})&=&{1 \over 2}|\bar{p}|^2-{1 \over |\bar{q}|}+\bar{p_1}\bar{q_2}-\bar{p_2}\bar{q_1}-\bar{q_1}^2+{1 \over 2}\bar{q_2}^2+c
\\ &=& -{1 \over 2c^2}-\bar{q_1}^2+{1 \over 2}\bar{q_2}^2
\\ &\le& {1 \over 2}\bar{q_2}^2-{1 \over 2c^2}
\\ &\le& {1 \over 2c^2}-{1 \over 2c^2}=0
\end{eqnarray*}
The last inequality can be proved by the following Claim.
\\Claim: For any $(\bar{q}, \bar{p}) \in \Sigma_R^{c+{1 \over 2c^2}}$, we have $|\bar{q}| \le {1 \over c}$.

\begin{proof}[Proof of Claim]
Since $(\bar{q}, \bar{p}) \in \Sigma_R^{c+{1 \over 2c^2}}$, we have
\begin{gather*}
{1 \over 2}|\bar{p}|^2-{1 \over |\bar{q}|}+\bar{p_1}\bar{q_2}-\bar{p_2}\bar{q_1}+c+{1 \over 2c^2}=0
\\ \Rightarrow {1 \over 2}(\bar{p_1}+\bar{q_2})^2+{1 \over 2}(\bar{p_2}-\bar{q_1})^2={1 \over |\bar{q}|}+{1 \over 2}|\bar{q}|^2-c-{1 \over 2c^2}
\end{gather*}
This implies
\begin{gather*}
{1 \over |\bar{q}|}+{1 \over 2}|\bar{q}|^2 \ge c+{1 \over 2c^2}
\end{gather*}
and so we have
\begin{gather*}
|\bar{q}| \le {1 \over c}
\end{gather*}
This proves the Claim.
\end{proof}
Therefore, we have that $\Sigma_R^{c+{1 \over 2c^2}} \subset M_H^c$. This proves Proposition 5.7.
\end{proof}

This proves (3) of Theorem B and thus this completes the proof of Theorem B. We will use these inclusions to get estimates of action of Hill's lunar problem in the next section.

\section{Spectrum estimates of Hill's lunar problem}

We have prepared every ingredient to estimate the action spectrum of Hill's lunar problem. As a result, we will prove Theorem C and D in this section. In fact, we have almost finished the proof of Theorem C in Section 3 and in the computation of symplectic homology using the rotating Kepler problem. We denoted by $\Sigma_R^c$ ($\Sigma_H^{c'}$) the regularized energy hypersurfaces of the rotating Kepler problem(Hill's lunar problem) of energy $-c$($-c'$). We know that $\Sigma_R^c$ and $\Sigma_H^{c'}$ are fiberwise convex hypersurfaces in $T^* S^2$ for $c>c_R^0={3 \over 2}$ and $c'>c_H^0={3^{4 \over 3} \over 2}$. We defined the Liouville domains $M_R^c$ and $M_H^{c'}$ in $T^* S^2$ enclosed by $\Sigma_R^c$ and $\Sigma_H^{c'}$, respectively. We recall the computation of $SH_*(M_R^c)$ in Example 4.1.7 and Remark 4.1. Then we know that the retrograde(direct) orbit represents a homology class, say $\delta_R \in SH_*(M_R^c)$($\delta_D \in SH_*(M_R^c)$) for $c>c_R^1=2^{2 \over 3}$. Then there are corresponding homology classes $\Delta_R$ and $\Delta_D$ in $H_*(\Lambda S^2)$ such that $\Psi_{M_R^c}(\Delta_R)=\delta_R$ and $\Psi_{M_R^c}(\Delta_D)=\delta_D$ for $c>2^{2 \over 3}$. Then we have
\begin{gather*}
c_{S^2}(M_R^c, \Delta_R)=\mathcal{A}(\gamma_R^c)=2\pi L_R(c),
\\ c_{S^2}(M_R^c, \Delta_D)=\mathcal{A}(\gamma_D^c)=-2\pi L_D(c)
\end{gather*}
for $c>2^{2 \over 3}$. As we discussed before, if $c>c_R^P$, then multiple covers of the retrograde and direct orbits, up to $P$-th-iteration, become the generators of symplectic homology of $M_R^c$. Therefore, when $c \ge c_R^P$, we can determine the symplectic capacity corresponding to these multiple covers. In this case, we define similarly the homology classes $\delta_{R, N}$ and $\delta_{D, N}$ in $SH_*(M_R^c)$ represented by the $N$-th iteration of the retrograde and direct orbits, respectively. Also, we denote by $\Delta_{R, N}, \Delta_{D, N} \in H_*(\Gamma S^2)$ the loop homology classes satisfying $\Psi_{M_R^c}(\Delta_{R, N})=\delta_{R, N}$ and $\Psi_{M_R^c}(\Delta_{D, N})=\delta_{D, N}$, respectively, for $N= 1, 2, \cdots, P$. Then we have
\begin{gather*}
c_{S^2}(M_R^c, \Delta_{R,N})=N\mathcal{A}(\gamma_R^c)=2\pi N L_R(c),
\\ c_{S^2}(M_R^c, \Delta_{D,N})=N\mathcal{A}(\gamma_D^c)=-2\pi N L_D(c)
\end{gather*}
for $c \ge c_R^P$ and $N=1, 2, \cdots, P$. This proves Theorem C. We will prove Theorem D in the rest of this section.

\begin{Thm}
For the homology classes $\Delta_R, \Delta_D \in H_*(\Lambda S^2)$ defined above, the following inequalities
\begin{gather*}
c_{S^2}(M_H^c, \Delta_R) \ge 2 \pi{-1+\sqrt{1+8c^3} \over 4c^2},
\\ c_{S^2}(M_H^c, \Delta_D) \ge 2 \pi{1+\sqrt{1+8c^3} \over 4c^2},
\end{gather*}
hold for all $c>c_H^0={3^{4 \over 3} \over 2}$.
\end{Thm}
\begin{proof}
By Proposition 5.7, we have that
\begin{gather*}
M_R^{c+{1 \over 2c^2}} \subset M_H^c.
\end{gather*}
Then we can deduce by Theorem A the inequalities
\begin{gather*}
c_{S^2}(M_R^{c+{1 \over 2c^2}}, \Delta_R) \le c_{S^2}(M_H^c, \Delta_R),
\\ c_{S^2}(M_R^{c+{1 \over 2c^2}}, \Delta_D) \le c_{S^2}(M_H^c, \Delta_D)
\end{gather*}
of symplectic capacity for $FSD(S^2)$ for all $c>c_H^0$. Since $c+{1 \over 2c^2}>2^{2 \over 3}$ for $c>c_H^0={3^{4 \over 3} \over 2}$, the homology class $\Psi_{M_R^{c+{1 \over 2c^2}}}(\Delta_R)=\delta_R$ is represented by the retrograde orbit $\gamma_R$ and same for the direct orbit. This implies that
\begin{gather*}
c_{S^2}(M_R^{c+{1 \over 2c^2}}, \Delta_R)=2 \pi L_R(c+{1 \over 2c^2}),
\\ c_{S^2}(M_R^{c+{1 \over 2c^2}}, \Delta_D)=-2 \pi L_D(c+{1 \over 2c^2})
\end{gather*}
Because $-1<L_D(c+{1 \over 2c^2})<0<L_R(c)$ are zeros of $c+{1 \over 2c^2}={1 \over 2x^2}-x$, we have
\begin{gather*}
L_R(c+{1 \over 2c^2})={-1+\sqrt{1+8c^3} \over 4c^2}
\\ L_D(c+{1 \over 2c^2})={-1-\sqrt{1+8c^3} \over 4c^2}
\end{gather*}
for all $c>c_H^0$. This completes the proof of Theorem 6.1.
\end{proof}

This provides us a simple and sharp lower bound for symplectic capacity for $FSD(S^2)$ of $M_H^c$. Let us discuss about upper bounds as well.

\begin{Thm}
For the homology classes $\Delta_R, \Delta_D \in H_* (\Lambda S^2)$, the following inequalities
\begin{eqnarray*}
c_{S^2}(M_H^c, \Delta_R) &<& 2\pi \times {1 \over 2}\sqrt{{3 \over 2(c-3^{-{2 \over 3}})}} \sec \left({1 \over 3} \arccos \left( \left({3 \over 2(c-3^{-{2 \over 3}})} \right)^{3 \over 2}\right)\right)
\\ &<& 2^{-{11 \over 6}} \cdot 3^{1 \over 2} \pi \sec \left( {1 \over 3} \arccos(2^{-{5 \over 2}}\cdot 3^{3 \over 2})\right) \approx 2\pi \times 0.490534
\\ c_{S^2}(M_H^c, \Delta_D)&<& -2\pi \times {1 \over 2}\sqrt{{3 \over 2(c-3^{-{2 \over 3}})}} \sec \left({1 \over 3} \arccos \left( \left({3 \over 2(c-3^{-{2 \over 3}})} \right)^{3 \over 2}\right)+{2\pi \over 3}\right)
\\ &<& -2^{-{11 \over 6}} \cdot 3^{1 \over 2} \pi \sec \left( {1 \over 3} \arccos(2^{-{5 \over 2}}\cdot 3^{3 \over 2})+{2 \pi \over 3}\right) \approx 2\pi \times 0.793701
\end{eqnarray*}
hold for all $c>c_H^0$.
\end{Thm}
\begin{proof}
By Theorem 5.3 and Corollary 5.4, we know that $M_H^c \subset M_R^{c-3^{-{2 \over 3}}} \subset M_R^{2^{2 \over 3}+\epsilon}$ for sufficiently small $\epsilon>0$ and, using monotonicity of $c_{S^2}$, we have the inequalities
\begin{eqnarray*}
c_{S^2}(M_H^c, \Delta_R) &<& c_{S^2}(M_R^{c-3^{-{2 \over 3}}}, \Delta_R)=2 \pi L_R(c-3^{-{2 \over 3}}) 
\\ &\le& c_{S^2}(M_R^{2^{2 \over 3}+\epsilon}, \Delta_R)=2\pi L_R(2^{2 \over 3}+\epsilon)<2 \pi L_R(2^{2 \over 3}),
\\ c_{S^2}(M_H^c, \Delta_D) &<&  c_{S^2}(M_R^{c-3^{-{2 \over 3}}}, \Delta_D)=-2 \pi L_D(c-3^{-{2 \over 3}})
\\ &\le& c_{S^2}(M_R^{2^{2 \over 3}+\epsilon}, \Delta_D)=-2\pi L_D(2^{2 \over 3}+\epsilon)<-2 \pi L_D(2^{2 \over 3})
\end{eqnarray*}
for all $c >c_H^0$. Theorem follows by expressing $L_R$ and $L_D$ explicitly.
\end{proof}

From the above Theorem, there is an obvious Corollary. Let $l_1(\Sigma, \lambda)$ be the period of the shortest periodic Reeb orbit. $l_1(\Sigma, \lambda)$ is called the systole of the contact manifold $(\Sigma, \lambda)$.

\begin{Cor}
We have the following estimates
\begin{gather*}
l_1(\Sigma_H^c, \lambda_{can})<2\pi \times {1 \over 2}\sqrt{{3 \over 2(c-3^{-{2 \over 3}})}} \sec \left({1 \over 3} \arccos \left( \left({3 \over 2(c-3^{-{2 \over 3}})} \right)^{3 \over 2}\right)\right)
\end{gather*}
for the systole of the regularized Hill's lunar problem. 
\end{Cor}
For example, if we consider $c=c_H^0$(in fact, arbitrarily close $c$ to $c_H^0$), then we have spectral gap
\begin{gather*}
2\pi \times 0.43029 \approx 2\pi \times {-1 +\sqrt{82} \over 3^{8 \over 3}}<c_{S^2}(M_H^{{3^{4 \over 3} \over 2}}, \Delta_R)<2\pi \times 0.49053,
\\ 2\pi \times 0.53713 \approx 2\pi \times {1 +\sqrt{82} \over 3^{8 \over 3}}<c_{S^2}(M_H^{{3^{4 \over 3} \over 2}}, \Delta_D)<2\pi \times 0.79370
\end{gather*}
for the contact manifold $(\Sigma_H^{{3^{4 \over 3} \over 2}}, \lambda_{can})$. This means 
\begin{gather*}
Spec(\Sigma_H^{{3^{4 \over 3} \over 2}}, \lambda_{can}) \cap (2\pi \times 0.43029, 2\pi \times 0.49053) \ne \phi,
\\ Spec(\Sigma_H^{{3^{4 \over 3} \over 2}}, \lambda_{can}) \cap (2\pi \times 0.53713, 2\pi \times 0.79370) \ne \phi
\end{gather*}
Because the upper bound of these estimates is global, we can say
\begin{gather*}
l_1(\Sigma_H^c, \lambda_{can})<\pi
\end{gather*} 
for every $c >c_H^0$. As we discussed in Example 4.1.7 and Remark 4.1, If the condition
\begin{gather*}
0<-L_D(c)^3<{1 \over P+1}
\end{gather*}
holds for some $P \in \mathbb{N}$, then we can use the $N$-th iteration of the retrograde and direct orbits as generators of symplectic homology for $N=1, 2, \cdots, P$. For such $c$, we denote these generators by $\delta_{R, N}$ and $\delta_{D, N}$, respectively, for each $N=1, 2, \cdots, P$. Moreover, one can easily compute that
\begin{gather*}
L_R(c_R^P)={-(P+1)+\sqrt{(P+1)(P+9)} \over 4(P+1)^{1 \over 3}},
\\ L_D(c_R^P)=-(P+1)^{-{1 \over 3}}
\end{gather*}
using $c_R^P={1 \over 2}(P+1)^{2 \over 3}+(P+1)^{-{1 \over 3}}$ for all $P \in \mathbb{N}$. Therefore, if we combine these fact with Corollary 5.6 and Proposition 5.7, then we have the following Theorem

\begin{Thm}
Suppose that $c \in [c_H^P, c_H^{P+1})$ for $P \in  \{2, 3, \cdots\}$. Then we have estimates of symplectic capacity
\begin{gather*}
2 \pi N{-1+\sqrt{1+8c^3} \over 4c^2} \le c_{S^2}(M_H^c, \Delta_{R, N}) \le 2\pi N {-(P+1)+\sqrt{(P+1)(P+9)} \over 4(P+1)^{1 \over 3}},
\\ 2 \pi N{1+\sqrt{1+8c^3} \over 4c^2} \le c_{S^2}(M_H^c, \Delta_{D, N}) \le 2\pi N (P+1)^{-{1 \over 3}}
\end{gather*}
for all $N=1, 2, \cdots, P$.
\end{Thm}
This provides spectral gaps for the contact manifold $(\Sigma_H^c, \lambda_{can})$. If $c \in [c_H^P, c_H^{P+1})$, then we have
\begin{gather*}
Spec(\Sigma_H^c, \lambda_{can}) \cap ([2 \pi N{-1+\sqrt{1+8c^3} \over 4c^2}, 2\pi N {-(P+1)+\sqrt{(P+1)(P+9)} \over 4(P+1)^{1 \over 3}}]; 2N-1) \ne \phi
\end{gather*}
and
\begin{gather*}
Spec(\Sigma_H^c, \lambda_{can}) \cap ([2 \pi N{1+\sqrt{1+8c^3} \over 4c^2}, 2\pi N (P+1)^{-{1 \over 3}}]; 2N+1) \ne \phi
\end{gather*}
for $N=1, 2, \cdots, P$. We denote by $([a, b]; k)$ the action value between $a$ and $b$ with Conley-Zehnder index $k$. This implies only the existence of orbits with an action range and an index. Thus we do not know whether they are geometrically different or a orbit is multiple cover of another orbit and so on. Unfortunately, it is hard to get such geometric informations from homological informations. Conjecturally, the author guesses that the retrograde orbit of Hill's lunar problem has period $c_{S^2}(M_H^c, \Delta_R)$ with index 1 and the direct orbit of Hill's lunar problem has period $c_{S^2}(M_H^c, \Delta_D)$ with index 3. However, there is no evidence for this guess.

\section{Appendix}

We recall the definition of the systole and the systolic volume of a contact manifold $(\Sigma, \lambda)$ from \cite{AB}.

\begin{Def}
The systole of a contact manifold $(\Sigma, \lambda)$ is the smallest period of its periodic Reeb orbits. We denote the systole of $(\Sigma, \lambda)$ by $l_1(\Sigma, \lambda)$. We define the systolic volume of $(\Sigma, \lambda)$ by
\begin{gather*}
\mathfrak{S}(\Sigma, \lambda)={Vol(\Sigma, \lambda) \over l_1(\Sigma, \lambda)^n}
\end{gather*}
where $(\Sigma, \lambda)$ is a $(2n-1)$-dimensional contact manifold and $Vol(\Sigma, \lambda)$ is the contact volume $\int_{\Sigma} \lambda \wedge \lambda^{n-1}$.
\end{Def}

The goal of this Appendix is to find the systolic volume $\mathfrak{S}(\Sigma_R^c, \lambda_{can})$ for the energy hypersurface of the regularized rotating Kepler problem at energy $-c$. We already know that the systole periodic orbit is the retrograde orbit and its action.

\begin{Cor}
The systole of $(\Sigma_R^c, \lambda_{can})$ is
\begin{gather*}
l_1(\Sigma_R^c, \lambda_{can})=2 \pi L_R(c)
\end{gather*}
where $L_R(c)$ is the positive zero of the equation
\begin{gather*}
c={1 \over 2x^2}-x
\end{gather*}
for each $c>{3 \over 2}$. 
\end{Cor}

It is enough to obtain the contact volume of $(\Sigma_R^c, \lambda_{can})$. We use the result in \cite{CFvK}. They compute the Finsler function
\begin{gather*}
F_c^*(q, p)={1 \over 4}(|p|^2+2c)|q|\Big( 1+ \sqrt{1+{16<p^{\perp}, q> \over |q|(|p|^2+2c)^2}}\Big)
\end{gather*}
corresponding to $(\Sigma_R^c, \lambda_{can})$ in the stereographic projection chart of $T^* S^2$. Namely, we have
\begin{gather*}
(F_c^*)^{-1}(1)=\Phi^{-1}(\Sigma_R^c)
\end{gather*}
for each $c>{3 \over 2}$. We also have that
\begin{gather*}
Vol(\Sigma_R^c, \lambda_{can})=\int_{\Sigma_R^c} \lambda_{can} \wedge d \lambda_{can}=\int_{M_R^c} \omega_{can}^2
\end{gather*}
by definition and Stokes' Theorem. Furthermore, we can deduce
\begin{eqnarray*}
Vol(\Sigma_R^c, \lambda_{can})&=&\int_{\Phi \circ \Phi^{-1}(M_R^c)} \omega_{can}^2=\int_{\Phi^{-1}(M_R^c)} (\Phi^* \omega_{can})^2
\\ &=& \int_{F_c^*(q, p) \le 1} 2 dq dp
\end{eqnarray*}
where the last term is the usual Riemann integral. We note that
\begin{eqnarray*}
& & 1=F_c^*(q, p)
\\ &\iff& 1={1 \over 4}(|p|^2+2c)|q|\Big( 1+ \sqrt{1+{16<p^{\perp}, q> \over |q|(|p|^2+2c)^2}}\Big)
\\ &\iff& 1={1 \over 2}(|p|^2+2c)|q|+|q|<p^{\perp}, q>
\end{eqnarray*}
We use the polar coordinates $(q_1, q_2)=(r \cos \theta, r \sin \theta)$. Then we can express the condition
\begin{eqnarray*}
& & 1=F_c^*(q, p)
\\ &\iff& <p^{\perp}, u_{\theta}>r^2+{1 \over 2}(|p|^2+2c)r-1=0
\end{eqnarray*}
in terms of $r, \theta$ where $u_{\theta}=(\cos \theta, \sin \theta)$. For fixed $p, c$, we have the polar equation
\begin{gather*}
r_{p, c}(\theta)={4 \over {(|p|^2+2c)+\sqrt{(|p|^2+2c)^2+16<p^{\perp}, u_{\theta}>}}}
\end{gather*}
for the trajectory of $q$. Therefore,
\begin{eqnarray*}
& & \int_{F_c^*(q, p) \le 1} 2 dq dp
\\ &=& \int_{\mathbb{R}^2} \Big( \int_{0}^{2\pi} 2\cdot{1\over 2}r_{p,c}(\theta)^2 d\theta \Big) dp
\\ &=&\int_{\mathbb{R}^2} \Big( \int_{0}^{2\pi} {16 \over ({(|p|^2+2c)+\sqrt{(|p|^2+2c)^2+16<p^{\perp}, u_{\theta}>}})^2} d \theta \Big) dp
\\ &=&\int_{\mathbb{R}^2} \Big( \int_{0}^{2\pi} {16 \over ({(|p|^2+2c)+\sqrt{(|p|^2+2c)^2+16|p| \cos \theta}})^2} d \theta \Big) dp
\end{eqnarray*}
If we use the polar coordinates $(p_1, p_2)=(R\cos \alpha, R \sin \alpha)$, then we get the integral
\begin{eqnarray*}
& & \int_{\mathbb{R}^2} \Big( \int_{0}^{2\pi} {16 \over ({(|p|^2+2c)+\sqrt{(|p|^2+2c)^2+16|p| \cos \theta}})^2} d \theta \Big) dp
\\ &=& \int_{0}^{\infty} \int_{0}^{2 \pi} {32\pi R \over {((R^2+2c)+\sqrt{(R^2+2c)^2+16R \cos \theta})^2}} d\theta dR
\end{eqnarray*}
for the volume. In sum, we have the contact volume
\begin{gather*}
Vol(\Sigma_R^c, \lambda_{can})=\int_{0}^{\infty} \int_{0}^{2 \pi} {32\pi r \over {((r^2+2c)+\sqrt{(r^2+2c)^2+16r \cos \theta})^2}} d\theta dr
\end{gather*}
and therefore we have proved the following Theorem.

\begin{Thm}
The systolic volume of the energy hypersurface of the regularized rotating Kepler problem is given by
\begin{gather*}
\mathfrak{S}(\Sigma_R^c, \lambda_{can})={1 \over (2 \pi L_R(c))^2} {\int_{0}^{\infty} \int_{0}^{2 \pi} {32\pi r \over {((r^2+2c)+\sqrt{(r^2+2c)^2+16r \cos \theta})^2}} d\theta dr}
\end{gather*}
for each $c>{3 \over 2}$ where $L_R(c)$ is the positive zero of the equation $c={1 \over 2x^2}-x$.
\end{Thm}

\end{document}